\documentclass[11pt,reqno]{article}
\usepackage[margin=1in,top=2.2cm,bottom=2.2cm]{geometry}

\usepackage[english]{babel}
\usepackage[T1]{fontenc}
\usepackage{setspace}
\usepackage{enumitem}
\usepackage[normalem]{ulem}
\usepackage{float}
\usepackage{xcolor}
\usepackage{array}
\newcolumntype{M}[1]{>{\centering\arraybackslash}m{#1}}
\newcolumntype{N}{@{}m{0pt}@{}}

\usepackage{amsmath}
\usepackage{amssymb, mathabx}
\usepackage{amsthm}
\usepackage{mathtools}
\usepackage{euscript, mathrsfs}
\usepackage{bbm}
\usepackage{enumitem}
\usepackage[colorlinks=true, allcolors=blue]{hyperref}
\usepackage{titlesec}
\usepackage[title]{appendix}
\titlelabel{\thetitle.\quad}

\mathtoolsset{showonlyrefs}

\hypersetup{
    colorlinks,
    linkcolor={blue},
    citecolor={blue},
    urlcolor={blue}
}
\usepackage{yfonts}
\usepackage{indentfirst}
\usepackage[parfill]{parskip}
\usepackage[
    citestyle=numeric,
    bibstyle=authoryear,
    dashed=false,
    maxnames=4]{biblatex}
\makeatletter
\input{numeric.bbx}
\makeatother
\renewbibmacro{in:}{}
\DeclareNameAlias{author}{last-first}

\DeclareFieldFormat*{title}{#1}
\DeclareFieldFormat*[book]{title}{\textit{#1}}
\DeclareFieldFormat*{volume}{\textbf{#1}}

\newtheorem{Th}{Theorem}[section]

\newtheorem{Lem}[Th]{Lemma}
\newtheorem{Prop}[Th]{Proposition}

\theoremstyle{definition}
\newtheorem{remark}[Th]{Remark}

\numberwithin{equation}{section}
\numberwithin{figure}{section}

\DeclareMathOperator\erf{erf}

\newcommand{\fc}{\mathcal{F}}
\newcommand{\pr}{\mathbb{P}}
\newcommand{\ex}{\mathbb{E}}
\newcommand{\eps}{\varepsilon}
\newcommand{\olh}{\mathcal{H}}

\allowdisplaybreaks

\title{\textsc{Grab It Before It's Gone: \\ Testing Uncertain Rewards under a Stochastic Deadline}}
\author{
\textsc{Steven Campbell} 
\thanks{ 
\,\textsc{Columbia University, Department of Statistics, 1255 Amsterdam Ave, New York, NY 10027, USA} (e-mail: {\it sc5314@columbia.edu})}
\and
\textsc{Georgy Gaitsgori\footnote{Corresponding author.}} 
\thanks{ 
\,\textsc{Columbia University, Department of Mathematics, 2990 Broadway, New York, NY 10027, USA} (e-mail: {\it gg2793@columbia.edu})}
\and 
\textsc{Richard Groenewald} 
\thanks{ 
\,\textsc{Columbia University, Department of Statistics, 1255 Amsterdam Ave, New York, NY 10027, USA} (e-mail: {\it rag2202@columbia.edu})}
\and
\textsc{Ioannis Karatzas}
\thanks{ 
\,\textsc{Columbia University, Departments of Mathematics and Statistics, 2990 Broadway, New York, NY 10027, USA} (e-mail: {\it ik1@columbia.edu})}
}
\date{\today}

\begin{document}
\maketitle

\begin{abstract}
We study a sequential estimation problem for an unknown reward in the presence of a random horizon. The reward takes one of two predetermined values that can be inferred from the drift of a Wiener process, which serves as a signal. The objective is to use the information in the signal to estimate the reward which is made available until a stochastic deadline that \textit{depends} on its value. The observer must therefore work quickly to determine if the reward is favorable and claim it before the deadline passes. Under general assumptions on the stochastic deadline, we provide a full characterization of the solution that includes an identification with the unique solution to a free-boundary problem. Our analysis derives regularity properties of the solution that imply its ``smooth fit'' with the boundary data, and shows that the free-boundary solves a particular integral equation. The continuity of the free-boundary is also established under additional structural assumptions that lead to its representation in terms of a continuous transformation of a monotone function. We provide illustrations for several examples of interest.
\end{abstract}


\vspace{4cm}

\section{Introduction}\label{sec_Introduction}

This paper studies the following problem. Suppose there is a reward, whose value is modeled by a random variable $R$ that can only take two known values, one positive and the other negative. However, $R$ cannot be observed directly, but only through a noisy stream of observations
\begin{equation*}
    X(t) = R \, t + W(t), \quad 0 \le t < \infty,
\end{equation*} 
where $W(\cdot)$ is a standard Brownian motion, independent of $R$. At any given moment of time, one can ``grab'' the reward and receive a payoff equal to the value of the random variable $R$; but with the proviso, or ``conditioning,'' that the decision must be made before a stochastic deadline $\gamma$. This deadline is assumed to have a known distribution, and depends on $R$ but is independent of $W(\cdot)$. In other words, one needs to choose a stopping time $\tau$ of the observation process $X(\cdot)$, and thereby obtain the expected payoff 
\begin{equation}\label{expected_payoff_heuristic}
    \ex[R \cdot \mathbf{1}_{\{\tau < \gamma\}}].
\end{equation}
Here, $\gamma$ is the random time at which the reward disappears and is no longer available to the observer. What is the optimal strategy $\tau$ to deploy, in order to maximize the expected payoff \eqref{expected_payoff_heuristic}? 

The above problem is a variation of a classical estimation problem for the drift of a Wiener process in the presence of a random horizon (equivalently, stochastic deadline). 
Aside from the random horizon, the primary distinction between classical formulations and ours lies in the reward structure. In traditional sequential testing problems, the objective is typically to minimize the Bayes risk associated with drift estimation, often incorporating a linear temporal penalty (see, for instance, the seminal paper by Shiryaev \cite{Shiryaev67}, or the literature review below). In contrast, the reward in our problem is defined as the drift value of the observation process itself. Thus, drift estimation serves as a means to evaluate whether the reward is worth stopping for, rather than being the primary goal. Despite these significant differences, the spirit of the problem remains aligned with that of sequential testing problems.

These problems form a research area within the broader field of sequential analysis -- a subfield of probability and statistics that addresses problems where the sample size is not fixed in advance but can vary, based on the observer's objectives.
The origins of sequential analysis are often attributed to Abraham Wald's foundational work \cite{Wald45}, which introduced sequential hypothesis testing and the celebrated Sequential Probability Ratio Test (SPRT) (see also Wald's monograph \cite{Wald47}). 
Since then, sequential analysis has flourished, yielding a myriad of theoretical and practical results. Therefore, we do not make it our goal to provide a comprehensive literature review of the field. Instead, we aim to trace the path from Wald’s original work to the specific problem studied in this paper and, additionally, to highlight recent advances in the theory of sequential testing problems, which has experienced a notable resurgence over the past 10–15 years.

In his seminal paper \cite{Wald45}, Wald introduced the problem of sequential hypothesis testing for i.i.d. observations, i.e., discrete-time observation processes. A few years later, Dvoretzky, Kiefer, and Wolfowitz \cite{DvoKieWol53} extended this framework to sequential decision-making for stochastic processes in continuous time. The sequential testing for the drift of a Wiener process soon followed, with contributions from several authors. Bather \cite{Bather62} solved the sequential testing problem for the sign of the drift under a Bayes risk payoff structure, focusing on the special case of a normal prior distribution. Chernoff \cite{Chernoff61} addressed a related problem with similar assumptions but with a payoff that depends on the magnitude of the unobservable drift (see also Zhitlukhin and Muravlev \cite{ZhiMur13} for an explicit solution to this problem). Shiryaev \cite{Shiryaev67} in turn advanced the field by solving the sequential testing problem for a two-point prior distribution using a connection with a time-homogeneous free-boundary problem (see also earlier treatment in Mikhalevich \cite{Mikhalevich56}, and Shiryaev's monographs \cite{Shiryaev73} and \cite{Shiryaev78} for more details), solidifying his work as a cornerstone in Bayesian sequential testing.

Following these foundational results, progress in this field slowed for some time. However, starting around 2000, new generalizations of the classical sequential testing problem began to emerge. One major direction involved extending the framework to more general observation processes. Peskir and Shiryaev \cite{PesShi00} derived an explicit solution to the sequential testing problem for the constant intensity rate of an observable Poisson process. Gapeev and Shiryaev \cite{GapShi11} investigated sequential testing of two simple hypotheses about the drift rate of a general observable diffusion process, with Ernst and Peskir \cite{ErnPes24} later addressing and confirming the Gapeev-Shiryaev conjecture on the monotonicity of optimal stopping boundaries. In addition, Johnson and Peskir \cite{JohPes18} extended sequential testing to Bessel processes.

Another major direction explored variations in payoff structures. Dyrssen and Ekström \cite{DyrEks18} examined sequential testing for the drift of Brownian motion with costly observations. Campbell and Zhang \cite{CampZha24} introduced and solved a soft classification version of the classical problem, while Ekström, Karatzas, and Vaicenavicius \cite{EksKarVai22} tackled the case of an $L^2$ penalty with general distribution for the drift. Notably, Ekström and Vaicenavicius \cite{EksVai15} had earlier also addressed the case of a general prior distribution for the unknown drift in the classical sequential testing problem, and managed to characterize the solution in terms of a particular integral equation (see also a preceding work of Zhitlukhin and Shiryaev \cite{ZhiShi11}, who studied a problem where the drift could take one of three values). Several additional contributions that we would like to mention include the following. Ernst, Peskir, and Zhou \cite{ErnPesZho20} analyzed the Wiener sequential testing problem in two and three dimensions. Johnson, Pedersen, Peskir, and Zucca \cite{Johnson22} considered a detection problem involving the presence of random drift in Brownian motion. Most recently, De Angelis, Garg, and Zhou \cite{DeAngelisGargZhou} have studied an optimal multiple stopping problem pertaining to the quickest detection of the presence of a drift of a Brownian motion, Ernst and Mei \cite{ErnstMei23} explored a minimax version of the Wiener sequential testing problem, and Campbell and Zhang \cite{CampZha24MF} studied a so-called mean-field game of sequential testing.

Despite variations in setups and structures, most of the works mentioned above share a key assumption: they consider problems posed on an infinite-time horizon. While this framework is mathematically convenient, it is often restrictive and less applicable to real-world scenarios. To address this, several authors have explored versions of sequential testing problems under finite or random horizons. Gapeev and Peskir \cite{GapPes04} tackled a finite time horizon version of the Wiener sequential testing problem \cite{Shiryaev67}. In the context of random horizons, much of the existing work has been conducted in discrete-time settings. For example, Fraizer and Yu \cite{FraYu07} examined discrete-time sequential testing of binary hypotheses under a stochastic deadline. Dayanik and Yu \cite{DayanikYu13} focused on maximizing sequential reward rates. Novikov and Palacios-Soto \cite{NovPal20} studied a variant of the modified Kiefer-Weiss problem, aiming to minimize the average sample size while imposing constraints on error probabilities. Zhang, Moustakides, and Poor \cite{ZhaMouPoor16} investigated a discrete-time problem similar to ours, where stopping is permitted only when a particular hypothesis should be accepted.

The more immediate motivation for this work, however, is a recent paper by Ekström and Wang \cite{EksWang24}, who address general optimal stopping problems involving uncertainty and random horizons and present several relevant examples. Of particular interest is their analysis of a classical sequential testing problem for a Wiener process with zero-drift, on a random horizon with exponential distribution (see Section 7 of \cite{EksWang24}). 
Additionally, they examine a so-called ``hiring problem’’ under a random horizon (Section 5), which is essentially equivalent to the problem \eqref{expected_payoff_heuristic} under consideration. The main limitation of their work on the hiring problem is the assumption that the random horizon follows an exponential distribution: they do not consider finite-horizons  or more general random horizons. The main contribution of our paper is the provision of a full solution in the much broader class of stochastic deadline distributions introduced in Section \ref{sec_Model} for both finite and infinite time horizons.

We now turn to an outline of our analysis of the problem \eqref{expected_payoff_heuristic}. We start by reformulating the question as an optimal stopping problem for a conditional probability process with general discount functions, then embed this problem into a broader Markovian framework. Next, we examine the value function of this general problem and demonstrate that, under suitable regularity conditions on the discount functions (or equivalently, on the conditional cumulative distribution functions of the stochastic deadline), the value function and optimal stopping boundary form the \textit{unique} solution to a specific free-boundary problem. 
In addition, we show that the optimal stopping time is the first hitting time of the associated boundary.
Subsequently, we focus on the stopping boundary and investigate its properties: we identify sufficient conditions for its continuity and monotonicity under an appropriate transformation, and show that the stopping boundary solves a specific integral equation. Finally, we conclude with illustrative examples and numerical results.

The outlined program and results will be familiar to the reader acquainted with optimal stopping problems. However, the primary challenge in our case stems from the general structure of the discount functions, which necessitates a very delicate analysis. In particular, unlike typical optimal stopping scenarios, our general assumptions do not ensure that the stopping boundary is monotone or even continuous. Consequently, additional effort is required to verify rigorously several key statements and results. To address these challenges, we leverage heavily the connection between our problem and American options and rely on partial differential equations theory. In the process, we develop novel probabilistic arguments that extend the regularity properties of solutions of finite-horizon free-boundary problems to their infinite-horizon counterparts. Moreover, it is worth emphasizing that the arguments we employ allow us to establish the \textit{uniqueness} of the solution to these problems. We believe that these contributions are of independent interest.

The structure of the paper is as follows. 
We formalize our setup in Section \ref{sec_Model}. Section \ref{sec_equivalent_formulations} is devoted to several equivalent formulations, which will be convenient for the proofs of our main results. In Section \ref{sec_main_results} we formulate our main results and investigate the value function of the problem. Section \ref{sec_boundary} studies the optimal boundary, discusses several examples, and presents a numerical analysis. Section \ref{sec_conclusion} concludes by discussing potential extensions. Due to their length and technical complexity, most proofs have been placed in the appendices, with only the proofs of the main results included in the main body of the paper.

\section{Model}\label{sec_Model}
We fix a probability space $(\Omega, \fc, \pr)$, supporting a Bernoulli random variable $\theta$ with $\pr(\theta=1) = p = 1 - \pr(\theta=0)$ for some $ 0 < p < 1$, and an independent standard Brownian motion $W = \{W(t), \, 0 \le t < \infty\}$. We consider an arithmetic Brownian motion $X = \{X(t), \, 0 \le t < \infty\}$, with dynamics given by
\begin{equation}\label{observation_process_model_sec}
    X(t) = (a\theta + b)\,t + W(t), \quad 0 \le t < \infty.
\end{equation}
Here, $a$ and $b$ are real constants with $a + b \ge 0 > b$ (such a choice of the parameters is made without loss of generality and is explained in Remark \ref{remark_parameters_explanation}).
We let $\mathbb{F} \coloneqq \{\mathcal{F}(t)\}_{t \ge 0}$ be the augmented filtration generated by the process $X(\cdot)$, and $\mathbb{F}^{\theta, W} \coloneqq \{\mathcal{F}^{\theta, W} (t)\}_{t \ge 0}$ be the augmented ``initial enlargement'' of the filtration generated by the Brownian Motion $W(\cdot)$, via the random variable $\theta$.
In other words, for all $t \ge 0$, we set $\mathcal{F}(t) \coloneqq \overline{\sigma}(X(s), 0 \leq s \le t)$ and $\mathcal{F}^{\theta, W}(t) \coloneqq \overline{\sigma}(\theta, W(s), 0 \leq s \le t)$, where $\overline{\sigma}$ denotes the usual augmentation by the null sets of the underlying $\sigma$-algebra; clearly $\mathcal{F}(t) \subseteq \mathcal{F}^{\theta, W}(t)$. We also denote the collection of admissible stopping times by
\begin{align}\label{collection_stop_times}
    \mathcal{T}_T \coloneqq \Big\{\tau : \tau \text{ is an } \mathbb{F} \text{--stopping time, } \pr(\tau \le T) = 1\Big\},
\end{align}
where $T \in [0, \infty]$ is the time horizon of the problem. 

With the above notation, we consider the following problem. For a fixed time horizon $T$, our goal is to find a pair $(\tau, d)$, consisting of a stopping time $\tau \in \mathcal{T}_T$ and of an $\mathcal{F}(\tau)$--measurable random variable $d:\Omega \to \{0, 1\}$, which maximizes the expected reward
\begin{equation}\label{expected_reward_function_with_decision}
    J_T(\tau, d) \coloneqq \ex\left[
        \big(
            a\theta + b
        \big)
        \cdot 
        \mathbf{1}_{\{d = 1\}} 
        \cdot 
        \mathbf{1}_{\{\tau < \gamma\}} 
    \right].
\end{equation}
Here, $\gamma: \Omega \to [0, T]$ is a stochastic deadline, marking the disappearance of the reward; and $d:\Omega \to \{0, 1\}$ has the significance of a ``decision rule,'' that we make at time $\tau$ concerning the value of the unobserved $\theta$. We assume that the vector $(\theta, \gamma)$ is independent of $W(\cdot)$. 

\begin{remark}\label{remark_gamma_filtration}
    We show in Appendix \ref{subsec_equality_of_gamma_problems} that the analysis of the above problem does not change if, at each time $t \in [0, \infty)$, we observe whether or not the stochastic deadline $\gamma$ has passed, i.e., we enlarge the observation $\sigma$-algebra $\mathcal{F}(t)$ to include the indicator process $\left(\mathbf{1}_{\{\gamma > t\}}\right)_{t\geq0}$.
\end{remark}

Before proceeding, we take a moment to provide some intuition. 
As in the ``hiring problem'' of Ekström and Wang \cite{EksWang24}, which motivated our paper, the process \eqref{observation_process_model_sec} can be interpreted as noisy observations of a job applicant’s interview performance, with true ability level given by $(a\theta + b)$. At any moment in time, the company may choose to stop the interview process and decide whether to hire the candidate. If the candidate is hired, the company’s reward corresponds to the applicant’s actual ability level. However, to account for competition in the job market, the company faces a stochastic deadline -- the point at which the candidate might accept an offer from another employer and become unavailable. Naturally, this deadline may depend on the candidate's skill level $\theta$, but is independent of the noise in the observation $W(\cdot)$.

\begin{remark}\label{rem_different_formulation}
The formulation in \eqref{expected_payoff_heuristic} is modified here by introducing a
terminal decision rule $d\in\{0,1\}$ and the factor $\mathbf 1_{\{d=1\}}$ inside the
expectation. The purpose is to allow the observer to \emph{reject} the terminal reward
$(a\theta+b)$ in finite time when it is believed to be unfavorable. In particular, it accommodates the finite-horizon case in which the terminal time $T$ is reached and the reward is expected to be negative. Note that at any $t<T$, early stopping with rejection is never beneficial. 

\vspace{-\parskip}
Indeed,
let $(\tau,d)$ be an admissible pair, set
$A \coloneqq \{\tau<T,\ d=0\}\in\mathcal F(\tau)$, and define a modified strategy by
\[
\bar\tau := T\,\mathbf 1_A + \tau\,\mathbf 1_{A^c},
\qquad
\bar d := d .
\]
On the event $A$ both strategies yield the same payoff, and on $A^c$ they coincide. Thus, one may restrict attention to
strategies that only stop early when they accept. Edge cases in which it is still optimal to reject early (for instance, when the
reward is unfavorable with certainty) may arise mathematically, but are degenerate and
fall outside the main spirit of the problem.
\end{remark}

\begin{remark}\label{remark_parameters_explanation}
    The objective \eqref{expected_reward_function_with_decision} justifies the assumption $a + b \ge 0 > b$ on the parameters $a$ and $b$, which we made earlier. Indeed, if we had $(a + b)\, b \ge 0$, it would obviously be optimal to stop immediately and choose $d = 1$ if $a + b \ge 0, \, b \ge 0$; whereas the problem is trivial in the case $a + b \leq 0, \, b \le 0$, since the reward is non-positive regardless of the choice of stopping time. Finally, the configuration $a + b \le 0 \le b$ is symmetric to the one we are considering by taking $\bar a \coloneqq -a, \bar b \coloneqq a + b$, $\bar \theta \coloneqq 1-\theta$.
\end{remark}

The expression under the expectation in \eqref{expected_reward_function_with_decision} can be rewritten in a form more amenable to analysis. First, since the decision rule $d$ is $\mathcal{F}(\tau)$--measurable, the tower property of conditional expectations expresses the expected reward $J_T(\tau, d)$ as
\begin{equation*}
    J_T(\tau, d) = \ex\Big[
        \ex\Big[
            \big(
            a\theta + b
            \big)
            \cdot 
            \mathbf{1}_{\{\tau < \gamma\}}  \, \Big\vert \, \mathcal{F}(\tau)
        \Big]
        \cdot
        \mathbf{1}_{\{d = 1\}}
    \Big].
\end{equation*}
Thus, given stopping time $\tau$, the optimal decision rule is given by
\begin{equation}\label{optimal_decision_rule}
    d = 
    \begin{cases}
        1, & \ex\Big[
                \big(
                a\theta + b
                \big)
                \cdot 
                \mathbf{1}_{\{\tau < \gamma\}}  \, \Big| \, \mathcal{F}(\tau)
            \Big] \ge 0, \\
        0, & \text{otherwise}.
    \end{cases}
\end{equation}

Secondly, it will be convenient to rewrite the conditional expectation in \eqref{optimal_decision_rule} in terms of a random discounting rather than an indicator. For that reason, we denote the conditional survival function of $\gamma$ by
\begin{equation}\label{def_survival_function}
    c(\theta, t) \coloneqq \ex \big[\mathbf{1}_{\{\gamma > t\}} \mid \theta \big] = \pr \big(\gamma > t \mid \theta \big), \quad 0 \le t < \infty,
\end{equation}
and use the tower property again, as well as the independence of $\theta$ and $W(\cdot)$, to rewrite 
\begin{equation}\label{rewritten_objective}
    \begin{split}
        \ex\Big[
            \big(
            a\theta + b
            \big)
            \cdot 
            \mathbf{1}_{\{\tau < \gamma\}}  \, \Big| \, \mathcal{F}(\tau)
        \Big]
        &= 
        \ex\Big[\ex\left[
            \big(
            a\theta + b
            \big)
            \cdot 
            \mathbf{1}_{\{\tau < \gamma\}}  \, \Big| \, \mathcal{F}^{\theta, W}(\infty)
        \right] \, \Big| \, \mathcal{F}(\tau) \Big]
        \\&=
        \ex\left[ \big(
            a\theta + b
            \big)
            \cdot 
            \ex\left[
            \mathbf{1}_{\{\tau < \gamma\}}  \, \Big| \, \mathcal{F}^{\theta, W}(\infty)
        \right] \, \Big| \, \mathcal{F}(\tau) \right]
        \\&=
        \ex\left[ \big(
            a\theta + b
            \big)
            \cdot 
            \ex\left[
            \mathbf{1}_{\{t \le \gamma\}}  \, \Big| \, \theta
        \right] \Bigg|_{t = \tau} \, \Big| \, \mathcal{F}(\tau) \right] 
        \\&=
        \ex\left[ \big(
            a\theta + b
            \big)
            \cdot 
            c(\theta, \tau)
            \, \Big| \, \mathcal{F}(\tau) \right] \eqqcolon L(\tau).
    \end{split}
\end{equation}
The third equality follows from the fact that $(\theta,\gamma)$ is independent of the Brownian motion $W(\cdot)$, using a fairly straightforward generalization of property (k) on page 88 of \cite{DWilliams1991} (see Appendix \ref{app_role_of_independence} for the proof).

Combining the above arguments, the problem of maximizing \eqref{expected_reward_function_with_decision} reduces to the optimal stopping problem of maximizing $\ex \left[ (L(\tau))^+ \right]$ as in \eqref{rewritten_objective}
over the collection of stopping times $\tau \in \mathcal{T}_T$, where $c: \{0, 1\} \times [0, T] \to [0, 1]$ is the function defined in \eqref{def_survival_function}, and $(x)^+ = \max (x, 0)$.

\subsection{Generalized model}

From now onward, we will consider a slightly more general formulation. Namely, we shall consider the optimal stopping problem of maximizing the expected payoff 
\begin{equation}\label{expected_reward_function}
    J_T(\tau) \coloneqq \ex\left[\Big(
        \ex\Big[\theta \cdot c_1(\tau) - (1-\theta) c_0(\tau) \, \Big| \, \mathcal{F}(\tau)
        \Big]\Big)^+
    \right],
\end{equation}
always in the setting of \eqref{observation_process_model_sec} and with the assumptions and notation developed there, and with some given appropriate ``discount'' functions $c_i(\cdot): [0, T] \to [0, \infty), \, i = 0, 1$. The original setting is a special case of the problem of maximizing \eqref{expected_reward_function}, with
\begin{equation}\label{eqn:orig.prob.embedding}
        c_0(t) = -b \cdot c(0,t) \quad \text{ and } \quad
        c_1(t) = (a+b) \cdot c(1,t).
\end{equation}

\begin{remark}
    This framework is general enough to allow for rewards in \eqref{expected_reward_function_with_decision} that depend on, but are not necessarily equal to the drift of the signal process in \eqref{observation_process_model_sec}. Additionally, the reward can be further multiplied by any deterministic discount factor. Both of these effects are ultimately absorbed into the functions $c_i(\cdot)$, $i=0, 1$.
\end{remark}

We will allow in \eqref{expected_reward_function} any discount functions $c_i(\cdot), \, i = 0, 1,$ that satisfy
\begin{enumerate}[leftmargin=35pt, label = \textbf{(A\arabic*)}]
    \item $c_i(t) > 0$ for every $t \in [0,T)$, $i = 0,1,$ \label{reg_assm_1}
	\item $c_i(\cdot), \, i = 0, 1$, are continuous on $[0, T]$ and of class $C^2([0, T))$,\label{reg_assm_2}
    \item $c_i(\cdot), \, i = 0, 1,$ are non-increasing, \label{reg_assm_3}
    \item $c_i'(\cdot)$, $i = 0, 1$, are bounded on $[0,T)$, and \label{reg_assm_4}
    \item For all $t \in [0, T)$ we have
    \begin{equation*}
        \frac{c_0'(t)}{c_0(t)} - \frac{c_1'(t)}{c_1(t)} \ge 0 \quad \text{ and } \quad \frac{c_1'(t)}{c_1(t)} < 0.
    \end{equation*}
    \label{reg_assm_5}
\end{enumerate}

Here, \ref{reg_assm_1}, \ref{reg_assm_2}, and \ref{reg_assm_4} are technical regularity assumptions. 
The assumption \ref{reg_assm_3} casts the functions $c_i(\cdot), \, i = 0, 1$, as discounts, so the non-increasing property aligns well with the nature of the problem in \eqref{def_survival_function}. 
Finally, the assumption \ref{reg_assm_5} provides a relationship between the growth rates of the functions $c_0(\cdot)$ and $c_1(\cdot)$, which requires the function $c_1(\cdot)$ to decay relatively faster than $c_0(\cdot)$. We believe this relationship is natural in the context of our problem, especially recalling the ``hiring'' formulation, since it can be interpreted as saying that a strong candidate becomes unavailable sooner than a weak one. The second inequality of \ref{reg_assm_5} ensures that $c_1(\cdot)$ is strictly decaying, and is mostly made for technical reasons. We note that, although the problem formulation implicitly assumes that the horizon $\gamma$ depends on the reward $\theta$, our assumptions also accommodate the case where $\gamma$ is independent of $\theta$; for instance, they include the common exponential discounting $c_i(t)=K_i e^{-rt}$.

\begin{remark} \label{relaxation_of_assumptions_remark}The assumption \ref{reg_assm_1} may be changed to $c_i(t) \ge 0$ for every $t \in [0,T]$, $i = 0,1,$ without loss of generality. This is because the problem of maximizing \eqref{expected_reward_function} over any horizon $T > T^* \coloneqq \min(T_1, T_2)$ with $T_i \coloneqq \inf\{t \ge 0 : c_i(t) = 0\}$, $i = 0,1$, is equivalent to the problem of maximizing \eqref{expected_reward_function} over the horizon $T^*$. This follows from the expression under the expectation in \eqref{expected_reward_function} and an observation that $c_1(\cdot)$ and $1-\theta$ are both non-negative. Thus, there is no reason to continue observations after any of $T_1$ or $T_2$.
Separately, by leaving \ref{reg_assm_1} \textit{unchanged}, \ref{reg_assm_3} can be relaxed to $c_0(\cdot) + c_1(\cdot)$ and \textit{either} of $c_i(\cdot), \, i = 0, 1$, must be non-increasing. This extension follows readily from our proofs, but we impose our stronger assumption above for the sake of readability.
\end{remark}

\begin{remark}
   When interpreted in terms of the (conditional) survival distributions of $\gamma$ as in \eqref{def_survival_function}, it is natural for $c_0(\cdot)$ and $c_1(\cdot)$ to be lower semicontinuous. At the same time, in order for the optimal stopping problem to be well-posed, it is critical to further require that these functions are continuous. Were this not the case, \eqref{expected_reward_function} would fail to be upper semicontinuous, and an optimal stopping time may no longer exist in general. Moreover, with $T = \infty$, the assumption \ref{reg_assm_2} means that the functions $c_i(\cdot), \, i = 0, 1$, are continuous on $[0, \infty]$ with respect to the topology associated with the one-point compactification of the real line.
    In particular, we require $c_i(\infty) = \lim_{t \rightarrow \infty} c_i(t), \, i = 0, 1$.
\end{remark}

We are now ready to study the problem \eqref{expected_reward_function} under the above assumptions, and denote the value function of the problem by
\begin{equation}\label{value_function}
    V_T \coloneqq \sup\limits_{\tau \in \mathcal{T}_T} J_T(\tau).
\end{equation}

\section{Equivalent Formulations}\label{sec_equivalent_formulations}

The optimal stopping problem \eqref{expected_reward_function}--\eqref{value_function} is formulated in terms of the process $X(\cdot)$ of \eqref{observation_process_model_sec}, and the random variable $\theta$. 
As is common in sequential testing (see, e.g., \cite{Shiryaev67, GapPes04, EksVai15}), we reformulate the problem in terms of the conditional (a posteriori) probability process $\Pi = \{\Pi(t), 0 \le t < \infty\}$, given by
\begin{equation}\label{conditional_probability_process_def}
    \Pi(t) = \pr\big(\theta = 1 \mid \mathcal{F}(t)\big), \quad 0 \le t < \infty.
\end{equation}
In addition, several of our proofs rely on further equivalent reformulations. This section presents these reformulations and demonstrates their equivalence to the original problem. 

In Subsection \ref{subsec_aposteriori_process}, we cast the problem \eqref{value_function} in terms of the a posteriori probability process $\Pi(\cdot)$, embed it into a general Markovian framework, and show how this setup can be reverted to the original problem. Subsection \ref{subsec_equivalent_osp} introduces a variation of this Markovian formulation with a slightly modified payoff function and establishes its equivalence to the prior formulations. Finally, in Subsection \ref{subsec_american_options_formulation}, we present yet another equivalent formulation, this time in the context of American options. This latter formulation proves particularly useful for verifying the smooth-fit property of the value function and analyzing some technical properties of the optimal stopping boundary. Subsections \ref{subsec_equivalent_osp} and \ref{subsec_american_options_formulation} can be skipped until their results are used in later sections.

\subsection{A Posteriori Probability Process and General Markovian Setup}\label{subsec_aposteriori_process}

We start by establishing the essential properties of the process $\Pi(\cdot)$, defined in \eqref{conditional_probability_process_def}.

\begin{Lem}\label{lem_properties_of_cond_prob_process}
    In the setting of \eqref{observation_process_model_sec} and Section \ref{sec_Model}, the conditional probability process $\Pi(\cdot)$, defined in \eqref{conditional_probability_process_def}, satisfies
    \begin{equation}\label{cond_prob_dynamics}
        \Pi(t) = p + \int_0^t a\Pi(s)(1-\Pi(s)) \, dB(s), \quad 0 \le t < \infty,
    \end{equation}
    where the ``innovation process''
    \begin{equation*}
        B(t) = X(t) - \int_0^t (a\Pi(s) + b) \, ds, \quad 0 \le t < \infty
    \end{equation*}
    is an $\mathbb{F}$--Brownian motion.
    The processes $B(\cdot)$, $X(\cdot)$, and $\Pi(\cdot)$ generate the same filtration $\mathbb{F}$, and $\Pi(\cdot)$ is a strong Markov process with respect to this filtration. Furthermore, the following relation between the processes $\Pi(\cdot)$ and $X(\cdot)$ holds almost surely:
    \begin{equation} \label{X_and_pi_relationship}
        \Pi(t) = \frac{p \cdot \exp\left((a + b) X(t) - \frac{(a + b)^2}{2}t \right)}{p \cdot \exp \left((a + b)X(t)- \frac{(a + b)^2}{2}t\right) + (1-p) \cdot \exp \left(bX(t) -  \frac{b^2}{2}t\right)}, \quad 0 \le t < \infty.
    \end{equation}
\end{Lem}

\begin{proof}
All statements of Lemma \ref{lem_properties_of_cond_prob_process} are well-known results in filtering theory. 
For detailed proofs, we refer the reader to the monograph Shiryaev \cite{Shiryaev78} (see pp. 180-182).
\end{proof}

We now recast the original optimal stopping problem \eqref{value_function} for the process $X(\cdot)$ as an equivalent problem for the process $\Pi(\cdot)$. 
We start by observing that \eqref{expected_reward_function} can be written as
\begin{equation}\label{expected_payoff_via_Pi_process}
    J_T(\tau) = \ex\left[\Big(c_1(\tau) \, \Pi(\tau) - c_0(\tau) \,(1 - \Pi(\tau))\Big)^+\right].
\end{equation}
Indeed, for any $\tau \in \mathcal{T}_T$, we have
\begin{equation*}
    \begin{split}
        J_T(\tau) &= 
        \ex\left[\Big(
            \ex\Big[\theta \cdot c_1(\tau) - (1 - \theta) \cdot c_0(\tau) \, \Big| \, \mathcal{F}(\tau)
            \Big]
        \Big)^+\right]
        \\&= 
        \ex\left[\Big(
            \ex\Big[
                c_1(\tau) \cdot \mathbf{1}_{\{\theta=1\}} - c_0(\tau) \cdot \mathbf{1}_{\{\theta=0\}} \, \Big| \, \mathcal{F}(\tau)
            \Big]
        \Big)^+\right],
    \end{split}
\end{equation*}
and this last expression is equal to the expectation in \eqref{expected_payoff_via_Pi_process}, from the definition of the process $\Pi(\cdot)$ and the fact that $\tau$ is $\mathcal{F}_\tau$--measurable.

\begin{remark}
    The expectation in \eqref{expected_payoff_via_Pi_process} is still well-defined for the infinite time horizon $T = \infty$. 
    Indeed, on the event $\{\tau < \infty\}$ this is obvious, while on the event $\{\tau = \infty\}$ the expression is well-defined since the functions $c_i(\cdot), \, i = 0, 1$ are defined on $[0, T]$, while the process $\Pi(\cdot)$ of \eqref{cond_prob_dynamics} has a limit at infinity (by the Martingale Convergence Theorem).
\end{remark}

The identity \eqref{expected_payoff_via_Pi_process} and the fact that the processes $X(\cdot)$ and $\Pi(\cdot)$ generate the same filtrations enable us to reinterpret the optimal stopping problem \eqref{value_function} as maximizing the expression \eqref{expected_payoff_via_Pi_process} over all stopping times $\tau$ of the filtration generated by the process $\Pi(\cdot)$. To solve this problem, we embed it into a general Markovian framework allowing for arbitrary initial data $(t,\pi)$. 

\subsubsection{General Markovian Framework}

We fix a probability space $(\Omega, \fc, \pr)$, supporting a standard Brownian motion $B = \{B(t), \, 0 \le t < \infty\}$. For each $\pi \in [0, 1]$, we consider a stochastic process $\Pi^\pi = \{\Pi^\pi(t), 0 \le t < \infty \}$ with state space $[0, 1]$, which satisfies 
\begin{equation}\label{cond_process_with_arbitrary_position}
    \Pi^\pi(t) = \pi + \int_0^t a\, \Pi^\pi(s)(1-\Pi^\pi(s)) \, dB(s), \quad 0 \le t < \infty.
\end{equation}
The processes $\Pi^\pi(\cdot), \, \pi \in [0,1]$ mimic the conditional process \eqref{cond_prob_dynamics} with arbitrary starting positions $\pi$. For any $\pi \in [0,1]$, the Lipschitz property of the coefficients on $[0, 1]$ guarantees the existence and uniqueness of the strong solution to the equation \eqref{cond_process_with_arbitrary_position} (see, e.g., Chapter 5.2.B in Karatzas \& Shreve \cite{BMSC}). 
Moreover, the processes $B(\cdot)$ and $\Pi^\pi(\cdot)$, $\pi \in (0,1)$, generate the same filtration, as in the proof of Lemma \ref{lem_properties_of_cond_prob_process}.
Hence, regardless of $\pi$, we let $\mathbb{F} \coloneqq \{\mathcal{F}(t)\}_{t \ge 0}$ be the augmentation of the filtration generated by the processes $\Pi^\pi(\cdot)$ $\pi \in (0,1)$, i.e., we set $\mathcal{F}(t) \coloneqq \overline{\sigma}(B(s), 0 \leq s \le t)$. In addition, we denote by $\mathcal{T}_T$ the collection of all $\mathbb{F}$--stopping times $\tau$ such that $\pr(\tau \le T) = 1$.

For fixed $\pi \in [0,1]$, time horizon $T$, and initial time $t$, we now consider the corresponding optimal stopping problem of maximizing the expected reward 
\begin{equation}\label{expected_reward_general_pi}
    J_T(t, \pi, \tau) \coloneqq \ex\left[\Big(c_1(t+\tau) \, \Pi^\pi(\tau) - c_0(t+\tau) \,(1 - \Pi^\pi(\tau))\Big)^+\right]
\end{equation}
over all stopping times $\tau \in \mathcal{T}_{T-t}$, where the functions $c_i: [0, T] \to [0, 1], \, i = 0, 1$, satisfy the assumptions \ref{reg_assm_1}--\ref{reg_assm_5}.
We denote the gain function of this problem by 
\begin{equation}\label{gain_function}
    G(t, \pi) \coloneqq \big(c_1(t)\pi - c_0(t)(1-\pi)\big)^+,
\end{equation}
and its value function by 
\begin{equation}\label{value_function_general_osp}
    V_T(t, \pi) \coloneqq \sup\limits_{\tau \in \mathcal{T}_{T-t}} J_T(t, \pi, \tau) = \sup\limits_{\tau \in \mathcal{T}_{T-t}} \ex\big[G \big(t+\tau, \Pi^\pi(\tau)\big)\big].
\end{equation}

Clearly, the original optimal stopping problem \eqref{value_function} can be embedded into the one above by setting $\pi = p$, $t = 0$. Therefore, if we find the value function $V_T(\cdot, \cdot)$ of \eqref{value_function_general_osp} and the corresponding optimal stopping time, we will automatically solve the original optimal stopping problem as well. For some future arguments, it will be convenient to treat the processes $\Pi^\pi(\cdot), \, \pi \in [0,1]$ of the new problem, as those obtained from the original problem, i.e., a posteriori probability processes. Thus, we conclude this subsection by describing how to revert the new problem to the original.

\subsubsection{Reverting Back to the Original Problem}\label{subsubsec_reversion_back}

Consider the given probability space $(\Omega, \fc, \pr)$, a starting position $\pi \in [0, 1]$, and the process $\Pi^\pi(\cdot)$ of \eqref{cond_process_with_arbitrary_position}. For any such $\pi$, consider a new probability space $(\Omega^\pi, \fc^\pi, \pr^\pi)$ rich enough to support a Bernoulli random variable $\theta^\pi$, satisfying $\pr^\pi(\theta^\pi=1) = \pi = 1 - \pr^\pi(\theta^\pi=0)$, and a standard Brownian motion $W^\pi = \{W^\pi(t), \, 0 \le t < \infty\}$ independent of $\theta^\pi$. Given $\pi, \, \theta^\pi$, and $W^\pi(\cdot)$, consider on this new probability space the stochastic process
\begin{equation*}
    X^\pi(t) = (a\theta^\pi + b)\,t + W^\pi(t), \quad 0 \le t < \infty.
\end{equation*}
Since we want to build a new process, which resembles the behavior of $\Pi^\pi(\cdot)$ but is also an \textit{a posteriori probability process}, we define the conditional probability process 
\begin{equation}\label{cond_process_general_pi}
    \widehat{\Pi}^\pi(t) \coloneqq \pr^\pi\left(\theta^\pi = 1 \mid \mathcal{F}^{X^\pi}(t)\right), \quad 0 \le t < \infty,
\end{equation}
where $\mathcal{F}^{X^\pi}(t) \coloneqq \overline{\sigma}(X^\pi(s), 0 \le s \le t)$.

Lemma \ref{lem_properties_of_cond_prob_process} implies that this process of \eqref{cond_process_general_pi} satisfies the equation \eqref{cond_process_with_arbitrary_position}. Thus, the strong uniqueness of the solution to this equation implies that the processes $\Pi^\pi(\cdot)$ and $\widehat{\Pi}^\pi(\cdot)$ have the same law. Hence, denoting by $\ex^{\pi}$ the expectation with respect to $\pr^{\pi}$ and recalling the notation of \eqref{gain_function} and \eqref{value_function_general_osp}, we obtain the desired equality
\begin{equation*} 
    V_T(t, \pi) = \sup_{\tau \in \widehat{\mathcal{T}}^\pi_{T-t}} \ex^{\pi} \left[ G\left(t + \tau, \widehat \Pi^{\pi}(\tau)\right)\right].
\end{equation*}
Here, $\widehat{\mathcal{T}}^\pi_{T-t}$ denotes the set of all $\left\{\mathcal{F}^{X^\pi}(t)\right\}_{t \ge 0}$ stopping times $\tau$ such that $\pr^{\pi}(\tau \leq T-t) = 1$, and we are able to use this filtration as a direct consequence of Lemma \ref{lem_properties_of_cond_prob_process}. It is now clear that the processes $\Pi^\pi(\cdot)$ and $\widehat{\Pi}^\pi(\cdot)$ are essentially identical. We use the notation above when we need the a posteriori probability property of the observation process.

\subsection{An Equivalent Formulation via a Different Payoff Function}\label{subsec_equivalent_osp}

For several future arguments, it will be convenient to study an alternative optimal stopping problem, whose only difference from the original problem \eqref{value_function_general_osp} is the modified gain function
\begin{equation}\label{gain_function_with_indicator}
    \widetilde{G}_T(t, \pi) \coloneqq g(t,\pi) \cdot \mathbf{1}_{\{t < T\}},\  
    \ \
    g(t, \pi) \coloneqq c_1(t)\pi - c_0(t)(1-\pi).
\end{equation}
The original gain function $G(\cdot, \cdot)$ of \eqref{gain_function} can be written as $G(t, \pi) = g(t, \pi)^+$, so the only difference between $\widetilde{G}_T(\cdot, \cdot)$ and $G(\cdot, \cdot)$ is that the positive part is ``changed'' to a time indicator.

For a fixed $\pi \in [0,1]$, consider now a corresponding optimal stopping problem for the same process $\Pi^\pi(\cdot)$ of maximizing the expected reward \begin{equation*}
        \widetilde{J}_T(t, \pi, \tau) 
        \coloneqq
        \ex\left[
            \widetilde{G}_T\big(t + \tau, \Pi^\pi(\tau) \big)
            \right] 
        = 
        \ex\left[
                \Big(c_1(t+\tau) \, \Pi^\pi(\tau) - c_0(t+\tau) \, \big(1 - \Pi^\pi(\tau) \big) \Big) \cdot \mathbf{1}_{\{\tau < T\}}
            \right]
\end{equation*}
again over all stopping times $\tau \in \mathcal{T}_{T-t}$, and denote the new value function by 
\begin{equation}\label{value_function_with_indicator_gain}
    \widetilde{V}_T(t, \pi) \coloneqq \sup\limits_{\tau \in \mathcal{T}_{T-t}} \widetilde{J}_T(t, \pi, \tau) = \sup\limits_{\tau \in \mathcal{T}_{T-t}} \ex\left[\widetilde{G}_T(t+\tau, \Pi^\pi(\tau))\right].
\end{equation}

This new problem may fail to admit an optimal stopping time. The reason is a potential failure (for some choices of the discount functions $c_i(\cdot), i = 0, 1$) of the gain function $\widetilde{G}_T(\cdot, \cdot)$ to be upper semi-continuous at the terminal time $T$, which is a common requirement in these types of problems (see, e.g., Chapter I.2.2 in \cite{PesShi06}). However, even under these circumstances, the value function \eqref{value_function_with_indicator_gain} is well-defined and we have the following equivalence.

\begin{Prop}\label{prop_gain_functions_equivalnce}
    The problems \eqref{value_function_general_osp} and \eqref{value_function_with_indicator_gain} have the same value functions
    \begin{equation}\label{equality_value_functions}
        V_T(t, \pi) = \widetilde{V}_T(t, \pi) = \sup_{\tau \in \widehat{\mathcal{T}}^\pi_{T-t}} \ex^{\pi} \left[ \widetilde{G}_T\left(t + \tau, \widehat \Pi^{\pi}(\tau)\right)\right], \quad \forall \,  (t, \pi) \in [0, T) \times [0, 1].
    \end{equation}
\end{Prop}

The proof follows standard arguments and is located in Appendix \ref{subsec_app_gain_equivalence}.

\subsection{An Equivalent Formulation via American Options}\label{subsec_american_options_formulation}

Another reformulation of the problem \eqref{value_function_general_osp} is in the context of American call options. This choice is motivated by the work of Jaillet, Lamberton, and Lapeyre \cite{Lamberton}, whose results will be crucial in our proof of the spatial differentiability (smooth-fit property) of the value function $V_T(\cdot, \cdot)$ and of some other arguments. We describe briefly a common framework for American options, and then show that the problem \eqref{value_function_general_osp} can be embedded into this framework after appropriate changes of measure and variables inspired by \cite{EksWang24}. Specifically, the measure change recovers the conditional distribution of the posterior given a realization of $\theta$ in the sequential testing experiment of \cite{Shiryaev67}.

In the American option framework, one faces the following optimal stopping problem. Given a filtered probability space $\left(\overline{\Omega}, \overline{\mathcal{F}}, \overline{\mathbb{F}}, \overline{\pr}\right)$, a finite time horizon $T$, an interest rate $r: [0, T] \to [0, \infty)$, and a sufficiently nice reward function $\psi: [0, \infty) \to [0, \infty)$, often of the form $\psi(y) = (e^y - K)^+$ or $\psi(y) = (K - e^y)^+$, one observes a family of processes $Y^{t, y}(\cdot)$ that satisfy 
\begin{equation*}
    Y^{t, y}(s) = y + \int_t^s \mu(v, Y^{t, y}(v)) \, dv + \int_t^s \sigma(v, Y^{t, y}(v)) \, d\overline{B}(v), \quad t \le s \le T.
\end{equation*}
Here, $(t, y)$ is a starting position of the process, $\mu(\cdot, \cdot)$ and $\sigma(\cdot, \cdot)$ are sufficiently regular functions, and $\overline{B}(\cdot)$ is an $\overline{\mathbb{F}}$--Brownian motion. For each starting position $(t, y)$, the goal is to find an $\overline{\mathbb{F}}$--stopping time $\tau$ such that $\overline{\pr}(0 \le \tau \le T - t) = 1$ and $\tau$ maximizes the expected payoff
\begin{equation}\label{expeted_payoff_american_options}
    \overline{\ex} \left[ 
        e^{-\int_t^{t + \tau} r(s) \, ds} \, \psi\big(Y^{t, y}(t + \tau)\big)
    \right].
\end{equation}
We denote the value function of this problem by
\begin{equation}\label{value_function_american_options}
    u_T(t, y) \coloneqq \sup_{0 \le \tau \le T-t} \overline{\ex} \left[ 
        e^{-\int_t^{t+\tau}
        r(s) \, ds} \, \psi\big(Y^{t, y}(t + \tau)\big)
    \right].
\end{equation}

\begin{remark}
    The above is a ``classical'' optimal stopping problem with a discounted reward for a diffusion process. Under some additional structural assumptions on the reward function $\psi(\cdot)$, such as convexity or the already mentioned special forms $\psi(y) = (e^y - K)^+$, $\psi(y) = (K - e^y)^+$, the problem falls into the classical framework for American options. In particular, for our purposes, we will have $\psi(y) = (e^y-1)^+$. 
    Moreover, the above notation purposefully aligns closely with the one in \cite{Lamberton}. Specifically, the value function \eqref{value_function_american_options} is similar to equation (2.2) of \cite{Lamberton}.
\end{remark}

We assume throughout this subsection that $\pi \in (0,1)$, and $c_i(T) > 0$, $i = 0, 1$, to fall into the setting of \cite{Lamberton}, which will be required as a starting point in some of our proofs. The ultimate generalization to $c_i(T)\geq 0, \, i = 0, 1$, will be obtained by limit-based arguments. We embed the problem \eqref{value_function_general_osp} into the above framework as follows. First, we recall that the process $\Pi^\pi(\cdot)$ satisfies the equation \eqref{cond_process_with_arbitrary_position} for any fixed starting position $\pi \in (0, 1)$. Therefore, for any $t \ge 0$, we have 
\begin{equation}\label{pi_as_stoch_exponentials}
    \Pi^\pi(t) = \pi \cdot \mathcal{E}\left(\int_0^t a \big(1 - \Pi^\pi(s)\big) \, dB(s) \right),\quad 
    1 - \Pi^\pi(t) = (1 - \pi) \cdot \mathcal{E}\left( - \int_0^t a \Pi^\pi(s) \, dB(s)\right),
\end{equation}
where $\mathcal{E}(\cdot)$ denotes the stochastic exponential. 
Then we define a new probability measure $\overline{\pr}(\cdot)$ on our  space $(\Omega, \mathcal{F})$ with the property  
\begin{equation}\label{new_measure}
    \overline{\pr}(A) 
    \coloneqq 
    \ex\left[
        \mathbf{1}_A \cdot \mathcal{E}\left(-\int_0^t a \Pi^\pi(s) \, dB(s) \right) 
    \right] \quad \text{for every } \, t \in [0, T], \, A \in \mathcal{F}(t),
\end{equation}
and obtain the following representation of the value function $V_T(\cdot, \cdot)$:
\begin{equation}\label{value_function_another_measure}
    \begin{split}
        V_T(t, \pi) &= \sup_{\tau \in \mathcal{T}_{T-t}} \ex\left[\Big(c_1(t+\tau)\Pi^\pi(\tau) - c_0(t+\tau)\big(1-\Pi^\pi(\tau)\big) \Big)^+ \right]
        \\
        &= (1-\pi) \sup_{\tau \in \mathcal{T}_{T-t}} \ex\left[
            \mathcal{E}\left( - \int_0^\tau a \Pi^\pi(s) \, dB(s)\right)
            c_0(t+\tau)\left(\frac{c_1(t+\tau)}{c_0(t+\tau)}\cdot \frac{\Pi^\pi(\tau)}{1-\Pi^\pi(\tau)} - 1 \right)^+ 
        \right]
        \\
        &= (1-\pi) \sup_{\tau \in \mathcal{T}_{T-t}} \overline{\ex}\left[
            c_0(t+\tau)\left(\frac{c_1(t+\tau)}{c_0(t+\tau)}\cdot \frac{\Pi^\pi(\tau)}{1-\Pi^\pi(\tau)} - 1 \right)^+ 
        \right].
    \end{split}
\end{equation}
Here, $\overline{\ex}[\, \cdot \,]$ denotes expectation with respect to the probability measure $\overline{\pr}$ in \eqref{new_measure}, the first equality is the definition of the value function, and the second equality follows from rearranging terms, using the representation \eqref{pi_as_stoch_exponentials} and exploiting the fact that both $c_0(\cdot)$ and $\mathcal{E}(\cdot)$ are non-negative. 

To reduce the expression under the expectation $\overline{\ex}[\, \cdot \,]$ to the form of \eqref{expeted_payoff_american_options}, we introduce the following notation. We define the functions $\beta_i: [0, T] \to \mathbb{R}, \, i = 0, 1,$ by
\begin{equation}\label{definition_beta_functions}
    \beta_i(t) \coloneqq \frac{c_i'(t)}{c_i(t)}, \, \, \text{ so that for any } \, \, 0 \le t \le s \le T \, \, \text{ we get }  \, \, \exp\left(\int_t^s \beta_i(v) \, dv \right) = \frac{c_i(s)}{c_i(t)}.
\end{equation}
We also define the process
\begin{equation}\label{girsanov_BM_definition}
    \overline{B}(t) \coloneqq B(t) + a \int_0^t \Pi^\pi(s) \, ds, \quad 0 \le t < \infty,
\end{equation}
and note that it is a standard Brownian motion under the measure $\overline{\pr}(\cdot)$ on the strength of the Girsanov theorem (see, e.g., Chapter 3.5 in \cite{BMSC}). With this notation, the process inside the expectation $\overline{\ex}[\, \cdot \,]$ in \eqref{value_function_another_measure} can be written as
\begin{equation*}
    \begin{split}
        \frac{c_1(t+\tau)}{c_0(t+\tau)}\cdot \frac{\Pi^\pi(\tau)}{1-\Pi^\pi(\tau)}
        &= 
        \frac{\pi \, c_1(t)}{(1-\pi)\, c_0(t)}\, \exp\left(\int_t^{t + \tau} \Big(\beta_1(s) - \beta_0(s) \Big) \, ds \right) \mathcal{E}\left(a \, \overline{B}(\tau)\right)
        \\& = 
        \frac{\pi \, c_1(t)}{(1-\pi) \, c_0(t)}\, \exp\left(\int_t^{t + \tau} \left(\beta_1(s) - \beta_0(s) - \frac{a^2}{2}\right) ds + a\,  \overline{B}(\tau)\right).
    \end{split}
\end{equation*}
Here, the first equality follows from the identity on the right of \eqref{definition_beta_functions} for the fraction of discount functions, the representation \eqref{pi_as_stoch_exponentials}, and the definition \eqref{girsanov_BM_definition} for the fraction of the a posteriori processes. The second equality follows from the definition of the stochastic exponential and the fact that $\overline{B}(\cdot)$ is Brownian motion under $\overline{\pr}(\cdot)$.

We deduce that the value function $V_T(\cdot, \cdot)$ can be written in a form similar to \eqref{value_function_american_options} as
\begin{equation}\label{value_function_sup_american_option}
    V_T(t,\pi) = (1-\pi) \, c_0(t) \sup_{\tau \in \mathcal{T}_{T - t}} \overline{\ex}\left[
        e^{\int_t^{t + \tau} \beta_0(s) \, ds} \, \psi\big(Y^{t, y}(t + \tau)\big)
    \right],
\end{equation}
where $\psi: \mathbb{R} \to [0, 1]$ and $y \coloneqq y(t, \pi)$ are defined by
\begin{equation}\label{change_of_initial_position}
    \psi(y) \coloneqq (e^y - 1)^+ 
    \quad \text{ and } \quad
    y(t, \pi) \coloneqq \log \left(\frac{\pi \, c_1(t)}{(1-\pi) \, c_0(t)} \right),
\end{equation}
and we have
\begin{equation*}
        Y^{t,y}(s) \coloneqq y + \int_t^s \left(\beta_1(v) - \beta_0(v) - \frac{a^2}{2} \right) dv + \int_t^s a \, d\overline{B}(v), \quad t \le s \le T.
\end{equation*}
Note that assumptions \ref{reg_assm_1}--\ref{reg_assm_5} imply $\beta_0(\cdot) \le 0$, whereas assumption \ref{reg_assm_5} also implies $\beta_1(\cdot) - \beta_0(\cdot) \le 0$, which together ensure well-posedness of the problem above when $T  = \infty$.

Crucially, we have the following equality between the value functions \eqref{value_function_general_osp} and \eqref{value_function_american_options}:
\begin{equation}\label{american_call_value_fn}
     V_T(t, \pi) = (1-\pi) \, c_0(t) \, u_T\big(t, y(t,\pi)\big).
\end{equation}

\begin{remark} \label{American_put_remark}
    Equation \eqref{american_call_value_fn} shows the relationship between $V_T(\cdot,\cdot)$ and a value function resembling that of an American call option. It is clear that instead of factoring out the term $c_0(t+\tau)\,(1-\Pi^\pi(\tau)$) in the above calculations, one could have factored out $c_1(t+\tau) \, \Pi^\pi(\tau)$ to obtain a relationship similar to that of an American put option. In particular, $V_T(t,\pi) = \pi \; c_1(t) \; \widetilde u_T(t, \widetilde y(t,\pi)),$
    where 
    \begin{align*}
        \widetilde u_T(t, \widetilde{y}) &\coloneqq \sup_{\tau \in \mathcal{T}_{T-t}} \widetilde{\ex} \left[e^{\int_t^{t+\tau}\beta_1(s)ds} \left(1 - e^{\widetilde Y^{t,\widetilde y}(t+\tau)}\right)^+ \right],
        \\
        \widetilde Y^{t, \widetilde y}(s) &\coloneqq \widetilde y + \int_t^s \left(\beta_0(v) - \beta_1(v) - \frac{a^2}{2} \right) + \int_t^s a \, d \widetilde{B}(v), \quad t \leq s \leq T,
        \\
        \widetilde y(t,\pi) &\coloneqq -y(t,\pi) = \log \left(\frac{(1-\pi) \; c_0(t)}{\pi\; c_1(t)}\right),
    \end{align*}
    and $\widetilde\ex[\, \cdot \,]$ denotes expectation with respect to a probability measure $\widetilde\pr(\cdot)$, under which $\widetilde B(\cdot)$ is standard Brownian motion.
\end{remark}

\section{Analysis of the Optimal Stopping Problem}\label{sec_main_results}

We turn now to the analysis of the problem \eqref{value_function_general_osp}. Our approach follows a structure similar to the standard methods in the field of optimal stopping. However, the general form of the discount functions 
$c_i(\cdot), \, i = 0, 1,$ introduces additional complexity and requires some delicate arguments. To provide clarity and context, we begin by outlining the structure of this section and the next.
Our main results consist of two parts: those about the value function $V_T(\cdot, \cdot)$, and those about the optimal stopping boundary. We cover the former in Section \ref{sec_main_results} and the latter in Section \ref{sec_boundary}. 

In Subsection \ref{subsec_preliminaries_main_res}, we start analyzing the problem \eqref{value_function_general_osp}, show the existence of a boundary that separates the so-called stopping and continuation regions in Proposition \ref{prop_structure_of_regions}, and state our main results in Theorem \ref{main_result}. Subsection \ref{subsec_properties_value_function} covers results about the value function; in particular
\begin{enumerate}
[topsep=-5pt, itemsep=-1ex, partopsep=1ex, parsep=2ex]
    \item Proposition \ref{prop_properties_of_value_func} shows that the value function is uniformly continuous in the interior of its domain, as well as states some of its regularity and monotonicity properties,
    \item Proposition \ref{prop_boundary_problem_value_func} confirms that the value function is indeed a solution to an associated free-boundary problem,
    \item Proposition \ref{prop_C1_value_func} proves the smooth-fit property, 
    \item Proposition \ref{ito_for_value_function} verifies the applicability of Itô's Lemma to $V_T(\cdot, \cdot)$, while
    \item Proposition \ref{prop_uniqueness_free_boundary}  establishes the uniqueness of the solution to the associated free-boundary problem in a suitable class of functions.
\end{enumerate}
Section \ref{sec_boundary} contains results about the optimal stopping boundary. Subsection \ref{subsec_properties_of_the_boundary} consists of:
\begin{enumerate}
[topsep=0pt, itemsep=-1ex, partopsep=1ex, parsep=2ex]
    \item Proposition \ref{prop_optimal_stoppinGime}, which shows that the optimal stopping time is the first hitting time of the corresponding boundary;
    \item A technical Lemma \ref{boundary_measure_zero}, guaranteeing that the boundary of the optimal stopping region has Lebesgue measure zero; and the key
    \item Proposition \ref{prop_boundary_equation}, which characterizes the optimal stopping boundary in terms of a particular integral equation.
\end{enumerate}
Subsection \ref{subsec_boundary_monotonicity} studies sufficient conditions for the monotonicity of the stopping boundary under an appropriate transformation as well as its continuity, while Subsection \ref{subsec_examples} concludes with examples and numerical results.

\subsection{The Main Results}\label{subsec_preliminaries_main_res}

Since $\tau \equiv 0$ is an admissible stopping time, the value function $V_T(\cdot, \cdot)$ of \eqref{value_function_general_osp} dominates the gain function $G(\cdot, \cdot)$ of \eqref{gain_function}; i.e., for all $(t, \pi) \in [0, T] \times [0, 1]$, we have $V_T(t, \pi) \ge G(t, \pi)$.
Therefore, we introduce the so-called \textit{continuation region} 
\begin{equation}\label{def_continuation_region}
    \mathcal{C}_T \coloneqq \{(t, \pi) \in[0, T] \times[0,1]: V_T(t, \pi) > G(t, \pi)\}
\end{equation}
and the \textit{stopping region}     
\begin{equation}\label{def_stopping_region}
    \mathcal{S}_T \coloneqq \{(t, \pi) \in[0, T] \times[0,1]: V_T(t, \pi) = G(t, \pi)\}.
\end{equation}
These are separated by some function $b_T: [0, T] \to [0, 1]$ as is established in the following proposition, the proof of which is given in Appendix \ref{subsec_app_structure_of_regions}.

\begin{Prop}\label{prop_structure_of_regions}
    For any $T \in [0 , \infty]$, there exists a function $b_T: [0, T] \to [0, 1]$, such that the continuation and stopping regions \eqref{def_continuation_region}, \eqref{def_stopping_region} are represented respectively as 
    \begin{equation}\label{regions_via_boundary}
        \begin{split}
            \mathcal{C}_T &= \{(t, \pi) \in [0, T] \times (0, 1): \pi < b_T(t)\}, \\
            \mathcal{S}_T &= \{(t, \pi) \in [0, T] \times [0, 1]: \pi \ge b_T(t) \text{ or } \pi = 0\},
        \end{split}
    \end{equation}
    where 
    \begin{equation}\label{boundary_definition}
        b_T(t) \coloneqq \inf \{ \pi \in (0, 1]: V_T(t, \pi) = G(t, \pi) \}, \quad 0 \le t \le T.
    \end{equation}
\end{Prop}

We are now ready to formulate our main result.

\begin{Th}\label{main_result}
    For any $T \in [0, \infty]$, the value function $V_T(\cdot, \cdot)$ of the optimal stopping problem \eqref{value_function_general_osp} and the function $b_T: [0, T] \to [0, 1]$ of \eqref{boundary_definition}, solve the free-boundary problem
    \begin{equation}\label{value_function_boundary_problem}
        \begin{cases}
            \partial_t V_T(t, \pi) + \dfrac{a^2}{2}\pi^2 (1 - \pi)^2 \, \partial_{\pi\pi} V_T(t, \pi) = 0, & 0 < \pi < b_T(t), \, t \in [0, T) \\
            V_T(t, \pi) = G(t, \pi), & \pi \ge b_T(t) \text{ or } \pi = 0, \, t \in [0, T), \\
            V_T(T, \pi) = G(T, \pi), & \pi \in [0, 1].
        \end{cases}
    \end{equation}
    In particular, $V_T(\cdot, \cdot)$ is of class $C^{1, 2}$
    away from $b_T(\cdot)$, and the so-called ``smooth-fit'' condition holds in the sense that $V_T(t, \cdot)$ is of class $C^1\big((0, 1)\big)$ for any $t \in [0, T)$.
    
    Furthermore, an optimal stopping time for the problem \eqref{value_function_general_osp} with starting position $(t, \pi)$ is given by
    \begin{equation}\label{optimal_stoppinGime}
        \tau^*_T \coloneqq \inf\{s \ge 0: \Pi^\pi(s) \ge b_T(t + s)\} \wedge (T - t),
    \end{equation}
    for the function $b_T(\cdot)$ of \eqref{boundary_definition}, which satisfies the integral equation
    \begin{equation}\label{integral_equation_for_boundary}
        \begin{split}
            b_T(t) \cdot \big(c_1(t) &+c_0(t)\big)- c_0(t)
            =
            \ex\left[G\left(T, \, \Pi^{b_T(t)}(T-t)\right)\right] 
            \\&- 
            \int_{0}^{T-t}
            \Big(c_1'(t+u) + c_0'(t+u)\Big) 
            \cdot
            \ex \left[
                \Pi^{b_T(t)}(u) 
                \cdot 
                \mathbf{1}_{\{\Pi^{b_T(t)}(u) \ge b_T(t + u)\}}
            \right]
            \\&\quad \quad \quad
            - 
            c_0'(t+u) 
            \cdot 
            \pr \left(
                \Pi^{b_T(t)}(u) \ge b_T(t + u)
            \right)
            du, \quad 0 \le t < T.
        \end{split}
    \end{equation}
\end{Th}

\begin{remark}
    In addition to the statements of Theorem \ref{main_result}, we also show that the pair $(V_T(\cdot, \cdot), b_T(\cdot))$ is the unique solution to the free-boundary problem \eqref{value_function_boundary_problem} in a suitable class of functions. Since the formulation of this uniqueness requires several technical details, it is relegated to Proposition \ref{prop_uniqueness_free_boundary}.
\end{remark}

We establish the properties of the value function $V_T(\cdot, \cdot)$ in the next subsection, and then the properties of the boundary $b_T(\cdot)$ in Section \ref{sec_boundary}. As mentioned in the introduction to this section, for clarity of exposition we divide our arguments into several propositions and lemmata. This way, the statements of Theorem \ref{main_result} follow from Propositions \ref{prop_boundary_problem_value_func}, \ref{prop_C1_value_func}, \ref{prop_optimal_stoppinGime}, and \ref{prop_boundary_equation}.

\subsection{Properties of the Value Function}\label{subsec_properties_value_function}

We need to show that the pair $(V_T(\cdot, \cdot), b_T(\cdot))$ solves the system \eqref{value_function_boundary_problem}, and $V_T(\cdot, \cdot)$ satisfies the smooth-fit principle. This requires some properties of $V_T(\cdot, \cdot)$, including its joint continuity. 

\begin{Prop}\label{prop_properties_of_value_func}
    For any $T \in [0, \infty]$,
    \begin{enumerate}[label={(\arabic*)}, topsep=-5pt,itemsep=-1ex,partopsep=1ex,parsep=1ex]
        \item The value function $V_T(\cdot,\cdot)$ of \eqref{value_function_general_osp} is uniformly continuous on the set $[0,T] \times [0,1]$. \label{value_func_continuity}
    \end{enumerate}
    Furthermore, for any fixed $t \in [0, T]$,
    \begin{enumerate}[label={(\arabic*)}, topsep=-5pt,itemsep=-1ex,partopsep=1ex,parsep=1ex, resume]
        \item The mapping $\pi \mapsto V_T(t, \pi)$ is convex.  \label{value_func_convexity}
        \item The mapping $\pi \mapsto V_T(t, \pi)$ is nondecreasing. \label{value_func_nondec}
        \item The mapping $\pi \mapsto V_T(t, \pi)$ is Lipschitz continuous, with Lipschitz constant $c_1(t)+c_0(t)$.\label{value_func_Lipschitz_pi}
    \end{enumerate}
\end{Prop}

The proof is located in Appendix \ref{subsec_app_prop_properties_value_func}.

\begin{remark}
    The convexity of the value function, along with the PDE in \eqref{value_function_boundary_problem}, imply that the function $V_T(\cdot, \pi)$ is decreasing in the continuation region $\mathcal{C}_T$ for each $\pi \in (0, 1)$.
\end{remark}

We are now ready to prove our main results about the value function $V_T(\cdot, \cdot)$. 

\begin{Prop}\label{prop_boundary_problem_value_func}
    For any $T \in [0, \infty]$, the value function \eqref{value_function_general_osp} solves the free-boundary value problem \eqref{value_function_boundary_problem}. In particular, the function $V_T(\cdot, \cdot)$ is $C^{1, 2}$ in $\mathcal{C}_T$.
\end{Prop}

\begin{proof}
The validity of the differential equation in \eqref{value_function_boundary_problem} follows from a fairly standard procedure in the theory of optimal stopping. Thus, we omit the argument and refer the reader to the proof of Theorem 7.7 in Chapter 2 of Karatzas and Shreve \cite{KarShr98} for details (see also \cite{Jacka91}); this relies heavily on the fact that the continuation region $\mathcal{C}_T$ is open, which is valid in our case. Indeed, it is clear that for $t = T$ the region $\mathcal{C}_T$ is empty, since $V_T(T, \cdot) \equiv G(T, \cdot)$, while for $t\in [0, T)$ the region $\mathcal{C}_T$ can be alternatively written as
\begin{equation*}
    \begin{split}
        \mathcal{C}_T &= \{(t, \pi) \in[0, T) \times (0,1): V_T(t, \pi) - G(t, \pi) > 0\} 
        \\&= \{(t, \pi) \in[0, T) \times (0,1): (t, \pi) \in H^{-1}(0, \infty)\},
    \end{split}
\end{equation*}
where $H \coloneqq V_T - G: [0, T) \times (0, 1) \to [0, \infty)$ is a continuous function due to \ref{reg_assm_1} and  Proposition \ref{prop_properties_of_value_func}, and we write $\pi \in (0,1)$ since $[0,T] \times \{0,1\} \subseteq \mathcal{S}_T$, trivially. Thus, the set $\mathcal{C}_T$ is open, being the preimage of the open set $(0, \infty)$ under a continuous mapping.
Two other identities of the system \eqref{value_function_boundary_problem} follow from Proposition \ref{prop_structure_of_regions}.
\end{proof}

\begin{remark}\label{remark_lsc}
   The fact that $\mathcal{C}_T$ is open implies the lower semi-continuity of the boundary $b_T(\cdot)$.
\end{remark}

For the smooth-fit property of the value function $V_T(\cdot, \cdot)$, we proceed as follows. First, we establish the smooth-fit for the finite time horizon case and under slightly more restrictive assumptions on the discount functions $c_i(\cdot)$ at terminal time $T$. This is done by exploiting the connection of our problem with American options, discussed in Subsection \ref{subsec_american_options_formulation}. Next, we show that the spatial derivative of the value function is uniformly locally Lipschitz. Finally, using both these results, we obtain the smooth-fit property for any $T \in [0, \infty]$ and under our original assumptions.

\begin{Prop}\label{prop_value_func_derivative_continuity}
    For any $T \in [0, \infty)$, the function $\partial_\pi V_T(\cdot, \cdot)$ exists and is jointly continuous on $[0, T) \times (0, 1)$ as long as $c_i(\cdot), \, i = 0, 1$, are of class $C^2([0, T])$ and $c_i(T) > 0, \, i = 0,1$.
\end{Prop}

This result is proved in Appendix \ref{subsec_app_v_pi_continuity}.

\begin{Prop}\label{prop_Lipschitz_property_spatial_derivative}
    Fix any $T \in [0,\infty]$ and $0 \leq t < t' < T$. The family of functions $\{\partial_\pi V_S(t, \cdot)\}_{t' \le S < T}$ is uniformly locally Lipschitz. More precisely, with $\mathcal{K}$ being the family of all compact subsets of $(0, 1)$, there exists a function $C: \mathbb{R}_+ \times \mathbb{R}_+ \times \mathcal{K} \to \mathbb{R}_+$ such that, for any fixed $(t, K) \in \mathbb{R}_+ \times \mathcal{K}$, the function $S \mapsto C(S, t, K)$ is non-increasing and for all $S \in (t, T)$ we have 
    \begin{equation}\label{lipschitzianity_spatial_derivative}
        \big| \partial_\pi V_S(t, x) - \partial_\pi V_S(t, y) \big| \le C(S, t, K) \cdot |x - y|, \quad \forall \, x, y \in K.
    \end{equation}
\end{Prop}

This result is proved in Appendix \ref{subsec_app_lipschitz} and requires high technical involvement, including time and space changes, introducing new processes, careful estimates, etc. We want to highlight that Propositions \ref{prop_value_func_derivative_continuity} and \ref{prop_Lipschitz_property_spatial_derivative}, together with the next result, provide a new approach to treating the spatial differentiability of the value function in problems of optimal stopping with an infinite time horizon.

\begin{Prop}\label{prop_C1_value_func}
    For any $T \in [0, \infty]$ and any fixed $t \in [0, T)$, the function $\pi \mapsto V_T(t, \pi)$ is continuously differentiable on $(0, 1)$.
\end{Prop}

\begin{proof} 
    Fix any $T \in [0, \infty]$ and $t \in [0, T)$. To prove that $V_T(t, \cdot)$ is of class $C^1\big((0, 1)\big)$, it suffices to show that $V_T(t, \cdot) \in C^1(K)$ for any compact subset $K$ of $(0, 1)$. Thus, we fix any such $K$ and proceed as follows.
    
    \underline{Step 1:} Let $\{T_n\}_{n \in \mathbb{N}}$ be any increasing sequence such that $T_n \in (t, T)$ and $T_n \uparrow T$. Note that the family of functions $\{V_{T_n}(t, \cdot)\}_{n \in \mathbb{N}}$ is increasing monotonically to $V_T(t, \cdot)$ by Lemma \ref{lem_value_incr_monotonically_T_upward} of the Appendix.
    
    \underline{Step 2:} Next, note that the functions $\{V_{T_n}(t, \cdot)\}_{n \in \mathbb{N}}$ are differentiable on $(0, 1)$, and, moreover, the family $\{\partial_\pi V_{T_n}(t, \cdot)\}_{n \in \mathbb{N}}$ is uniformly equicontinuous and uniformly bounded on $K$. The former is a consequence of Proposition \ref{prop_Lipschitz_property_spatial_derivative}, with uniform Lipschitz constant $C(T_1, t, K)$, while the latter follows from the Lipschitz property of the value functions proved in part \ref{value_func_Lipschitz_pi} of Proposition \ref{prop_properties_of_value_func}. As a result, we can apply the Arzel\`a--Ascoli theorem to find a subsequence $\{\partial_\pi V_{T_{n_j}}(t, \cdot)\}_{n \ge 0}$ that converges uniformly on $K$ as $j \to \infty$.
    
    \underline{Step 3:} The uniform convergence of $\{\partial_\pi V_{T_{n_j}}(t, \cdot)\}_{j \in \mathbb{N}}$ as well as the convergence of $V_{T_{n_j}}(t, \cdot)$ to $V_T(t, \cdot)$, in turn, imply that the function $V_T(t, \cdot)$ is differentiable and $\lim_{j \to \infty} \partial_\pi V_{T_{n_j}}(t, \pi) = \partial_\pi V_T(t, \pi)$ for any $\pi \in K$ (see, e.g., Theorem 7.17 in Rudin \cite{Rudin}). Since the convergence of the continuous functions $\partial_\pi V_{T_{n_j}}(t, \cdot)$ is uniform, $\partial_\pi V_T(t,\cdot)$ is continuous, completing the proof.
\end{proof}

\begin{remark}
    For any  $T \in [0, \infty)$, if  $c_i(T)>0$ and $c_i(\cdot)\in C^2([0,T])$ for $i=0,1$, the function $\partial_t V_T(\cdot,\cdot)$ is continuous almost everywhere across the boundary; i.e., for almost every 
    $t \in [0,T)$, $\lim_{(t',x')\to (t,b(t))} \partial_tV_T(t',x') =\partial_t G(t,b(t)).$
    To obtain this result, we apply Theorem 3.6 of \cite{Lamberton} to $u_T(\cdot,\cdot)$ once again; this implies that $u_T(\cdot,\cdot)$ satisfies Equation (2.1) of \cite{Blanchet}. Then, by Theorem 2.6 of \cite{Blanchet}, the variational inequality in Equation (1.1) of the same paper is also satisfied by $u_T(\cdot,\cdot)$. Theorem 1.1 of \cite{Blanchet} then provides the claimed almost everywhere continuity of $u_T(\cdot,\cdot)$, which extends to $V_T(\cdot,\cdot)$ thanks to \eqref{american_call_value_fn}.
\end{remark}

We next establish the following result about the applicability of Itô's lemma in our setting. The main purpose of this result is to support the derivation of the integral characterization \eqref{integral_equation_for_boundary} of the optimal stopping boundary, which will be formally established in Proposition \ref{prop_boundary_equation}. We also need the applicability of Itô's lemma to establish the uniqueness of the solution to the free boundary problem \eqref{value_function_boundary_problem}, established in Proposition \ref{prop_uniqueness_free_boundary}.

\begin{Prop}\label{ito_for_value_function}
    For any $T \in [0, \infty]$, It\^{o}'s rule is applicable to the value function $V_T(\cdot, \cdot)$. In particular, for any $(t, \pi) \in [0,T] \times [0,1]$ and $s \in [0,T-t]$, we have
    \begin{equation}
        \begin{split}
            V_T(t+s, \Pi^\pi(s)) = V_T(t, \pi) 
            &+ \int_0^s \mathbf{1}_{\{\Pi^\pi(u) \ge b_T(t+u) \}} \Big(c_1'(t+u)\Pi^\pi(u)
            -c_0'(t+u)(1-\Pi^\pi(u)) \Big)\, du
            \\&+ a\int_0^s \frac{\partial V_T}{\partial \pi}\big(t+u, \Pi^\pi(u)\big) \, \Pi^\pi(u) \, (1-\Pi^\pi(u)) \, dB(u).
        \end{split} \label{ito_equation_for_value_function}
    \end{equation}
\end{Prop}

The proof of Proposition \ref{ito_for_value_function} relies on standard mollification techniques and is carried out in Appendix \ref{subsec_app_ito}.

\begin{remark} 
    One could obtain the equation \eqref{ito_equation_for_value_function} by applying naively Itô's rule to the process $V_T(t + \cdot, \Pi^\pi(\cdot))$. 
    However, the classical two-dimensional Itô's lemma requires $C^{1, 2}$ global regularity of the underlying function, which is not the case in our setting, or in optimal stopping problems more generally. Therefore, a more delicate analysis is usually needed to obtain identities such as \eqref{ito_equation_for_value_function}. In the context of free-boundary problems, the applicability of Itô's rule can often be granted when additional structural properties of the value function and stopping boundary are available; for instance, it is helpful to have a continuous boundary with finite variation (see, e.g., Peskir \cite{Peskir05}). We circumvent such requirements for our problem by leveraging 
    \begin{enumerate}
        \item[(1)] the global existence and local boundedness of the weak derivatives of $V_T(\cdot,\cdot)$ (obtained through the relationship with American options derived in Section \ref{subsec_american_options_formulation}),
        \item[(2)] the joint continuity of $\partial_\pi V_T(\cdot, \cdot)$ (cf. Proposition \ref{prop_C1_value_func}), and 
        \item[(3)] the fact that the boundary of the continuation region has Lebesgue measure zero; see Lemma \ref{boundary_measure_zero}, proved independently.
    \end{enumerate}
    Note that these three conditions guarantee the applicability of Itô's rule in more general settings than the one of the current paper, as can be seen from the proof of Proposition \ref{ito_for_value_function}.
\end{remark}

We conclude this section by stating the uniqueness of the solution to the free-boundary problem \eqref{value_function_boundary_problem}. We write $W_b^{1, 2}$ to denote the space of bounded continuous functions with locally bounded weak derivatives up to order one in the temporal variable and order two in the spatial variable. 

\begin{Prop}\label{prop_uniqueness_free_boundary}
    Fix any $T \in [0, \infty]$ and a pair $\big(f(\cdot, \cdot), d(\cdot)\big)$ of functions $f: [0, T] \times [0, 1] \to \mathbb{R}$ and $d:[0, T] \to [0, 1]$. Define the strict hypograph of $d(\cdot)$ by
    $\mathcal{D} \coloneqq \{(t, \pi) \in [0, T) \times (0, 1): \pi < d(t) \}$.
    If $\big(f(\cdot, \cdot), d(\cdot)\big)$ satisfies
    \begin{align}
        &f(\cdot, \cdot) \in W_b^{1, 2},  \label{uninq_1}\\
        &\partial_t f(t, \pi) + \dfrac{a^2}{2}\pi^2 (1 - \pi)^2 \, \partial_{\pi\pi} f(t, \pi) = 0, & 0 < \pi < d(t), \, t \in [0, T), \label{uninq_2} \\
        & f(t, \pi) = g(t, \pi), & \pi \ge d(t), t \in [0, T), \label{uninq_3_1} \\
        & f(T, \pi) = G(T, \pi), & \pi \in [0, 1], \label{uninq_3_2} \\
        &f(t, \pi) \ge G(t, \pi), & (t, \pi) \in [0, T] \times [0, 1], \label{uninq_4} \\
        &d(\cdot) \text{ is lower semi-continuous,} \label{uninq_5} \\
        & \text{Leb}(\partial \mathcal{D}) = 0, \label{uninq_6}
    \end{align}
    for $G(\cdot, \cdot)$ and $g(\cdot, \cdot)$ as in \eqref{gain_function} and \eqref{gain_function_with_indicator}, then $f(\cdot, \cdot) \equiv V_T(\cdot, \cdot)$ everywhere and $ d(\cdot)=b_T(\cdot)$ Lebesgue almost everywhere.
\end{Prop}

The proof is given in Appendix \ref{subsec_app_uniqueness_free_boundary}. The uniqueness of the solution $f(\cdot, \cdot)$ to the above system is argued as in Theorem 7.12 of \cite{KarShr98}. To prove the uniqueness of the boundary $d(\cdot)$, we provide additional arguments based on the strict negativity in the continuation region $\mathcal{C}_T$ of the generator of $\Pi^\pi(\cdot)$ applied to the gain function $G(\cdot, \cdot)$.

\begin{remark}
    When the optimal stopping boundary $b_T(\cdot)$ is continuous (as, e.g., in Proposition \ref{prop_boundary_continuity} below), we can further obtain that $d(\cdot)\equiv b_T(\cdot)$.  To wit, if $b_T(\cdot)$ is continuous and $d(\cdot)$ is lower semi-continuous then $\{t: b_T(t) < d(t)\}$ is open. Were it also non-empty, it would have positive Lebesgue measure in contradiction to Proposition \ref{prop_uniqueness_free_boundary}.
\end{remark}

\section{Study of the Stopping Boundary}\label{sec_boundary}

\subsection{Properties of the Boundary}\label{subsec_properties_of_the_boundary}

The following result is a direct consequence of Propositions \ref{prop_structure_of_regions}, \ref{prop_properties_of_value_func}, and the general theory of optimal stopping; see, e.g., Chapter I.2.2. of \cite{PesShi06} and in particular, Corollary 2.9, whose proof carries over to our infinite-horizon setup due to the boundedness of the gain function $G(\cdot, \cdot)$.

\begin{Prop}\label{prop_optimal_stoppinGime}
    For any $T \in [0, \infty]$, the stopping time $\tau^*_T$ of \eqref{optimal_stoppinGime} is optimal for the problem \eqref{value_function_general_osp}.
\end{Prop}

\begin{remark}
    The stopping time $\tau^*_T$ of \eqref{optimal_stoppinGime} is not only optimal, but is also the smallest optimal stopping time for any $(t, \pi) \in [0, T] \times (0, 1]$. However, for $(t, 0), t \in [0, T]$, we have $\tau^*_T = T - t$, which makes it, instead, the largest optimal stopping. Note, however, that for all such points any stopping time is optimal since the payoff is zero regardless of the stopping moment.
\end{remark}

We show in this section that the stopping boundary satisfies a particular integral equation. To do that, we shall need to apply Itô's rule to $V_T(t + \cdot, \Pi^\pi(\cdot))$; this, in turn, will require establishing the following result.

\begin{Lem} \label{boundary_measure_zero}
    For any $T \in [0, \infty)$, the boundary $\partial \mathcal{C}_T$ of the continuation region $\mathcal{C}_T$ has zero Lebesgue measure as long as $c_i(\cdot), \, i = 0, 1$, are of class $C^2$ on the closed interval $[0, T]$ and $c_i(T) > 0$, $i = 0,1$.
\end{Lem}

The technical proof of Lemma \ref{boundary_measure_zero} is relegated to Appendix \ref{subsec_app_measure_zero}.

\begin{remark}
    Boundaries of optimal stopping problems often have nice monotonicity properties. In other cases, it is possible to prove the boundary's continuity (see, e.g., recent developments in \cite{DeAngelis15}, \cite{DeAngPes20}, \cite{DeAngelisStabile}). In each of these scenarios, the statement of Lemma \ref{boundary_measure_zero} follows immediately. However, as one can see from the examples in Subsection \ref{subsec_examples}, the boundary behavior in our setting can be very complicated, and even its continuity is hard to establish (even though we conjecture that the boundary should be continuous as a consequence of the $C^2$ regularity of the discount functions). Thus, establishing Lemma \ref{boundary_measure_zero} is indeed essential to proving the 
    applicability of Itô's lemma.
    
\end{remark}

We are finally ready to show that the function $b_T(\cdot)$ satisfies an integral equation.

\begin{Prop}\label{prop_boundary_equation}
    For any $T \in [0, \infty]$, the boundary function $b_T(\cdot)$, defined in         \eqref{boundary_definition}, satisfies the  integral equation \eqref{integral_equation_for_boundary}.
\end{Prop}

\begin{proof}
    Fix $T \in [0, \infty], \, t \in [0, T)$, $s \in [0, T-t)$, and $\pi \in (0, 1)$. Recall Itô's formula \eqref{ito_equation_for_value_function} and note that the last term on the right-hand side is a martingale, bounded in $L^2$ over the time horizon $[0,s]$ because the integrand process is bounded and $s$ is finite. This is a consequence of the boundedness of the process $\Pi^\pi(\cdot)$ and of the spatial derivative of the value function, as implied by Proposition \ref{prop_C1_value_func} and Proposition \ref{prop_properties_of_value_func}. Taking expectations on both sides of \eqref{ito_equation_for_value_function}, and invoking the optional sampling theorem for martingales which are bounded in $L^2$, we obtain 
    \begin{equation*}
            \ex\big[ V_T(t+s,  \Pi^\pi(s))\big] 
            = V_T(t, \pi) 
            + 
            \ex\left[
                \int_0^{s} \mathbf{1}_{\{\Pi^\pi(u) \ge b_T(t+u) \}} \Big(c_1'(t+u)\Pi^\pi(u) -c_0'(t+u)(1-\Pi^\pi(u)) \Big)\, du 
            \right].
    \end{equation*}
    Recall that at $\pi = b_T(t)$, $V_T(t, \pi) = G(t, \pi) = c_1(t)\pi - c_0(t)(1-\pi)$, and $V_T(T, \pi) = G(T, \pi)$ for $\pi \in (0, 1)$. For the case $T < \infty$, the integral equation \eqref{integral_equation_for_boundary} follows immediately by applying Fubini's theorem (using the boundedness of all terms under the integral), plugging in $\pi = b_T(t)$, and letting $s \uparrow T-t$. For the infinite time horizon, we can do the same by noting that the absolute value of the integrand in the Lebesgue integral above is bounded by $-c_0'(t+\cdot) - c_1'(t+\cdot)$, which is integrable over $[0, \infty]$ because $c_i(\infty), \, i = 0, 1$, are well defined. The result then follows by the dominated convergence theorem, Fubini's theorem, and letting $s \uparrow \infty$. This concludes the proof.
\end{proof}

Combining all of the above results, we obtain the statement of Theorem \ref{main_result}.

\subsection{Sufficient Conditions for Monotonicity of the Boundary under Appropriate Transformation}\label{subsec_boundary_monotonicity}

As mentioned already, the stopping boundary of our problem need not possess a nice structure. However, the established equivalence to the appropriate American call option allows us to obtain sufficient conditions for the boundary to be continuous and monotone under a corresponding explicit transformation. Choosing a suitable transformation to induce boundary monotonicity was also useful in the study of variable annuities in \cite{de2024variable}.  We shall impose the following assumptions on the functions $\beta_0(\cdot)$ and $\beta_1(\cdot)$ of \eqref{definition_beta_functions}:

\begin{enumerate}[leftmargin=35pt, label = \textbf{(B\arabic*)}]
    \item The functions $\beta_0(\cdot)$ and $\beta_1(\cdot) - \beta_0(\cdot)$ are non-increasing; \label{bdy_assm}
    \item The function $\beta_0(\cdot)$ is negative (equivalently, the function $c_0(\cdot)$ is strictly decreasing). \label{bdy_assm_2}
\end{enumerate}

In addition, we make the following standing assumption for this section, which is a slight modification of the assumption \ref{reg_assm_1}:

\begin{enumerate}[leftmargin=35pt, label = \textbf{(B\arabic*)}]
\setcounter{enumi}{2}
    \item $c_i(T) > 0$ for both $i = 0,1$. \label{bdy_assm_3}
\end{enumerate}

\begin{remark}
    Note that \ref{bdy_assm_2} resembles the assumption \ref{reg_assm_5}, which can be equivalently stated as follows: the function $\beta_1(\cdot) - \beta_0(\cdot)$ is negative. This way, under \ref{bdy_assm_2}, we add a monotonicity assumption to the already required negativity one.
\end{remark}

\begin{remark}
    The results of this subsection, i.e., Propositions \ref{prop_boundary_monotonicity} and \ref{prop_boundary_continuity}, may be obtained under the following alternative assumptions:
        \begin{enumerate}[leftmargin=35pt, label = \textbf{(B\arabic*)$^\prime$}]
    	\item The functions $\beta_1(\cdot)$ and $\beta_1(\cdot) - \beta_0(\cdot)$ are non-increasing; \label{bdy_assm_alt}
        \item The function $\beta_1(\cdot)$ is negative (equivalently, the function $c_1(\cdot)$ is strictly decreasing). \label{bdy_assm_2_alt} 
    \end{enumerate}
    This is due to the arguments of the Remark \ref{American_put_remark}. Since the proofs are identical up to the corresponding change of notation, they are left to the reader.
\end{remark}

\begin{Prop}\label{prop_boundary_monotonicity}
    Fix time horizon $T \in [0, \infty)$ and assume \ref{bdy_assm}, \ref{bdy_assm_3}. Then, the function 
    \begin{equation}\label{bijection_between_stopping_boundaries}
        \widecheck{b}_T(t) \coloneqq \log \left( \frac{c_1(t)}{c_0(t)}\frac{b_T(t)}{(1-b_T(t))} \right), \quad 0 \le t \le T,
    \end{equation}
    is non-increasing.
\end{Prop}

\begin{proof}
    The proof relies on the equivalent reformulation of the original stopping problem in terms of American options as in Subsection \ref{subsec_american_options_formulation}. Using the notation there, we recall that the value function $V_T(\cdot, \cdot)$ can be represented as
    $V_T(t, \pi) = (1 - \pi) \, c_0(t) \, u_T(t, y(t, \pi))$, where $u_T(\cdot, \cdot)$ is the value function of the associated American call option stopping problem in \eqref{value_function_sup_american_option}--\eqref{american_call_value_fn}. In particular, it is straightforward to verify that the optimal stopping time of the American call is the first hitting time of the boundary $\widecheck{b}_T: [0, T] \to \mathbb{R}$ of \eqref{bijection_between_stopping_boundaries}. As a result, in order to obtain the statements of Proposition \ref{prop_boundary_monotonicity}, it suffices to show that the corresponding American option stopping problem admits a non-increasing stopping boundary.
    
    Recall that the value function $u_T(\cdot, \cdot)$ can be defined as
    \begin{equation}\label{american_option_value_function}
            u_T(t, y) = \sup_{\tau \in \mathcal{T}_{T - t}} \ex \left[
            e^{\int_t^{t + \tau} \beta_0(s) \, ds} \, \psi\big(Y^{t, y}(t + \tau)\big)
        \right]
    \end{equation}
    for the payoff function $\psi(y) = (e^y-1)^+$ of \eqref{change_of_initial_position} and the process
    \begin{equation}\label{price_dynamics_monotonicity_proof}
        Y^{t,y}(s) \coloneqq y + \int_t^s \left(\beta_1(v) - \beta_0(v) - \frac{a^2}{2} \right) dv + \int_t^s a \, d\widetilde{B}(v), \quad t \le s \le T,
    \end{equation}
    with $\widetilde{B} \coloneqq \{\widetilde{B}(t), 0 \le t \le T\}$ a standard Brownian motion. Moreover, note that the gain function for the American option is time-homogeneous. Therefore, to show that the boundary is monotone, it suffices to show that, under the assumptions of Proposition \ref{prop_boundary_monotonicity}, the function $u_T(\cdot, y)$ is non-increasing for any $y \in \mathbb{R}$. This, in turn, follows from the following observations.

    First, if $t$ increases, the supremum in \eqref{american_option_value_function} is taken over a smaller set, which can only decrease the value. Second, the assumptions on monotonicity of the functions $\beta_0(\cdot)$ and $\beta_1(\cdot) - \beta_0(\cdot)$ imply that if the time parameter $t$ increases, both the interest rate and the drift of the corresponding price process under the expectation of \eqref{american_option_value_function} decrease; while the latter, in turn, implies that the corresponding price process started at a later time will almost surely be below the earlier one, by the comparison principle of Proposition 5.2.18 in \cite{BMSC}. Consequently, the value of the problem can also only decrease, which implies the non-increasing property of the boundary $\widecheck{b}_T(\cdot)$.
\end{proof}

\begin{Prop}\label{prop_boundary_continuity}
    Fix the time horizon $T \in [0, \infty)$ and assume \ref{bdy_assm}--\ref{bdy_assm_3}. Then the boundary $b_T(\cdot)$ is continuous on $[0, T)$ with $\lim\limits_{t \to T} b_T(t) = c_0(T) / (c_0(T) + c_1(T))$.
\end{Prop}

\begin{proof}
    The idea of the proof is taken from \cite{EksVai15} and can be formulated heuristically as follows. If the boundary of the optimal stopping problem is monotone, then the so-called \emph{smooth-fit} principle holds; that is, the value function is then of class $C^1$ in the spatial argument. If, in addition, the second spatial derivative of the value function is appropriately bounded from below in the continuation region, then the boundary must be continuous.
    
    To obtain this, we place ourselves again in the framework for American options and proceed as follows.
    We fix $T \in [0, \infty)$ and define the functions $w_T: [0, T] \times (0, \infty) \to \mathbb{R}_+$, $\varphi:\nobreak (0, \infty) \to \mathbb{R}_+$, and $p_T: [0, T] \to \mathbb{R}_+$ by
    \begin{equation}\label{spatial_transformation}
        w_T(t, z) \coloneqq u_T(t, \log(z)), \quad 
        \varphi(z)\coloneqq \psi(\log(z)), \quad \text{ and } \quad 
        p_T(t) \coloneqq \exp\left(\widecheck{b}_T(t)\right),
    \end{equation}
    where $\widecheck{b}_T(\cdot)$, $u_T(\cdot, \cdot)$, and $\psi(\cdot)$ are defined in \eqref{bijection_between_stopping_boundaries}, \eqref{american_option_value_function}, and \eqref{price_dynamics_monotonicity_proof}, respectively. 
    This spatial transformation writes the corresponding call option in terms of the exponential coordinates. 
    Here, $p_T(\cdot)$ represents the transformed stopping boundary $\widecheck{b}_T(\cdot)$, the exercise boundary of the American call with value function $u_T(\cdot, \cdot)$, as shown in the proof of Proposition \ref{prop_boundary_monotonicity}.
    The approach of the proof will be to show that the new boundary $p_T(\cdot)$ is continuous on $[0, T)$ with $\lim_{t \to T} p_T(t) = 1$, from which the corresponding statements for the boundary $b_T(\cdot)$ will follow by inverting \eqref{bijection_between_stopping_boundaries}. To simplify notation, we write 
    $w \coloneqq w_T, p \coloneqq p_T$, and $\widecheck{b} \coloneqq \widecheck{b}_T$.
    
    On the strength of Proposition \ref{prop_boundary_monotonicity}, the boundary $p(\cdot)$ is necessarily non-increasing as a composition of $\widecheck{b}(\cdot)$ with $\exp(\cdot)$. Together with its lower semi-continuity, which is a consequence of Remark \ref{remark_lsc}, we deduce that $p(\cdot)$ is right-continuous.
    To show the left-continuity of $p(\cdot)$, we first obtain a lower bound on the partial derivative $\partial_{zz} w(\cdot, \cdot)$ in the corresponding continuation region $\widecheck{\mathcal{C}} \coloneqq\{(t, z) \in [0, T] \times \mathbb{R}_+ : z < p(t)\}$ (cf. \eqref{w_zz_final_bound} below),
    and then argue by contradiction.
    
    To obtain a preliminary estimate on $\partial_{zz} w(\cdot, \cdot)$, we proceed as follows. First, we note that the function $w(\cdot, \cdot)$ satisfies the following PDE in the continuation region $\widecheck{\mathcal{C}}$:
    \begin{equation}\label{pde_for_w}
        \partial_t w(t, z) + \big(\beta_1(t) -\beta_0(t)\big) z \,\partial_z w(t, z) + \frac{a^2z^2}{2} \partial_{zz} w(t, z) =  -\beta_0(t) \, w(t, z).
    \end{equation}
    This is an immediate consequence of the relationships \eqref{american_call_value_fn}, \eqref{spatial_transformation}, and the PDE for the value function $V_T(\cdot, \cdot)$ in \eqref{value_function_boundary_problem}. 
    Next, we observe that $\partial_t w(t, z) \leq 0$ and $\partial_z w(t, z) \geq 0$ for any $(t, z) \in \widecheck{\mathcal{C}}$. The former follows from the definition of $w(\cdot, \cdot)$ in \eqref{spatial_transformation} and a similar statement for the function $u_T(\cdot, \cdot)$ of \eqref{american_option_value_function}. The latter follows from the probabilistic representation of $w(\cdot, \cdot)$ via the expectation, in the manner of \eqref{american_option_value_function}, and a comparison principle (Proposition 5.2.18 of \cite{BMSC}) for the processes of \eqref{price_dynamics_monotonicity_proof}.
    As a result, by rearranging terms in \eqref{pde_for_w}, noting that $z \in (0, \infty)$, and recalling the assumption $\beta_1(\cdot) - \beta_0(\cdot) \le 0$ from \ref{reg_assm_5}, we obtain the following intermediate bound for any $(t, z) \in \widecheck{\mathcal{C}}$:
    \begin{equation}\label{w_zz_bound_preliminary}
        \frac{a^2 z^2}{2} \partial_{zz} w(t, z) 
        \geq
        -\beta_0(t) \, w(t, z) - \big(\beta_1(t) -\beta_0(t)\big) z\,  \partial_z w(t, z) 
        \ge
        -\beta_0(t) \, w(t, z) \ge -\beta_0(t)(z-1).
    \end{equation}
    Here, the last inequality follows from the fact that the value function $w(t, \cdot)$ from \eqref{spatial_transformation}, \eqref{american_option_value_function} always dominates the gain function $\varphi(\cdot) = (\cdot - 1)^+ \ge (\cdot - 1)$.

    To obtain the final bound, it remains to note that $z < p(0) < \infty$ holds for any $(t, z) \in \widecheck{\mathcal{C}}$. The first inequality is a consequence of the monotonicity of $p(\cdot)$, while the second follows from a comparison with the stopping boundary for the infinite-horizon problem in this setting. To be precise, one can consider an infinite-horizon American call option, similar to \eqref{american_option_value_function}--\eqref{price_dynamics_monotonicity_proof}, but with the constant discount rate $\beta_0(0)$ and constant drift $\beta_1(0) - \beta_0(0) - a^2/2$. By arguing similarly to the proof of Proposition \ref{prop_boundary_monotonicity}, it becomes clear that the value function for the infinite-horizon problem dominates $w(\cdot, \cdot)$, so the boundary $p(\cdot)$ lies below its infinite-horizon counterpart. The latter is well-known to be a real constant 
    (see, e.g., Chapter 2.6-2.7 in \cite{KarShr98}). 

    Combining the above observations with the bound \eqref{w_zz_bound_preliminary}, we obtain our final estimate
    \begin{equation}\label{w_zz_final_bound}
        \partial_{zz} w(t, z) \ge  -\frac{2\beta_0(t)}{a^2 \,p(0)^2}(z-1), \quad \forall \, (t, z) \in \widecheck{\mathcal{C}}.
    \end{equation}
    Before we proceed, it is also important to note that $p(t) \ge 1$ for any $t \in [0, T]$, since it is never optimal to exercise the option below the strike. Formally speaking, this follows from the positivity of $w(\cdot, \cdot)$, which in turn follows from the positivity of $V_T(\cdot, \cdot)$ proved in Proposition \ref{prop_structure_of_regions}.
    
    We now use the estimate \eqref{w_zz_final_bound} to show the left-continuity of the boundary $p(\cdot)$. We argue by contradiction. 
    Suppose that there exists $t_0 \in(0,T)$ such that $p(\cdot)$ is not left-continuous at $t_0$.
    Fix any $s \in [0, t_0)$ and note that the monotonicity of $p(\cdot)$ implies $p(s) \ge p(t_0-) > p(t_0) \ge 1$.
    Let $z = \big(p(t_0-) + p(t_0)\big)/2 \in (p(t_0), p(t_0-))$ and note that the point $(s, z)$ lies the continuation region $\widecheck{\mathcal{C}}$ above the point $(s, p(t_0))$.  
    By \eqref{spatial_transformation}, \eqref{american_call_value_fn}, and Proposition \ref{prop_C1_value_func}, we know that the function $w(t, \cdot)$ is of class $C^1$ for any $t \in [0, T)$. 
    From this and the fact that $w(t, \cdot) = \varphi(\cdot)$ in $\widecheck{\mathcal{C}}^c$, we deduce
    \begin{equation*}
        \begin{split}
            (w-\varphi)(s,z) 
            &= 
            (w-\varphi)(s,p(s)) - \int_{z}^{p(s)}\partial_z(w-\varphi)(s,h) \, dh
            \\&=
            -\int_{z}^{p(s)}\left(\partial_z (w-\varphi)(s,p(s))
            -\int_{h}^{p(s)}\partial_{zz} (w-\varphi)(s, u) \, du \right) dh
            \\&=
            \int_{z}^{p(s)} \int_{h}^{p(s)} \partial_{zz}(w-\varphi)(s, u) \, du \, dh
            =\int_{z}^{p(s)} \int_{h}^{p(s)} \partial_{zz} w(s, u) \, du \, dh
            \\&\ge
            -\frac{2\beta_0(s)}{a^2 \, p(0)^2} \int_{z}^{p(s)} \int_{h}^{p(s)} (u-1) \, du \, dh
            \\&\ge -\frac{2\beta_0(s)}{a^2 \, p(0)^2} \, (z -1) \int_{z}^{p(s)} (p(s) - h) \, dh
            =
            -\frac{\beta_0(s)}{a^2 \, p(0)^2} \, (z -1) \, (p(s) - z)^2.
        \end{split}
    \end{equation*}
    The fourth equality is a consequence of the fact that $\varphi(u) = u - 1$ for all $u \ge z \ge p(t_0) \ge 1$, the first inequality relies on the bound \eqref{w_zz_final_bound}, and the last inequality follows from the fact that $u \ge h \ge z$.
    
    As a result, for times $s < t_0$ we get the bound
    \begin{equation*}
        (w-\varphi)\left(s, z\right) 
        \ge
        -\frac{\beta_0(s)}{a^2 \, p(0)^2} \, (z -1) \left(p(s) - z\right)^2 
        \ge
        -\frac{\beta_0(s)}{a^2 \, p(0)^2}\eps,
    \end{equation*}
    where $\eps := \left(\frac{p(t_0-) + p(t_0)}{2} -1\right)\left(p(t_0-)-\frac{p(t_0-)+p(t_0)}{2}\right)^2>0$ is independent of $t$.
    Sending $s \to t_0$, using the continuity of $w(\cdot, \cdot)$ with the assumption \ref{bdy_assm_2}, and noting that the point $(t_0, z)$ belongs to the stopping region, 
    we arrive at the absurdity $0 = (w-\varphi)\left(t_0, z \right) \ge -\eps \, \beta_0(t_0) \left( a \, p(0)\right)^{-2} > 0$,
    which forces us to conclude that $p(\cdot)$ is continuous on $[0, T)$. 

    To obtain that $\lim_{t \to T} p(t) = 1$, it suffices to repeat the above arguments with $t_0 = T$, formally identifying $p(T) = 1$, and derive a contradiction to $p(T-)>1$. The identification $p(T) = 1$ is needed since the terminal value of $p(\cdot)$, implied by our boundary definition in \eqref{boundary_definition}, is zero.
\end{proof}

\subsection{Examples}\label{subsec_examples}

We conclude our study of boundary behavior by analyzing some particular examples of stopping boundaries via numerical analysis. More precisely, we consider six situations in which the discount functions are given by exponential, linear, or (smoothed) step functions, and illustrate the effect of these functions and reward magnitudes on the resulting boundaries. The latter are obtained numerically using the dynamic programming principle and backward induction.
We summarize our cases in Table \hyperlink{tbl1}{1} in terms of the following specifications.

\begin{table}[htbp]
\caption{Parameters for the examples.}
\begin{center}
\begin{tabular}{|M{1.33cm} | M{0.6cm} | M{0.9cm} | M{0.6cm} | M{4.2cm} | M{2.6cm}|  N}
    \hline
    & \hypertarget{tbl1}{$a$} & $b$  & $k$ & $c(0,t)$ & $c(1,t)$ \\ \hline
    \hypertarget{eg1}{Example 5.1} &  $2$ & $-1$ & $\displaystyle\frac{2}{5}$ & $\exp\left(-k t\right)$  & $\exp\left(- t\right)$ & \rule[2ex]{0pt}{4ex} \\ \hline
    \hypertarget{eg2}{Example 5.2} & $6$ & $-1$ & $\displaystyle\frac{2}{5}$ & $\exp\left(-k t\right)$  & $\exp\left(- t\right)$ & \rule[2ex]{0pt}{4ex} \\ \hline
    \hypertarget{eg3}{Example 5.3} & $2$ & $\displaystyle-\frac{1}{2}$ & $\displaystyle\frac{2}{5}$ & $\exp\left(-kt\right)$  & $\exp\left(-t\right)$ & \rule[2ex]{0pt}{4ex} \\ \hline
    \hypertarget{eg4}{Example 5.4} & $\sqrt{8}$ & $\displaystyle-\frac{1}{\sqrt{2}}$ & $\displaystyle\frac{1}{2}$ & $\displaystyle\frac{t+k}{1+k}$  & $\exp\left(-2t\right)$ & \rule[2ex]{0pt}{4ex} \\ \hline
    \hypertarget{eg5}{Example 5.5} & $4$ & $-1$ & $\displaystyle\frac{49}{100}$ & $\exp(-kt)$  & $\displaystyle 1 - \frac{t}{2}$ & \rule[2ex]{0pt}{4ex} \\ \hline
    \hypertarget{eg6}{Example 5.6} & $2$ & $\displaystyle -\frac{1}{3}$ & $10^3$ & $1-\frac{1}{3} \left(f_{k}^{1/3}\left(t\right) + f_{k}^{2/3}\left(t\right)\right)$  &  $ 1 - \frac{1}{2} f_{k}^{1/2}\left(t\right) - \frac{t}{10} $ & \rule[2ex]{0pt}{4ex} \\ \hline
\end{tabular}
\end{center}
\end{table}

We let $T=1$ be the time horizon of the problem, and recall that $a$ and $b$ are parameters of the reward $(a \theta + b)$. Additionally, we let $k$ be a parameter for the survival functions $c(0,t)$ and $c(1,t)$ of \eqref{def_survival_function} (see \eqref{eqn:orig.prob.embedding} for the relation to $c_0(t)$ and $c_1(t)$), and define $f_{k}^s(t)$ to be $C^\infty$ approximations to the indicator functions $f^s(t) \coloneqq \mathbf{1}(t>s)$; that is,
\begin{equation}\label{step_fn_approx}
    f_k(t,s) \coloneqq \begin{cases}
        \exp\left(-\frac{1}{k(t-s)} \right), & t > s, \\
        0, & \text{otherwise.}
    \end{cases}
\end{equation}
The reader may check that the assumptions \ref{reg_assm_1}--\ref{reg_assm_5} are satisfied in all of the above examples, except Example \hyperlink{eg4}{5.4}. However, it does satisfy the assumptions of Remark \ref{relaxation_of_assumptions_remark}, and thus our framework indeed applies here.
It is also easy to verify that the assumptions \ref{bdy_assm}--\ref{bdy_assm_3} are satisfied in Examples \hyperlink{eg1}{5.1}--\hyperlink{eg3}{5.3} and \hyperlink{eg5}{5.5}. We obtain the following results.

Examples \hyperlink{eg1}{5.1}, \hyperlink{eg2}{5.2}, and \hyperlink{eg3}{5.3} examine the scenario of exponential discounting under identical parameters, showcasing the impact of varying the constants \(a\) and \(b\), as illustrated in Figure \ref{figure1}. Intuitively, reducing \(b\) should result in a lower stopping boundary (on average), and this is confirmed by the comparison between Examples \hyperlink{eg1}{5.1} and \hyperlink{eg3}{5.3}. With a smaller downside for an incorrect decision and a larger upside for a correct one, the observer is more willing to stop with less evidence of being in the ``good’’ regime (\(\theta = 1\)). However, this intuition does not always hold when \(b\) is fixed and \(a\) is increased, as highlighted by the differences between Examples \hyperlink{eg1}{5.1} and \hyperlink{eg2}{5.2}. This reflects the nuanced tradeoff faced by an optimal stopper: extending observation time enhances the likelihood of accurate detection but also incurs a loss in reward due to the decreasing nature of the discount function \(c_1(\cdot)\). Additionally, the decreasing behavior of \(c_0(\cdot)\) in these examples further mitigates the penalty for incorrect decisions made at later times, influencing the balance of this tradeoff. We highlight here that the boundaries in these regimes seem to be concave with a long interval of increase and then a steep decline close to the terminal time $T$.
\setlength{\abovecaptionskip}{-5pt}

\begin{figure}[H]
    \centering\includegraphics[scale=0.3]{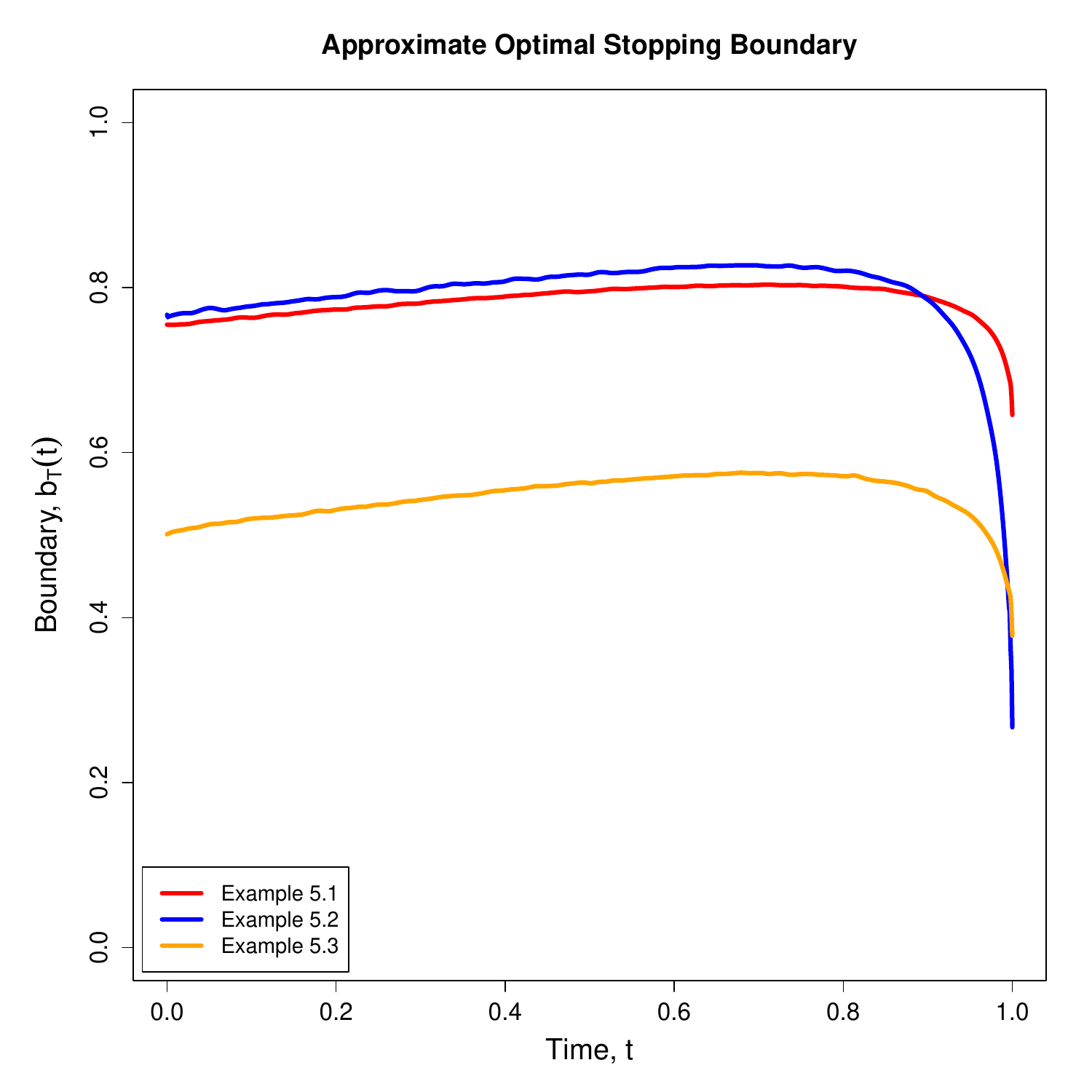}
        \hspace{0.2cm}
        \includegraphics[scale=0.3]{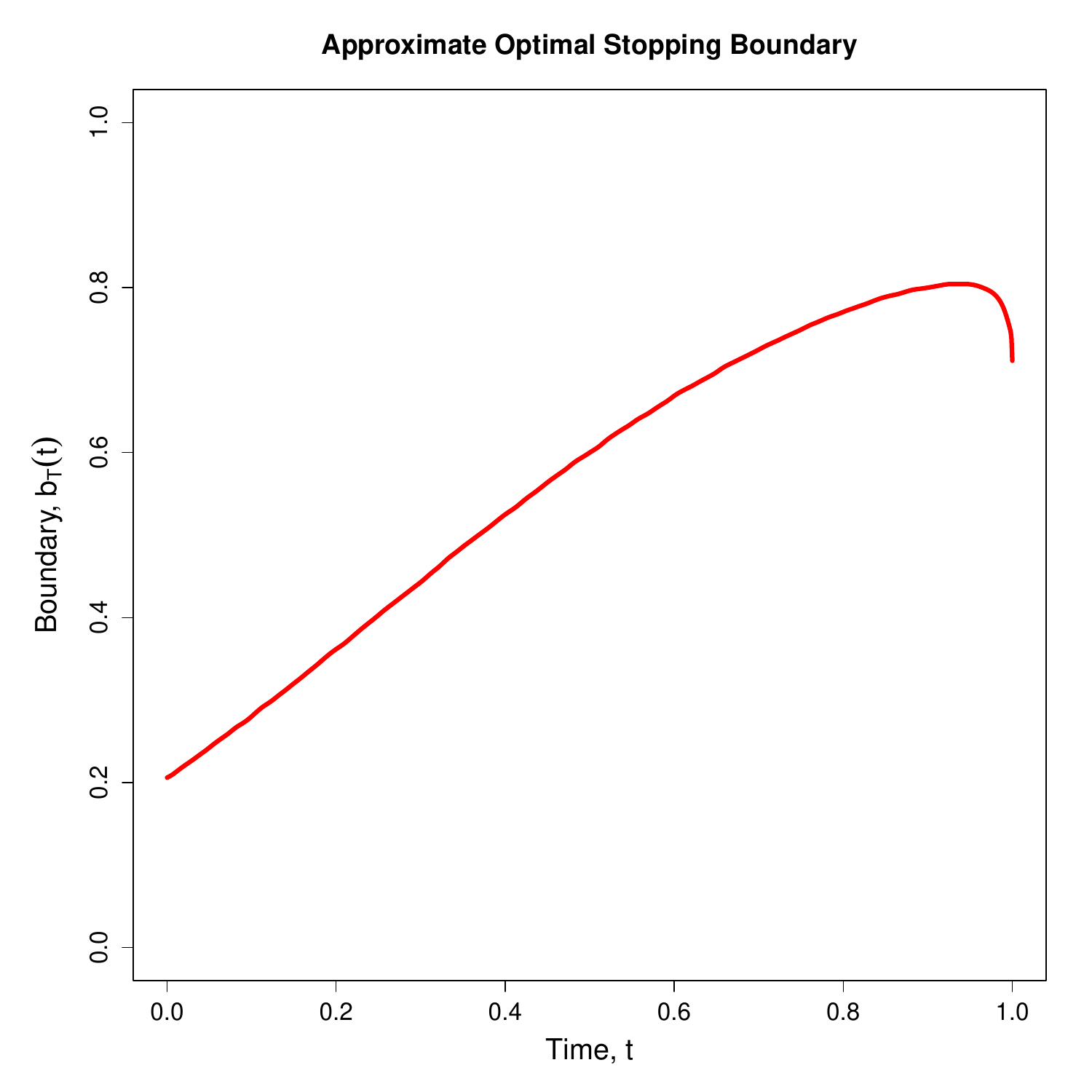}
    \caption{Optimal stopping boundary for the parameters and discount functions of Examples \protect\hyperlink{eg1}{5.1}, \protect\hyperlink{eg2}{5.2}, and \protect\hyperlink{eg3}{5.3}, (left), and \protect\hyperlink{eg4}{5.4}  (right).} \label{figure1}
\end{figure}

Example \hyperlink{eg4}{5.4}, shown in Figure \ref{figure2}, explores the effect of a non-decreasing \(c_0(\cdot)\), which falls under our framework given that $c_1(\cdot)$ and $c_0(\cdot) + c_1(\cdot)$ are both non-increasing (recall again Remark \ref{relaxation_of_assumptions_remark}).
The non-decreasing property of \(c_0(\cdot)\) models a scenario where the penalty for accepting the unknown drift (\(\theta = 0\)) increases with delay. This represents a departure from the ``random time horizon’’ framework, since this \(c_0(\cdot)\) no longer behaves as a survival function. Within the context of the hiring problem mentioned in the introduction, an increasing penalty for hiring a suboptimal candidate may correspond to additional costs accrued from hiring a poor candidate, or from prolonged deliberation before making the decision. We see that the boundary in Example \hyperlink{eg4}{5.4} in some sense resembles the ones from Examples \hyperlink{eg1}{5.1}-- \hyperlink{eg3}{5.3}. However, notably, this time the boundary is not concave.

Two other examples are displayed in Figure \ref{figure2}. In Example \hyperlink{eg5}{5.5} we see the effect of different rates of decay for the functions $c_0(\cdot)$ and $c_1(\cdot)$. In particular, the given choice of exponential decay for $c_0(\cdot)$ (i.e., exponential decay for the penalty received if one guesses that the sign of the drift is positive incorrectly) counteracts the effect of the linear decay for $c_1(\cdot)$ (i.e., linear decay of the reward if one waits to correctly identify the sign of the drift as being positive), and we obtain a non-increasing boundary. This is precisely the opposite effect as in the previous examples.

Finally, in Example \hyperlink{eg6}{5.6}, we create a scenario where the discount functions $c_0(\cdot)$ and $c_1(\cdot)$ have rather rapid rates of decay around certain points and are relatively flat and linear elsewhere, due to the relationship \eqref{step_fn_approx}. The effect of such discount functions is clear and intuitive: rapid decreases in $c_0(\cdot)$ result in rapid decreases in the boundary since our reward-to-penalty ratio jumps appropriately, whereas the opposite is true for rapid decreases in $c_1(\cdot)$. The interval of approximately $(0, T/3)$,  where $b_T(\cdot)$ increases relatively slowly, corresponds to the (approximately) linear decay of $c_1(\cdot)$ and unchanging $c_0(\cdot)$ on this set.

We conclude with the observation that, as $t \uparrow T$, the boundaries approach 
$c_0(T) / \big(c_0(T) +\nobreak c_1(T)\big)$ in all of our examples. This behavior, however, is expected due to the fact that at the terminal time $T$ we have the equality
$$V_T(T,\pi) = \Big(c_1(T) \pi - c_0(T) (1-\pi) \Big)^+.$$
Moreover, in Examples \hyperlink{eg1}{5.1}--\hyperlink{eg3}{5.3} and \hyperlink{eg5}{5.5} this behavior follows from Proposition \ref{prop_boundary_continuity}.

    \begin{figure}[H]
        \centering\includegraphics[scale=0.3]{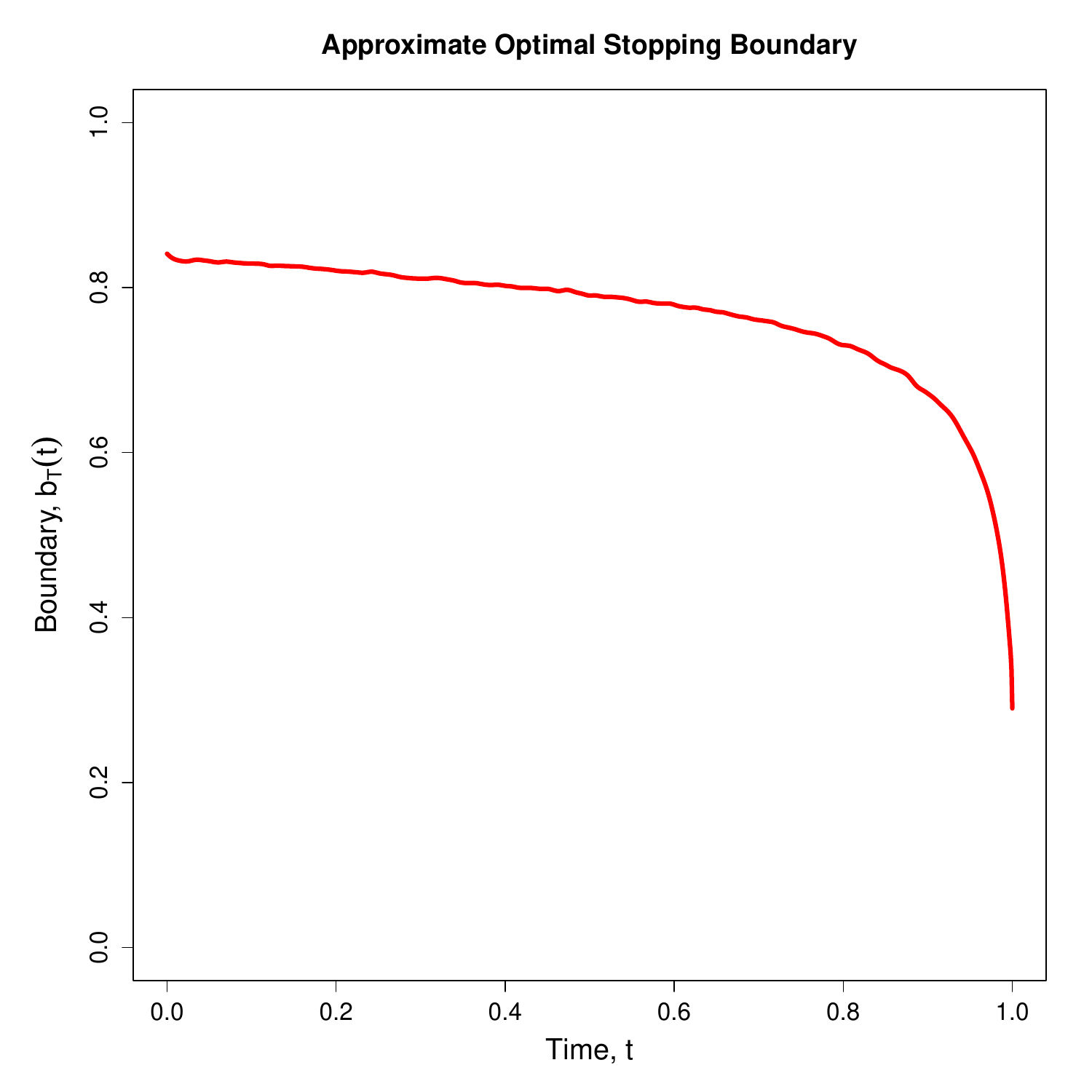}
        \hspace{0.2cm}
        \includegraphics[scale=0.3]{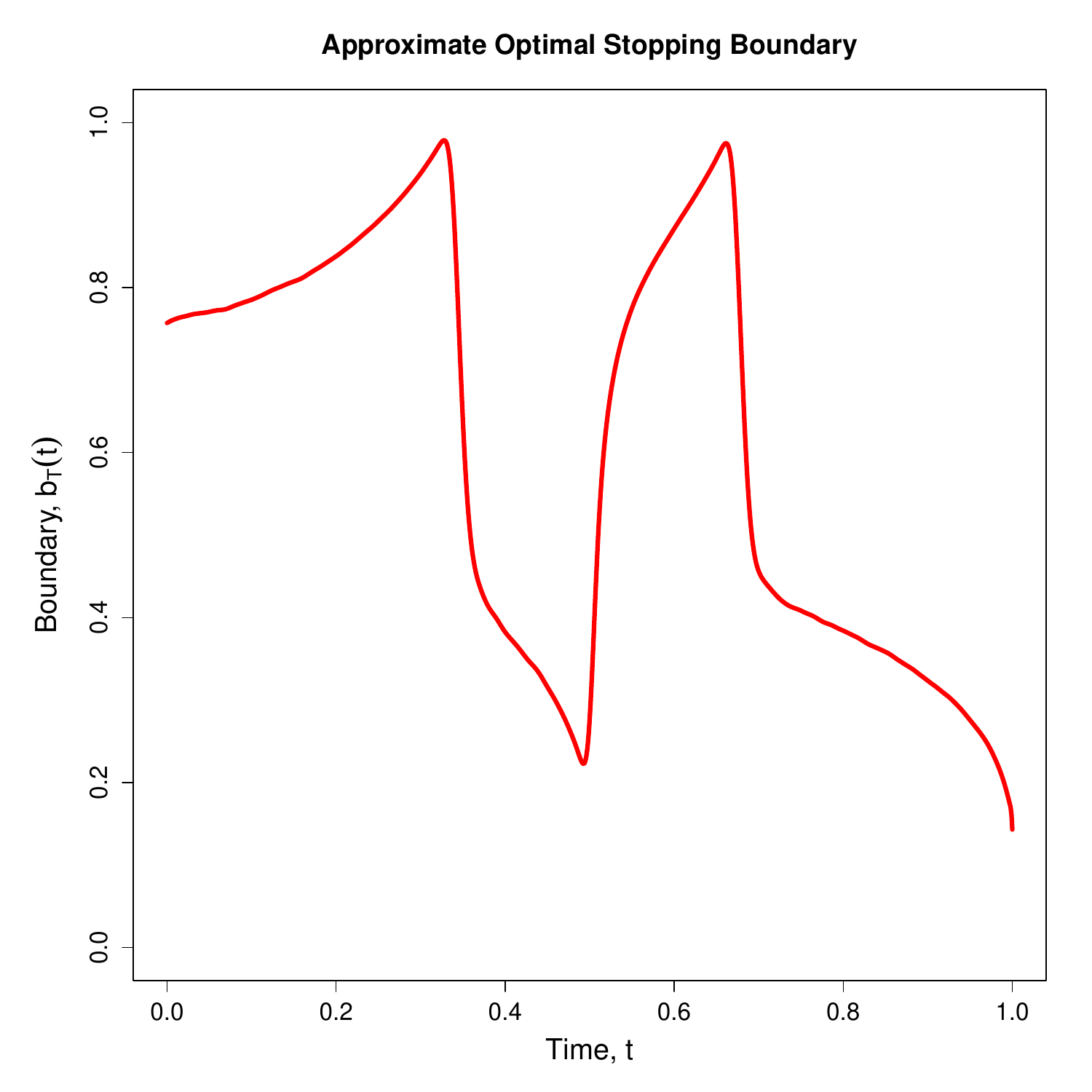}
        \caption{Optimal stopping boundary for the parameters and discount functions of Examples \protect\hyperlink{eg5}{5.5} (left) and \protect\hyperlink{eg6}{5.6} (right).} \label{figure2}
    \end{figure}

\section{Concluding Remarks and Future Research}\label{sec_conclusion}

In this paper, we solve a filtering problem for an unknown reward in the presence of a stochastic deadline. Our framework is quite flexible insofar as it can accommodate a vast set of deadline distributions for both finite and infinite-horizons that, crucially, may \textit{depend on the reward}. Attaining this level of generality required methodological novelty and careful threading of both probabilistic and analytic arguments. Our results suggest a number of natural directions for further research.

\begin{enumerate}
    \item We believe that the continuity of the stopping boundary in our problem holds under much milder assumptions than those of Proposition \ref{prop_boundary_continuity} (and perhaps no additional assumptions beyond \ref{reg_assm_1}--\ref{reg_assm_5}). Thus, the first direction would be to explore boundary regularity; this, however, requires developing novel techniques, since the existing approaches (e.g., \cite{DeAngelis15, DeAngelisStabile, DeAngelisLamberton24}) do not apply directly due to the broad class of discount factors that we consider and the non-differentiability of the gain function. 
    \item Another direction would be to admit additional rewards that can be claimed if stopping \textit{after} the realization of the stochastic deadline is allowed. We expect that many of the techniques in this paper can be applied in that setting.
    \item Our problem formulation is partly motivated by a stopping game proposed by Erik Ekst\"om where two players compete to obtain an unknown and stochastic reward. Games of optimal stopping in the presence of uncertainty have been extensively studied; see several recent examples in \cite{DeangFerMor, GaitsGroen23, SeelStrack16} and references therein. However, the situation changes dramatically if the noise in the problem is common to all players, in which case the only  
    related example known to the authors is \cite{CampZha24MF}. Such problems are incredibly difficult to solve since the intrinsic challenges that arise from stopping time interactions are further compounded by a coupled information structure. In this work, we provide a general solution to the single-player problem which we hope will open the door to solutions in the game setting. 
\end{enumerate}

\noindent
\textbf{Acknowledgments}

We thank Erik Ekström for suggesting to us a problem, still open, which motivated us to work along the present lines. 
We also gratefully acknowledge financial support from the National Science Foundation under grant NSF-DMS-25-06199 and from a Lenfest Award at Columbia University. SC further acknowledges support from a Columbia University CDFT Research Grant and an NSERC Postdoctoral Fellowship (PDF-599675-2025).

\begin{appendices}

\section{Technical Preliminaries}
\subsection{Discusion on Remark \ref{remark_gamma_filtration}}\label{subsec_equality_of_gamma_problems}

Remark \ref{remark_gamma_filtration} states that maximizing \eqref{expected_reward_function_with_decision} over $\tau \in \mathcal{T}_T$ and $\mathcal{F}(\tau)$--measurable decision rules $d$ is equivalent to maximizing  \eqref{expected_reward_function_with_decision} over stopping times $\tau \leq T$ of the filtration $\mathbb{G} \coloneqq \left\{\mathcal{G}(t)\right\}_{t \ge 0}$ and $\mathcal{G}(\tau)$--measurable decision rules $d$, where $\mathcal{G}(t)\coloneqq \overline{\sigma}\big(X(s), \mathbf{1}_{\{\gamma > s\}}, s \leq t\big)$ and $\overline \sigma$ denotes the usual augmentation with respect to $(\Omega, \mathcal{F}, \pr)$. This observation is summarized by the next proposition. Its proof is technical and may be skipped by the reader who is not interested in the finer measure-theoretic details. 
\begin{Prop} \label{equality_of_gamma_problems_prop} 
    For every $T \ge 0$, with $J_T(\cdot, \cdot)$ defined in \eqref{expected_reward_function_with_decision}, we have 
    \begin{equation} \label{equality_of_gamma_problems}
        \sup_{\substack{\tau \in \mathcal{T}_T, \\ d \in \mathcal{F}(\tau)}} J_T(\tau,d) = \sup_{\substack{\tau \in \mathcal{T}^\mathbb{G}_T, \\ d \in \mathcal{G}(\tau)}} J_T(\tau,d),
    \end{equation}
    where $\mathcal{T}_T^{\mathbb{G}}$ denotes the set of all $\mathbb{G}$ stopping times which are almost surely bounded from above by $T$. Additionally, any optimal policy $(\tau,d)$ for the problem on the left hand side of \eqref{equality_of_gamma_problems} is optimal for the problem on the right hand side as well.
\end{Prop}

\begin{proof}
First, we define the filtration $\mathbb{H} \coloneqq \left\{\mathcal{H}(t) \right\}_{t \ge 0}$ with
\begin{equation} \label{barGdefn}
    \mathcal{H}(t) \coloneqq \Big\{A \in \mathcal{F} : A \cap \{\gamma > t\} = A_t \cap \{\gamma > t\} \text{ for some } A_t \in \mathcal{F}(t)\Big\}.
\end{equation}
Note that $\mathcal{F}(t) \subseteq \mathcal{G}(t) \subseteq \olh(t)$ for each $t \ge 0$. The first inclusion is trivial, and the second follows from the fact that (i) ${\sigma}\big(X(s), \mathbf{1}_{\{\gamma > s\}}, s \leq t\big) \subseteq \olh(t)$, and (ii) $\mathbb{H}$ satisfies the usual conditions, i.e., it is right-continuous and complete. (i) is trivial, and we show (ii) using the fact that $\mathbb{F}$ satisfies the usual conditions. Indeed, if $A \in \bigcap_{n \in \mathbb{N}} \olh(t+1/n)$, there exist sets $A_{t+1/n} \in \mathcal{F}(t+1/n)$ such that $A \cap \{\gamma > t + 1/n\} = A_{t+1/n}\cap \{\gamma > t+1/n\}$ for every $n \in \mathbb{N}$. Define $A_t \coloneqq \bigcap_{n=1}^\infty \bigcup_{k=n}^\infty A_{t+1/k}$. Since the sets $\bigcup_{k=n}^\infty A_{t+1/k}$ are decreasing in $n$, $A_t = \bigcap_{n = m}^\infty \bigcup_{k=n}^\infty A_{t+1/k} \in \mathcal{F}(t+1/m)$ for every $m \in \mathbb{N}$ and thus $A_t \in \mathcal{F}(t)$ by the right-continuity of $\mathbb{F}$. Furthermore,
\begin{equation*}
    A_t \cap \{\gamma > t\} = \bigcap_{n=1}^\infty \bigcup_{k=n}^\infty \Big(A_{t+1/k} \cap \{\gamma > t + 1/k\}\Big) = A \cap \{\gamma > t\}.
\end{equation*}
Thus, $A \in \olh(t)$ by the definition of $\olh(t)$ in \eqref{barGdefn}, showing the right-continuity of $\mathbb{H}$. Completeness follows since trivially $\mathcal{F}(t) \subseteq \olh(t)$.

A key observation is that for any $\mathbb{H}$--optional process $g(\cdot)$, there exists an $\mathbb{F}$--optional process $f(\cdot)$ such that
\begin{equation} \label{G_represented_F}
    g(s) = f(s), \quad \forall \, s \in [0, \gamma), \; a.s.
\end{equation}
A similar statement for predictable processes and $s \in [0, \gamma]$ can be found in the lemma on page 378 of Protter \cite{Protter}, the proof of which can be easily adapted to our setup. 
Namely, to obtain \eqref{G_represented_F}, we first let $A \in \olh(t)$, and note by the definition of $\olh(t)$ in \eqref{barGdefn} that there exists $A_t \in \mathcal{F}(t)$ with
\begin{equation*}
    \mathbf{1}_{A \times [t, \infty)} (\omega, s) = \mathbf{1}_{A_t \times [t,\infty)} (\omega, s) \quad \forall \; s \in [0, \gamma), \; a.s.
\end{equation*}
Since $A_t \times [t, \infty)$ is $\mathbb{F}$--optional, and $\{A \times [t, \infty) : A \in \olh(t), t \ge 0\}$ is a $\pi$-system generating the $\mathbb{H}$--optional sigma field, the claim follows by the monotone class theorem.

Now, let $\tau \in \mathcal{T}^\mathbb{G}_T$, $d \in \mathcal{G}(\tau)$ with $d \in \{0,1\}$. From \eqref{G_represented_F}, there exists $\mathbb{F}$-optional $f(\cdot)$ such that
\begin{equation*}
    \mathbf{1}_{\{\tau \leq t\}} = f(t), \quad \forall \; t < \gamma, \; a.s.
\end{equation*}
Define $\sigma \coloneqq \inf\{t \ge 0 : f(t) > 0\} \wedge T$, which is in $\mathcal{T}_T$ of \eqref{collection_stop_times} by the Debut Theorem, since $\mathbb{F}$ satisfies the usual conditions. Note that
    $\{\tau < \gamma\} \Delta\{\sigma < \gamma\}$
is a null set and $\tau = \sigma$ almost surely on the set $\{\tau < \gamma\}$. Furthermore, $d \cdot \mathbf{1}_{\{\tau \leq t \}}$ is $\mathbb{H}$--optional, so again there exists $\mathbb{F}$--optional $h(\cdot)$ such that 
\begin{equation} \label{decisionrule_Fprojection}
    d \cdot \mathbf{1}_{\{\tau \leq t\}} = h(t), \quad \forall \; t < \gamma, \; a.s.
\end{equation}
Define $\widetilde{d} \coloneqq h(\sigma) \mathbf{1}_{\{h(\sigma) \in \{0,1\}\}}$ so that $\widetilde d$ is $\mathcal{F}(\sigma)$--measurable and $\{0,1\}$--valued. We have
\begin{align*}
    \mathbf{1}_{\{\tau < \gamma\}} \cdot \mathbf{1}_{\{d=1\}} &= \mathbf{1}_{\{\sigma < \gamma\}} \cdot d \cdot \mathbf{1}_{\{\tau \leq \sigma\}} = \mathbf{1}_{\{\sigma<\gamma\}} \cdot h(\sigma) \cdot \mathbf{1}_{\{h(\sigma) \in \{0,1\}\}} = \mathbf{1}_{\{\sigma < \gamma\}} \cdot \widetilde d,
\end{align*}
where all equalities hold $\pr$--almost surely. The first equality follows from $\tau = \sigma$ almost surely on the set $\{\tau < \gamma\}$.
The second equality follows from \eqref{decisionrule_Fprojection} and the fact that $h(\cdot)$ is almost surely $\{0, 1\}$-valued on the interval $[0, \gamma)$ from its definition. The third equality follows from the definition of $\widetilde d$.
This establishes \eqref{equality_of_gamma_problems}, concluding the proof. 
\end{proof}

\subsection{The Role of Independence in Conditional Expectations}\label{app_role_of_independence}

Let $(\Omega,\mathcal F,\mathbb P)$ be a probability space, let $(E, \mathcal{B}(E))$ and $(\tilde{E}, \mathcal{B}(\tilde{E}))$ be standard Borel spaces, let $X:\Omega\to E$ be measurable, and let $\mathcal G\subseteq\mathcal F$ be a sub-$\sigma$-algebra. Since $(E, \mathcal{B}(E))$ is a standard Borel space, there exists a regular conditional distribution $\kappa(\omega,dx)$ of $X$ given $\mathcal G$. For every bounded Borel function $f: E \times \tilde{E} \to\mathbb{R}$, define
\[
\Phi_f:\Omega\times \tilde{E}\to\mathbb{R},
\qquad
\Phi_f(\omega,y):=\int_{E} f(x,y)\,\kappa(\omega,dx).
\]
Then $\Phi_f$ is $\mathcal G\otimes\mathcal B(\tilde{E})$-measurable, and for each fixed $y\in\tilde{E}$, the map $\Phi_f(\cdot,y)$ is a version of $\mathbb{E}[f(X,y)\mid\mathcal G](\cdot)$. Accordingly, we interpret the informal notation
\[
\left.\mathbb{E}[f(X,y)\mid\mathcal G]\right|_{y=Y}
\]
as meaning $\Phi_f(\cdot,Y(\cdot))$ for any measurable mapping $Y: \Omega \to \tilde{E}$. The proof of the following result is a straightforward generalization of known arguments; however, as we were unable to locate a precise reference, we provide the details for completeness.

\begin{Th}
Let $(\Omega,\mathcal F,\mathbb P)$ be a probability space, let
$X:\Omega\to E$ and $Y:\Omega\to \tilde{E}$ be measurable, and let
$\mathcal G,\mathcal H\subseteq\mathcal F$ be sub-$\sigma$-algebras. Suppose
that $\mathcal H$ is independent of $\sigma(X,\mathcal G):=\sigma(\sigma(X),\mathcal G)$,
and that $Y$ is $\sigma(\mathcal G,\mathcal H)$-measurable. Then, for every bounded Borel function $f:E \times \tilde{E} \to\mathbb R$,
\begin{equation}\label{eq:main-condexp}
\mathbb{E}\left[f(X,Y)\mid \sigma(\mathcal G,\mathcal H)\right]
=
\Phi_f(\cdot,Y)
=
\left.\mathbb{E}\left[f(X,y)\mid\mathcal G\right]\right|_{y=Y}
\qquad\text{a.s.}
\end{equation}
\end{Th}

\begin{proof}
Let $\mathcal C$ denote the class of bounded Borel functions $f:E\times \tilde{E} \to\mathbb R$ for which \eqref{eq:main-condexp} holds.
Then $\mathcal C$ is a vector space, and it is closed under bounded monotone convergence: if $f_n\uparrow f$ pointwise and
$\sup_n\|f_n\|_\infty<\infty$, then
\[
\mathbb{E}[f_n(X,Y)\mid \sigma(\mathcal G,\mathcal H)]
\uparrow
\mathbb{E}[f(X,Y)\mid \sigma(\mathcal G,\mathcal H)]
\qquad\text{a.s.},
\]
by conditional monotone convergence, while
\[
\Phi_{f_n}(\cdot,Y)\uparrow \Phi_f(\cdot,Y)
\qquad\text{a.s.},
\]
by monotone convergence in the definition of $\Phi_f$. Hence $f\in\mathcal C$ whenever $f_n\in\mathcal C$ is a bounded sequence with $f_n\uparrow f$.

Now finite linear combinations of product functions of the form $f(x,y)=g(x)h(y)$, with $g:E\to \mathbb{R}$ and $h:\tilde{E} \to\mathbb R$ bounded Borel, form an algebra containing the indicators of Borel rectangles $A\times B$, and these rectangles generate $\mathcal B(E)\otimes\mathcal{B}(\tilde{E})$. Therefore, by the functional monotone class theorem, it suffices to prove \eqref{eq:main-condexp} for functions $f(x,y)=g(x)h(y)$ of this type.
For such $f$, by the definition of $\Phi_f$,
\[
\Phi_f(\cdot,Y)
=
h(Y)\int_{E} g(x)\,\kappa(\cdot,dx)
=
h(Y)\,\mathbb{E}[g(X)\mid\mathcal G]
\qquad\text{a.s.}
\]
Thus, it suffices to verify
\begin{equation}\label{eq:product-case-generalized}
\mathbb{E}\left[g(X)h(Y)\mid \sigma(\mathcal G,\mathcal H)\right]
=
h(Y)\,\mathbb{E}[g(X)\mid\mathcal G].
\end{equation}

As $Y$ is $\sigma(\mathcal G,\mathcal H)$-measurable, so is $h(Y)$ and
\[
\mathbb{E}\left[g(X)h(Y)\mid \sigma(\mathcal G,\mathcal H)\right]
=
h(Y)\,\mathbb E[g(X)\mid \sigma(\mathcal G,\mathcal H)].
\]
To obtain \eqref{eq:product-case-generalized} it is therefore enough to show
\begin{equation}\label{eq:independent-enlargement}
\mathbb E[g(X)\mid \sigma(\mathcal G,\mathcal H)]
=
\mathbb E[g(X)\mid\mathcal G]
\qquad\text{a.s.}
\end{equation}
The right hand side is $\mathcal G$-measurable, hence $\sigma(\mathcal G,\mathcal H)$-measurable. It remains to check the defining property of the conditional expectation.

We recall that $\sigma(\mathcal G,\mathcal H)$ is the $\sigma$-algebra generated by intersections $G\cap H$, with $G\in\mathcal G$ and $H\in\mathcal H$. With arbitrary $G\in\mathcal G$ and $H\in\mathcal H$, we have
\begin{equation}\label{eq:pi-system-1}
\mathbb E[g(X)\mathbf 1_G\mathbf 1_H]
=
\mathbb E[g(X)\mathbf 1_G]\,\mathbb P(H).
\end{equation}
because $\mathbf 1_H$ is $\mathcal H$-measurable, while $g(X)\mathbf 1_G$ is $\sigma(X,\mathcal G)$-measurable, and these two $\sigma$-algebras are independent. On the other hand, we likewise have
\begin{equation}\label{eq:pi-system-2}
\mathbb E\left[\mathbb E[g(X)\mid\mathcal G]\mathbf 1_G\mathbf 1_H\right]
=
\mathbb E\left[\mathbb E[g(X)\mid\mathcal G]\mathbf 1_G\right]\,\mathbb P(H).
\end{equation}
By iterated conditioning it follows that $\mathbb E\left[\mathbb E[g(X)\mid\mathcal G]\mathbf 1_G\right]
=
\mathbb E[g(X)\mathbf 1_G]$ so taking \eqref{eq:pi-system-1} and \eqref{eq:pi-system-2} together we conclude
\begin{equation}\label{eq:pi-system-end}
\mathbb E[g(X)\mathbf 1_{G\cap H}]
=
\mathbb E[\mathbb E[g(X)\mid\mathcal G] \mathbf 1_{G\cap H}].
\end{equation}

Now define finite signed measures $\mu$ and $\nu$ on $\sigma(\mathcal G,\mathcal H)$ by
\[
\mu(F)
:=
\mathbb{E}\left[g(X)\mathbf 1_F\right],
\qquad
\nu(F):=
\mathbb{E}\left[\mathbb{E}[g(X)\mid\mathcal G]\,\mathbf 1_F\right].
\]
By \eqref{eq:pi-system-end}, the signed measure $\delta:=\mu-\nu$ vanishes on the collection
\begin{equation}\label{eq:pi-system}
G\cap H,
\qquad
G\in\mathcal G,\ H\in\mathcal H.
\end{equation}
In particular, $\delta(\Omega)=0$. Hence $\Lambda:=\{M\in\sigma(\mathcal G,\mathcal H):\delta(M)=0\}$ is a $\lambda$-system. Since the collection in \eqref{eq:pi-system} is a $\pi$-system generating $\sigma(\mathcal G,\mathcal H)$, Dynkin's $\pi$-$\lambda$ theorem yields $\Lambda=\sigma(\mathcal G,\mathcal H)$. Therefore,
\begin{equation}\label{eq:measures-agree}
\mathbb{E}\left[g(X)\mathbf 1_M\right]
=
\mathbb{E}\left[\mathbb{E}[g(X)\mid\mathcal G]\,\mathbf 1_M\right],
\qquad
\forall\,M\in\sigma(\mathcal G,\mathcal H).
\end{equation}
This proves \eqref{eq:independent-enlargement} and hence also \eqref{eq:product-case-generalized}. Thus, \eqref{eq:main-condexp} holds for product functions $f(x,y)=g(x)h(y)$ and therefore, by the monotone class argument at the start of the proof, also for every bounded Borel function $f:E\times \tilde{E} \to\mathbb R$.
\end{proof}

\section{Proofs of Technical Results}\label{sec_appendix}

We make use of the following simple lemma several times throughout the paper.
\begin{Lem} \label{lem_value_incr_monotonically_T_upward}
    Fix $T \in [0,\infty]$, $(t,\pi) \in [0,T)\times[0,1]$. For any increasing sequence $\{T_n\}_{n \in \mathbb{N}}$ such that $T_n \uparrow T$, the sequence $\{V_{T_n}(t, \pi)\}_{n \in \mathbb{N}}$ is increasing and we have
    \begin{equation*}
        \lim_{n \rightarrow \infty} V_{T_n}(t, \pi) = V_T(t,\pi).
    \end{equation*}
\end{Lem}
\begin{proof}
    Assume that $T_n > t$ for every $n$ without loss of generality. 
    Note that for any $S \in [0, T]$, the stopping problems \eqref{value_function_general_osp} with values $V_S(t, \pi)$ have the same gain functions and the same dynamics for the underlying processes, but the optimization sets of stopping times become larger as $S$ increases. 
    Thus, we obtain 
    $V_T(t, \pi) \ge \lim_{S \to T} V_S(t, \pi), \, \pi \in [0, 1]$, as well as the claimed monotonicity.
    To obtain the reverse inequality, fix any stopping time $\tau \leq T-t$, and,
    for any fixed $t \le S \leq T$, consider the stopping time $\tau_S \coloneqq \tau \wedge (S - t)$. It is clear that $\lim_{S \to T} \tau_S = \tau$ almost surely.
    Therefore, we obtain
    \begin{equation*}
        \lim_{S \to T} V_S(t, \pi) 
        \ge \lim_{S \to T} \ex\left[ G\big(t + \tau_S, \Pi^\pi(t + \tau_S)\big)\right]
        = \ex\left[ G\big(t + \tau, \Pi^\pi(t + \tau)\big)\right].
    \end{equation*}
    Here, the equality follows from the bounded convergence theorem and the martingale convergence theorem for the process $\Pi^\pi(\cdot)$. Since $\tau \leq T-t$ was arbitrary, the claim follows.
\end{proof}

\subsection{Proof of Proposition \ref{prop_gain_functions_equivalnce}}\label{subsec_app_gain_equivalence}

    The result is trivially true for $\pi \in \{0,1\}$ so we will assume $\pi \in (0,1)$. First, we note that for all $(t, \pi) \in [0, T] \times (0, 1)$ we have $\widetilde{G}_T(t, \pi) \le G(t, \pi)$, which implies $\widetilde{V}_T(t, \pi) \le V_T(t, \pi)$. 
    Therefore, to obtain the equality \eqref{equality_value_functions}, it suffices to show that, for every $(t, \pi) \in [0, T) \times (0, 1)$ and $\eps > 0$, there exists a sequence of stopping times $\left\{\tau_\eps^{(n)}\right\}_{n=1}^\infty$ such that 
    \begin{equation}\label{j_tilde_v_plain_inequality}
        \liminf_{n \to \infty} \widetilde{J}_T\left(t, \pi, \tau_\eps^{(n)}\right) \ge V_T(t, \pi)) - \eps.
    \end{equation}
    To this effect, fix any $(t, \pi) \in [0, T) \times (0, 1)$ and $\eps > 0$, and let the stopping time $\tau^*_\eps$ be $\eps$--optimal for the problem $V_T(t, \pi)$, i.e., $\tau^*_\eps$ satisfy $J_T(t, \pi, \tau^*_\eps) \ge V_T(t, \pi) - \eps$. 
    Then, consider stopping times 
    \begin{equation*}
        \tau_\eps^{(n)} = 
        \begin{cases}
            \tau^*_\eps, & \tau^*_\eps \le t_n \, \text{ and } \, G\left(t + \tau^*_\eps, \Pi^\pi\left(\tau^*_\eps\right)\right)> 0,
            \\
            t_n, & \tau^*_\eps > t_n \, \text{ and } \, G\left(t + t_n, \Pi^\pi(t_n)\right) > 0,
            \\
            T-t, &\text{otherwise},
        \end{cases}
    \end{equation*}
    for $n \in \mathbb{N}$, with $t_n \coloneqq T-t-1/n$ in the case $T < \infty$, and $t_n = n$ in the case $T = \infty$. 
    Due to the choice of $\tau^*_\eps$, in order to obtain \eqref{j_tilde_v_plain_inequality}, it suffices to show that 
    \begin{equation*}
        \liminf_{n \to \infty} \widetilde{J}_T\left(t, \pi, \tau_\eps^{(n)}\right) \ge J_T(t, \pi, \tau^*_\eps),
    \end{equation*}
    and this follows from the following chain of inequalities:
    \begin{equation*}
        \begin{split}
            \liminf_{n \to \infty} \ &\widetilde{J}_T(t, \pi, \tau_\eps^{(n)})
            - J_T(t, \pi, \tau^*_\eps) =
            \liminf_{n \to \infty} 
                \ex\left[ 
                    \widetilde{G}_T\left(
                        t + \tau_\eps^{(n)}, \Pi^\pi\left(\tau_\eps^{(n)}\right)
                    \right)
                \right] 
                -
                \ex\Big[
                    G\left(
                        t + \tau^*_\eps, \Pi^\pi(\tau^*_\eps)
                    \right) 
                \Big] 
            \\&\ge
            \liminf_{n \to \infty}
                \ex\left[ 
                    \Big(
                        G\left(
                            t + t_n, \Pi^\pi(t_n)
                        \right)  
                        -
                        G\left(
                            t + \tau^*_\eps, \Pi^\pi(\tau^*_\eps)
                        \right) 
                    \Big)
                    \cdot
                    \mathbf{1}_{
                        \left\{
                        \tau^*_\eps \in (t_n, T-t]\right\}
                    }
                \right]
            \\&\ge
                \liminf_{n \to \infty} \ex\left[\inf_{r, s \in [t_n, T-t]}
                    \Big(
                        G\left(
                            t + r, \Pi^\pi(r)
                        \right)  
                        -
                        G\left(
                            t + s, \Pi^\pi(s)
                        \right) 
                    \Big)
                \right]
            = 0.
        \end{split}
    \end{equation*}
    Here, the first inequality follows from the choice of $\tau_\eps^{(n)}$ and the fact that $G(t+\tau^*_\eps, \Pi^\pi(\tau^*_\eps)) \equiv 0$ on the set $\{\tau^*_\eps \leq t_n, \widetilde{G}_T(t+\tau^*_\eps, \Pi^\pi(\tau^*_\eps)) \leq 0\}$, and the last equality follows from the bounded convergence theorem and the continuity of the process $G(t+\cdot, \Pi^\pi(\cdot))$. The applicability of the bounded convergence theorem is clear in the case $T < \infty$, whereas for $T = \infty$ it follows from the continuity of the functions $c_i(\cdot), \, i = 0, 1$ at infinity due to \ref{reg_assm_2}, and the fact that $\Pi^\pi(\infty)$ is well-defined.
    
    In the case $T < \infty$ the applicability of the .
    \hfill \qed {\parfillskip0pt\par}

\subsection{Proof of Proposition \ref{prop_structure_of_regions}}\label{subsec_app_structure_of_regions}
    We have clearly $V_T(T, \cdot) = G(T, \cdot)$, thus also \eqref{regions_via_boundary} by letting $b_T(T) = 0$. 
    For $t \in [0, T)$ it suffices to show that, if at some space-time point $(t, \pi)$ it is better to stop, i.e., $V_T(t, \pi) = G(t, \pi)$, then $V_T(t, p) = G(t, p)$ for any $p \ge \pi$ as well. Since $V_T(\cdot,\cdot)$ is continuous by Proposition \ref{prop_properties_of_value_func}, which is proved independently, and $[0,T] \times \{0,1\} \subseteq \mathcal{S}_T$ holds trivially, the representation \eqref{regions_via_boundary} will hold on $[0, T)$ with the function $b_T: [0, T) \to [0, 1]$, defined by
    \begin{equation}\label{boundary_definition_2}
        b_T(t) \coloneqq \inf \{\pi \in (0, 1): V_T(t, \pi) = G(t, \pi) \},
    \end{equation}
    with the convention $\inf \emptyset = 1$ since the state space of our problem is $[0, 1]$ (note that the same representation \eqref{boundary_definition_2} will then also hold on the whole interval $[0, T]$).
    
    To obtain the above statement for the functions $V_T(\cdot, \cdot)$ and $G(\cdot, \cdot)$, we obtain it first for the functions $\widetilde{V}_T(\cdot, \cdot)$ and $\widetilde{G}_T(\cdot, \cdot)$, defined in \eqref{gain_function_with_indicator}, \eqref{value_function_with_indicator_gain}, and then show
    \begin{equation}\label{stopping_sets_equality}
    \begin{split}
        \{(t, \pi) \in [0, T) &\times (0, 1] : V_T(t, \pi)= G(t, \pi)\} \\&= \{(t, \pi) \in [0, T) \times (0, 1] : \widetilde{V}_T(t, \pi) = \widetilde{G}_T(t, \pi)\}.
        \end{split}
    \end{equation}
    
    Now, suppose that for some point $(t, \pi)$ we have $\widetilde{V}_T(t, \pi) = \widetilde{G}_T(t, \pi)$, or equivalently 
    \begin{equation}\label{inequality_for_stopping_points}
        \ex\left[\widetilde{G}_T(t + \tau, \Pi^\pi(\tau))\right] \le \widetilde{G}_T(t, \pi), \quad \forall \, \tau \in \mathcal{T}_{T-t}.
    \end{equation}
    Note that the gain function of \eqref{gain_function_with_indicator} is $\widetilde{G}_T(t, \pi) 
        =
        \pi \cdot (c_1(t)+c_0(t)) \cdot \mathbf{1}_{\{t < T\}} 
        -
        c_0(t) \cdot \mathbf{1}_{\{t < T\}}$
    is linear in the spatial variable. Consequently, for any $p > \pi$, the point $(t, p)$ will also satisfy \eqref{inequality_for_stopping_points} for any $\tau \in \mathcal{T}_{T-t}$, and thus belong to the stopping region, by virtue of the following chain of inequalities:
    \begin{align*}
            \widetilde{G}_T&(t, p) 
            =
            p \cdot (c_1(t)+c_0(t)) \cdot \mathbf{1}_{\{t < T\}} 
            - 
            c_0(t)\cdot \mathbf{1}_{\{t < T\}}
            \\&= 
            \pi \cdot (c_1(t)+c_0(t)) \cdot \mathbf{1}_{\{t < T\}}
            -
            c_0(t)\cdot \mathbf{1}_{\{t < T\}}
            + 
            (p - \pi) \cdot (c_1(t)+c_0(t)) \cdot \mathbf{1}_{\{t < T\}}
            \\&\ge
            \ex\left[\Pi^\pi(\tau) \cdot (c_1(t+\tau)+c_0(t+\tau)) \cdot \mathbf{1}_{\{t+\tau < T\}}- c_0(t+\tau)\cdot \mathbf{1}_{\{t+\tau < T\}}\right]
            \\& \quad \quad
            + 
            (p - \pi) \cdot (c_1(t)+c_0(t)) \cdot \mathbf{1}_{\{t < T\}}
            \\&=
            \ex\left[\Pi^p(\tau) \cdot (c_1(t+\tau)+c_0(t+\tau)) \cdot \mathbf{1}_{\{t+\tau < T\}} - c_0(t+\tau)\cdot \mathbf{1}_{\{t+\tau < T\}}\right] 
            \\&\quad \quad
            + 
            \ex\big[(\Pi^\pi(\tau) - \Pi^p(\tau))
             \left. \cdot (c_1(t+\tau)+c_0(t+\tau)) \cdot \mathbf{1}_{\{t+\tau < T\}}\right]
            +
            (p - \pi) \cdot (c_1(t)+c_0(t)) \cdot \mathbf{1}_{\{t < T\}}
            \\&\ge
            \ex\left[\Pi^p(\tau) \cdot (c_1(t+\tau)+c_0(t+\tau)) \cdot \mathbf{1}_{\{t+\tau < T\}} - c_0(t+\tau)\cdot \mathbf{1}_{\{t+\tau < T\}}\right] 
            =
            \ex\left[\widetilde{G}_T(t+\tau, \Pi^p(\tau)\right].
    \end{align*}
    The first inequality follows from the choice of $(t, \pi)$ and \eqref{inequality_for_stopping_points}, and the second from the optional sampling theorem for the process $(\Pi^p(s) - \Pi^\pi(s)) \cdot (c_1(t+s)+c_0(t+s)) \cdot \mathbf{1}_{\{t+s < T\}}, \, 0 \le s \le T-t$, which is a supermartingale. Indeed, the function $(c_1(\cdot)+c_0(\cdot)) \cdot \mathbf{1}_{\{\cdot < T\}}$ is non-increasing due to \ref{reg_assm_3}, while the process $\Pi^p(\cdot) - \Pi^\pi(\cdot)$ is a non-negative martingale by the martingale properties of the processes $\Pi^\pi(\cdot), \, \Pi^p(\cdot)$ and a well-known comparison result (Proposition 5.2.18 of \cite{BMSC}).
    
    It remains to prove \eqref{stopping_sets_equality}. However, this follows from the observations below:
    \begin{enumerate}[topsep=-5pt,itemsep=-1ex,partopsep=1ex,parsep=1ex]
        \item For all $(t, \pi) \in [0, T) \times (0, 1]$ we have  $V_T(t, \pi) = \widetilde{V}_T(t, \pi)$;
        \item For all $(t, \pi) \in [0, T) \times (0, 1]$ with either $G(t, \pi) > 0$ or $\widetilde{G}_T(t, \pi) > 0$, we have $G(t, \pi) = \widetilde{G}_T(t, \pi)$;
        \item For all $(t, \pi) \in [0, T) \times (0, 1]$ we have $\widetilde{V}_T(t, \pi) > 0$, and consequently $V_T(t, \pi) > 0$ as well.
    \end{enumerate}
    Here, the first statement follows from Proposition \ref{prop_gain_functions_equivalnce}, the second from the definitions of the gain functions \eqref{gain_function} and \eqref{gain_function_with_indicator}, and the third from the fact that, for any $(t, \pi) \in [0, T) \times (0, 1]$, the expected payoff $\widetilde{J}(t, \pi, \tau_{1-\eps})$, with $\tau_{1-\eps} \coloneqq \inf\{s \ge 0: \Pi^\pi(s) \ge 1 - \eps \}$, is strictly positive for small enough $\eps$, since $\pr(\tau_{1-\eps} < T-t) > 0$ for every $\eps > 0$.
    \hfill \qed {\parfillskip0pt\par}

\subsection{Proof of Proposition \ref{prop_properties_of_value_func}}\label{subsec_app_prop_properties_value_func}
    We will first prove the following lemma regarding the continuity of the value function $V_T(\cdot, \cdot)$ in the time argument for any fixed $\pi \in [0,1]$.
    \begin{Lem} \label{lem_time_continuity_fixed_pi}
        Fix any $T \in [0,\infty]$. For any $\pi \in [0,1]$, the mapping $t \mapsto V_T(t,\pi)$ is uniformly continuous, and this continuity is uniform with respect to $\pi$.
    \end{Lem}
    \begin{proof}
    For the case $T=\infty$, we note that $[0,\infty]$ is metrizable (see \cite{mandelkern}, for example). We will denote the associated metric by $d(t,s)$, $s,t \in [0,\infty]$, and we refer to uniform convergence with respect to this metric, under which $[0, \infty]$ is compact. We consider the two cases $T = \infty$ and $T < \infty$ separately.
    
    \noindent \underline{Case \#1, $T = \infty$:} Fix $\pi \in [0, 1]$, $\eps > 0$, and consider any $t_1, t_2 \in [0, \infty]$. Let $\tau_\eps \in \mathcal{T}_\infty$ be an $\eps$-optimal stopping time that satisfies $V_\infty(t_1, \pi) \leq J_\infty(t_1, \pi, \tau_\eps) + \eps$,
    where the function ${J}_\infty(\cdot, \cdot, \cdot)$ is defined in \eqref{expected_reward_general_pi}. We obtain
    \begin{equation*}
        \begin{split}
            {V}_\infty(t_1, \pi) - {V}_\infty(t_2, \pi) 
            &\le
            {J}_\infty(t_1, \pi, \tau_\eps) + \eps - {J}_\infty(t_2, \pi, \tau_\eps) 
            \\&= 
            \eps + \ex \big[ {G}\left(t_1 + \tau_\eps, \Pi^{\pi}(\tau_\eps) \right) - {G}\left(t_2 + \tau_\eps, \Pi^{\pi}(\tau_\eps) \right) \big]
            \\&\le 
            \eps + \sup_{\tau \in \mathcal{T}_\infty} \ex \Big[ \left| {G}\left(t_1 + \tau, \Pi^{\pi}(\tau) \right) - {G}\left(t_2 + \tau, \Pi^{\pi}(\tau) \right) \right| \Big].
        \end{split}
    \end{equation*}
    Since $\eps$ was chosen arbitrarily, we have by symmetry
    \begin{equation}\label{upper_bound_onn_difference_value_funct}
        |{V}_\infty(t_1, \pi) - {V}_\infty(t_2, \pi)| 
        \leq 
        \sup_{\tau \in \mathcal{T}_\infty} \ex \left[ \Big| {G}\left(t_1 + \tau, \Pi^{\pi}(\tau) \right) - {G}\left(t_2 + \tau, \Pi^{\pi}(\tau) \right) \Big| \right].
    \end{equation}
    
    The functions $c_i(\cdot), \, i = 0, 1$ are continuous on the compact metric space $[0,\infty]$, thus uniformly continuous. Therefore, for the fixed above $\eps > 0$, there exists $\delta > 0$ such that $d(t_1,t_2) < \delta$ implies $|c_i(t_1) - c_i(t_2)| < \eps /4 $, $i = 0, 1$. Furthermore, for any $t_1, t_2 \ge 0$ with $d(t_1,t_2) < \delta$, we have 
    \begin{equation*}
        \begin{split}
            \sup_{\tau \in \mathcal{T}_\infty} \ex &\left[ \Big|{G}\left(t_1 + \tau, \Pi^{\pi}(\tau) \right) - {G}\left(t_2 + \tau, \Pi^{\pi}(\tau) \right) \Big|\right] \\
            & \leq
            \sup_{\tau \in \mathcal{T}_\infty} \ex \left[ \left|\Pi^{\pi}(\tau) \Big(c_1(t_1 + \tau) - c_1(t_2 + \tau) \Big)
            - \big(1-\Pi^{\pi}(\tau)\big) \Big(c_0(t_1 + \tau) - c_0(t_2 + \tau) \Big) \right| \right] 
            \\
            &\leq \sup_{\substack{t,s \in [0, \infty]: \\ d(t,s) < \delta }} \big|c_1(t) - c_1(s)\big| + \sup_{\substack{t,s \in [0, \infty]: \\ d(t,s) < \delta }} \big|c_0(t) - c_0(s)\big| \leq \frac{\eps}{2} < \eps.
        \end{split}
    \end{equation*}
    Combining the above bound with the inequality \eqref{upper_bound_onn_difference_value_funct}, we obtain 
    $\Big|{V}_\infty(t_1, \pi)- {V}_\infty(t_2, \pi)\Big| < \eps$,
    for any $t_1, t_2$ with $d(t_1,t_2) < \delta$. Since $\delta > 0$ does not depend on $\pi$, this proves the result for $V_\infty(\cdot, \cdot)$.
    
    \noindent \underline{Case \#2: $T < \infty$:} Fix $\pi \in [0, 1]$ and $\eps > 0$.  Choose $\delta > 0$ such that
    \begin{equation*}
        f(\delta) \coloneqq \ex\left[
        \sup_{(\pi,u,v,t, s)\in A_\delta}
            \Big|
                G\left(u, \Pi^\pi(t)\right) - G\left(v, \Pi^\pi(s)\right)
            \Big|
    \right] < \frac{\eps}{2}
    \end{equation*}
    with $G(\cdot, \cdot)$ as in \eqref{gain_function_with_indicator} and
    $A_\delta \coloneqq \Big\{(\pi,u,v,t,s) \in [0, 1]\times [0, T]^4:  |u-v| \leq \delta, |t - s| \leq \delta \Big\}.$ 
    Such a choice of $\delta$ is possible due to the uniform continuity of $c_i(\cdot)$, $i = 0, 1$, the uniform continuity and boundedness of the function $G(\cdot, \cdot)$ on $[0, T] \times [0,1]$, and the (pathwise) uniform continuity of the mapping $(t, \pi) \mapsto \Pi^\pi(t)$ on $[0, T] \times [0,1]$. The continuity of the mapping $(t, \pi) \mapsto \Pi^\pi(t)$ follows from a standard application of Gronwall's inequality and the Kolmogorov-Chentsov theorem (see, e.g., Theorem 4.3 of Kallenberg \cite{Kallenberg}). As a result, the dominated convergence theorem implies $\lim_{\delta \rightarrow 0} f(\delta) = 0$.
    
    Without loss of generality, suppose $s < t < s+\delta$. We first show that ${V}_T(t, \pi) - {V}_T(s, \pi) < \eps$. Let $\tau_{\eps,t}$ be an $\eps/2$--optimal for the problem with initial state $(t, \pi)$, and note that $\tau_{\eps,t}$ is a feasible stopping time for the problem with initial state $(s, \pi)$. Therefore, from our choice of $\delta$, we obtain
    \begin{equation*}
        \begin{split}
            {V}_T(t, \pi) - {V}_T(s, \pi) &\leq {J}(t, \pi, \tau_{\eps, t}) + \frac{\eps}{2} - {J}(s, \pi, \tau_{\eps,t}) \\
            & \le
            \ex\left[
                \Big|G(t+\tau_{\eps, t}, \Pi^\pi(\tau_{\eps, t})) -
                G(s+\tau_{\eps, t}, \Pi^\pi(\tau_{\eps, t}))\Big|
                \right] + \frac{\eps}{2}
            < \eps.
        \end{split}
    \end{equation*}
    
    We now show that ${V}_T(s, \pi) - {V}_T(t, \pi) < \eps$. Let $\tau_{\eps, s}$ be an $\eps/2-$optimal stopping time for the problem with initial state $(s, \pi)$, and define $\sigma_{\eps, s} \coloneqq \tau_{\eps,s} \wedge(T-t)$, so that $\sigma_{\eps, s}$ is feasible for the problem started at $(t, \pi)$. Note that $|s+\tau_{\eps,s} - (t+\sigma_{\eps,s})| < \delta$ and $|\tau_{\eps,s} - \sigma_{\eps,s}| < \delta$. Therefore,
    \begin{align*}
        {V}_T(s, \pi) - {V}_T(t, \pi)
        & \leq {J}_T(s, \pi, \tau_{\eps, s}) + \frac{\eps}{2} - {J}_T(t, \pi, \sigma_{\eps, s})
        \\
        &\leq \ex\left[
                    \Big|
                        G\big(s + \tau_{\eps, s}, \Pi^\pi(\tau_{\eps, s} )\big) 
                        -
                        G\big(t + \sigma_{\eps, s}, \Pi^\pi(\sigma_{\eps, s} )\big) 
                    \Big|
                \right]
        + \frac{\eps}{2}
         \leq f(\delta) + \frac{\eps}{2}
        < \eps.
    \end{align*}
    
    Since all of the above arguments hold irrespective of the choice of $\pi$, this proves the continuity of the function $V_T(\cdot, \pi)$, uniformly in $\pi$.
    \end{proof}
    
    \noindent \textit{\textbf{Proof of Proposition \ref{prop_properties_of_value_func}:}} We first fix $T < \infty$ and then prove the infinite horizon case via limit-based arguments.
    
    \underline{Proof of \ref{value_func_continuity}:} By Gronwall's inequality and It\^{o} isometry, 
    \begin{equation} \label{GronwallforPi}
        \ex \left[ \sup_{0 \leq s \leq T} \big|\Pi^\pi(s) - \Pi^p(s)\big|^2\right] \leq e^{T^2}|\pi-p|^2,
    \end{equation}
    for any $\pi, p \in [0,1]$. It is easy to see then that
    \begin{equation} \label{bad_Lipschitz_value}
        \begin{split}
            \big|V_T(t,\pi) - V_T(t, p)\big| &\leq \sup_{\tau \in \mathcal{T}_{T-t}} \big|J_T(t,\pi,\tau) - J_T(t,p,\tau) \big|
            \\ &\leq \Big(c_0(0) + c_1(0)\Big) \;\ex\left[ \sup_{0 \leq s \leq T} \big|\Pi^\pi(s) - \Pi^p(s) \big| \right]
            \\& \leq \Big(c_1(0) + c_0(0)\Big) \, e^{T^2/2} \, |\pi - p|, 
        \end{split}
    \end{equation}
    where the first inequality follows from a similar argument to that of \eqref{upper_bound_onn_difference_value_funct}, the second follows from the fact that $x \mapsto x^+$ is Lipschitz continuous with constant $1$ and $c_1(\cdot) + c_0(\cdot)$ is a non-increasing function, and the last from H\"older's inequality and \eqref{GronwallforPi}. Coupling \eqref{bad_Lipschitz_value} with Lemma \ref{lem_time_continuity_fixed_pi} shows that $V_T(\cdot,\cdot)$ is continuous and establishes \ref{value_func_continuity}.

    We now proceed to prove \ref{value_func_convexity}--\ref{value_func_Lipschitz_pi} in order. All properties hold for the case $t=T$ by the equality $V_T(T, \cdot) \equiv G(T, \cdot)$ and the fact that properties \ref{value_func_convexity}--\ref{value_func_Lipschitz_pi} are vacuous for the function $G(T, \cdot)$. Therefore, from now on, we will only consider the case $t \in [0, T)$.

    \underline{Proof of \ref{value_func_convexity}:} We first prove the convexity of $\widetilde{V}_T(t, \cdot)$ in \eqref{value_function_with_indicator_gain}, then invoke the first equality of \eqref{equality_value_functions}. 
    Moreover, to obtain the convexity of $\widetilde{V}_T(t, \cdot)$, it will be useful to revert to the original optimal stopping problem and use the observations of Subsection \ref{subsubsec_reversion_back}. More specifically, we will work with the conditional process $\widehat \Pi^\pi(\cdot)$ of \eqref{cond_process_general_pi} and use the representation \eqref{equality_value_functions} for the value function $\widetilde{V}_T(\cdot, \cdot)$. Fix any $\pi \in [0, 1]$ and consider any stopping time $\tau \in \widehat{\mathcal{T}}^\pi_{T-t}$. For this choice of $\pi$ we recall the definitions of $\theta^\pi$ and $W^\pi$ from Subsection \ref{subsubsec_reversion_back}. Applying the tower property of conditional expectations with the representation 
    $\widehat \Pi^\pi(\tau)=\ex^\pi\left[\mathbf{1}_{\{\theta^\pi = 1\}} \mid \mathcal{F}^{X^\pi}(\tau)\right]$
    we obtain,
    \begin{equation}
        \begin{split}
            \ex^{\pi}\left[ \widetilde{G}_T\left(t + \tau, \widehat \Pi^\pi(\tau)\right)\right] 
            &=
            \ex^\pi\Big[
                \Big(\widehat \Pi^\pi(\tau) \cdot (c_1(t + \tau)+c_0(t + \tau)) - c_0(t + \tau)\Big)
                \cdot 
                \mathbf{1}_{\{t + \tau < T\}}
            \Big]
            \\&=
            \ex^\pi\Big[
                \Big(\mathbf{1}_{\{\theta^\pi = 1\}} \cdot (c_1(t + \tau)+c_0(t + \tau)) - c_0(t + \tau)\Big)
                \cdot
                \mathbf{1}_{\{t + \tau < T\}}
            \Big].
        \end{split}\label{eqn.g.motivation}
    \end{equation}
    
    It is a very subtle fact that we may assume that $\tau = f_\tau(\theta^\pi, W^\pi)$ for some measurable function $f_\tau(\cdot, \cdot)$ without loss of generality. We discuss this subtlety in detail in Remark \ref{naturalfiltrationrelaxationremark} below. Motivated by the last equality in \eqref{eqn.g.motivation} we define the function $\mathfrak{g}_\tau:\{0, 1\} \to \mathbb{R}$ by 
    \begin{equation} \label{gfunction}
        \begin{split}
            \mathfrak{g}_\tau(x) \coloneqq \ex^\pi\bigg[\bigg(\mathbf{1}_{\{x = 1\}} \Big(c_1(t + f_\tau(x, W^\pi))&+c_0(t + f_\tau(x, W^\pi))\Big)
            \\&- c_0(t + f_\tau(x, W^\pi))\bigg)\mathbf{1}_{\{t + f_\tau(x, W^\pi) < T\}}\bigg].
        \end{split}
    \end{equation}
    Note that since $\mathfrak{g}_\tau(\cdot)$ only depends on the function $f_\tau(\cdot, \cdot)$ and the law of the Brownian motion $W^\pi(\cdot)$, it does not depend on $\pi$. With the help of the functions $f_\tau(\cdot, \cdot)$ and $\mathfrak{g}_\tau(\cdot)$,
    and the independence of $W^\pi$ and $\theta^{\pi}$, we can rewrite the expected reward as
    \begin{equation*}
        \ex^{\pi}\left[ \widetilde{G}_T\left(t + \tau, \widehat \Pi^\pi(\tau)\right)\right] 
        =
        \ex^\pi \left[\mathfrak{g}_\tau(\theta^\pi) \right] 
        =
        \pi \mathfrak{g}_\tau(1) + (1-\pi) \mathfrak{g}_\tau(0).
    \end{equation*}

    We are finally ready to prove the convexity of $\widetilde{V}_T(t, \cdot)$. Fix $\pi_1, \pi_2 \in [0,1]$, and $\lambda \in [0,1]$, and let $\pi_\lambda \coloneqq \lambda \pi_1 + (1-\lambda) \pi_2$. For any stopping time $\tau$ in the natural filtration of $\widehat \Pi^{\pi_\lambda}(\cdot)$, we have
    \begin{equation} \label{convexity_ineq}
        \begin{split}
            J_T(t, \pi_\lambda, \tau) = \pi_\lambda \mathfrak{g}_\tau(1) + (1-\pi_\lambda) \mathfrak{g}_\tau(0) &= \lambda J_T(t,\pi_1, \tau_1) + (1-\lambda) J_T(t, \pi_2, \tau_2) \\
            &\leq \lambda V_T(t,\pi_1) + (1-\lambda) V_T(t,\pi_2).
        \end{split}
    \end{equation}
    Here, $\tau_i \coloneqq f_\tau\left(\theta^{\pi_i}, W^{\pi_i}\right)$, $i = 1, 2$, are stopping times with respect to the natural filtrations of the processes $\widehat \Pi^{\pi_i}(\cdot)$, $i = 1,2$, respectively, as discussed in Remark \ref{naturalfiltrationrelaxationremark}, so the second equality follows from \eqref{eqn.g.motivation} and \eqref{gfunction}, and the inequality follows by definition. Taking the supremum over the appropriate set of stopping times $\tau$ on the left-hand side of \eqref{convexity_ineq} establishes \ref{value_func_convexity}.
    
    \underline{Proof of \ref{value_func_nondec}:}  We fix an arbitrary stopping time $\tau \in \mathcal{T}_{T - t}$, 
    and note that with $0 \leq \pi_1 \le \pi_2 \leq 1$, we have $0 \leq \Pi^{\pi_1}(t) \le \Pi^{\pi_2}(t) \leq 1$ for all $t \in [0, T]$ almost surely, by Proposition 5.2.18 of \cite{BMSC}. Therefore, since the functions $c_0(\cdot)$ and $c_1(\cdot)$ are non-negative, we obtain
    \begin{equation*}
        \begin{split}
            \Big(\Pi^{\pi_1}(t+s) &\cdot \big(c_1(t+s)  +c_0(t+s)\big) - c_0(t+s)\Big)^+ 
            \\& \le 
            \Big(\Pi^{\pi_2}(t+s) \cdot \big(c_1(t+s)+c_0(t+s)\big) - c_0(t+s)\Big)^+ \quad \forall \; 0 \le s \le T-t \text{ a.s.,}
        \end{split}
    \end{equation*}
    and thus $J_T(t, \pi_1, \tau) \le J_T(t, \pi_2, \tau)$,
    which implies the statement of part \ref{value_func_nondec}.
    
    \underline{Proof of \ref{value_func_Lipschitz_pi}:} As in the proof of part \ref{value_func_convexity}, we prove statement \ref{value_func_Lipschitz_pi} for the value function $\widetilde{V}_T(t, \cdot)$ in the representation \eqref{equality_value_functions}, and then invoke the first equality of \eqref{equality_value_functions}.
    We use again the formulation from Subsection \ref{subsubsec_reversion_back}.  Fix any $0 \leq \pi_1 \leq \pi_2 \leq 1$.  By the definition of the value function, for every $\eps > 0$, there exists an $\eps$-optimal stopping time $\tau_\eps \in \widehat{\mathcal{T}}^{\pi_2}_{T-t}$ that satisfies
    \begin{equation*}
        \widetilde{V}_T(t, \pi_2) \leq \ex^{\pi_2}\left[ \widetilde{G}_T\left(t + \tau_\eps, \widehat \Pi^{\pi_2}(\tau_\eps)\right)\right] + \eps.
    \end{equation*}
    Abusing notation, we denote the corresponding feasible stopping time for $\widehat \Pi^{\pi_1}(\cdot)$ referenced in Remark \ref{naturalfiltrationrelaxationremark} by $\tau_\eps$ as well. The trivial inequality 
    $\widetilde{V}_T(t, \pi_1) \ge \ex^{\pi_1}\left[ \widetilde{G}_T\left(t + \tau_\eps, \widehat \Pi^{\pi_1}(\tau_\eps)\right)\right]$
    and the fact that $\widetilde{V}_T(t, \cdot)$ is non-decreasing by \ref{value_func_nondec} imply
    \begin{equation}\label{inequality_chain_lipschitz_proof}
        \begin{split}
            0 \leq \widetilde{V}_T(t, \pi_2) - \widetilde{V}_T(t, \pi_1) 
            &\leq
            \ex^{\pi_2}\left[ \widetilde{G}_T\left(t + \tau_\eps, \widehat \Pi^{\pi_2}(\tau_\eps)\right)\right] - \ex^{\pi_1}\left[ \widetilde{G}_T\left(t + \tau_\eps, \widehat \Pi^{\pi_1}(\tau_\eps)\right)\right] + \eps 
            \\&=
            \pi_2 \cdot \mathfrak{g}_{\tau_\eps}(1) + (1-\pi_2) \cdot \mathfrak{g}_{\tau_\eps}(0) - \pi_1 \cdot  \mathfrak{g}_{\tau_\eps}(1) - (1-\pi_1) \cdot \mathfrak{g}_{\tau_\eps}(0) + \eps 
            \\&=
            (\pi_2 - \pi_1) \cdot (\mathfrak{g}_{\tau_\eps}(1) - \mathfrak{g}_{\tau_\eps}(0)) + \eps 
            \\&
            \leq |\pi_2 - \pi_1| \cdot (c_1(t) + c_0(t)) + \eps.
        \end{split}
    \end{equation}
    Here, the function $\mathfrak{g}_{\tau_\eps}(\cdot)$ is defined by analogy with \eqref{gfunction}, as an expectation with respect to the Brownian motion $W^\pi$. Consequently, it does not depend on the measure $\pr^{\pi}$, which justifies the first equality in \eqref{inequality_chain_lipschitz_proof}. The last inequality follows from the fact that $\mathfrak{g}_{\tau_\eps}(1) - \mathfrak{g}_{\tau_\eps}(0) \le c_1(t) + c_0(t)$ as is clear from the definition \eqref{gfunction}. As $\eps > 0$ was arbitrary, this suffices.

    Finally, we consider the case $T = \infty$. Properties \ref{value_func_convexity}--\ref{value_func_Lipschitz_pi} follow immediately from Lemma \ref{lem_value_incr_monotonically_T_upward} and \ref{value_func_convexity}--\ref{value_func_Lipschitz_pi} for the finite-horizon problem. The joint continuity property \ref{value_func_continuity} of $V_\infty(\cdot,\cdot)$ then follows immediately from the Lipschitz property \ref{value_func_Lipschitz_pi} and Lemma \ref{lem_time_continuity_fixed_pi}.
    \hfill \qed {\parfillskip0pt\par}

    \begin{remark} \label{naturalfiltrationrelaxationremark} \textbf{Regarding the Representation $\tau = f_\tau(\theta^\pi, W^\pi)$ in \eqref{gfunction}:}
        To see that we may assume such a representation, we first note that by Theorem 2.7 of Peskir and Shiryaev \cite{PesShi06}, the smallest optimal stopping time associated with $V_T(t,\pi)$ may be written as the first hitting time of the set $\mathcal{S}_T$, defined in \eqref{def_stopping_region}, by the continuous process $\big(t+\cdot, \widehat \Pi^\pi(\cdot)\big)$. Due to \ref{value_func_continuity}, $\mathcal{S}_T$ is a closed set, which implies, without use of the general Debut theorem for filtrations satisfying the usual conditions, that this hitting time is a stopping time with respect to the natural filtration of $\widehat \Pi^\pi(\cdot)$ (which indeed does not necessarily satisfy the usual conditions). Thus, the supremum over all $\widehat{\mathcal{T}}^\pi_{T-t}$ stopping times may be replaced by the supremum over all stopping times with respect to the natural filtration of $\widehat \Pi^\pi(\cdot)$. This fact holds, of course, for all the representations proved in Section \ref{sec_equivalent_formulations} of the paper. It is straightforward to show that any stopping time $\tau$ with respect to the natural filtration of $\widehat \Pi^\pi(\cdot)$ may be written as $f\big(\widehat \Pi^\pi(\cdot)\big)$ for some measurable function $f(\cdot)$. Furthermore, one may verify that for any $[0,1]$--valued continuous process $Y = \{Y(t)\}_{t \in [0,T]}$ on an \emph{arbitrary} probability space, $f(Y(\cdot))$ is a stopping time with respect to the natural filtration of $Y(\cdot)$. Since $\widehat \Pi^\pi(\cdot)$ is a measurable function of $\theta$ and $W^\pi(\cdot)$ through its relationship with $X^\pi(\cdot)$, described by \eqref{X_and_pi_relationship} and \eqref{cond_process_general_pi}, the desired representation follows.
    \end{remark}

\subsection{Proof of Proposition \ref{prop_value_func_derivative_continuity}}\label{subsec_app_v_pi_continuity}

    Fix $T < \infty$.
    We prove that $\partial_\pi V_T(\cdot, \cdot)$ exists and is jointly continuous by exploiting the connection of our problem with American options, discussed in Subsection \ref{subsec_american_options_formulation}. 
    
    Recall the representation \eqref{american_call_value_fn} of the value function $V_T(\cdot, \cdot)$ via the value function $u_T(\cdot, \cdot)$ of the corresponding American option defined in \eqref{value_function_sup_american_option}--\eqref{american_call_value_fn}. Since the function $y(\cdot, \cdot)$ of \eqref{change_of_initial_position} is continuously differentiable in the spatial argument on $[0, T) \times (0, 1)$, 
    to obtain the same property of the function $V_T(\cdot, \cdot)$ it suffices to show joint continuity of the function $\partial_y u_T(\cdot, \cdot)$ on $[0, T) \times \mathbb{R}$. The latter, however, is immediately implied by Theorem 3.6 and Corollary 3.7 of \cite{Lamberton}, whose framework is applicable to our situation. 
    
    To be precise, we match the notation of that paper: the assumption (H1) is satisfied by $\beta(t) \coloneqq \beta_1(t) - \beta_0(t) - a^2/2$, defined in \eqref{definition_beta_functions}, and our assumption \ref{reg_assm_2}. Assumptions (H2) and (H3) are trivially satisfied by the constant diffusion coefficient $a > 0$. Assumption (H4) is satisfied by the function $r(t) \coloneqq -\beta_0(t)$ due to \ref{reg_assm_2} and \ref{reg_assm_3}. Assumption (H5) is satisfied due to the fact that the function $e^{-x} \, \psi(x) = e^{-x}(e^x - 1)^+$ is a bounded, continuous function with a bounded weak derivative. And assumption (H6) is satisfied since, in our case, the function $\psi(\log(x)) = (x-1)^+$ is convex on $(0, \infty)$. 
    Therefore, the value function $u_T(\cdot, \cdot)$ coincides with the one defined in equation (2.2) of \cite{Lamberton}, implying continuous differentiability of $u_T(t,\cdot)$ for each $t \in [0,T)$. \hfill \qed {\parfillskip0pt\par}

\subsection{Proof of Proposition \ref{prop_Lipschitz_property_spatial_derivative}}\label{subsec_app_lipschitz}

We fix any $T \in [0, \infty]$, $0 \le t < t' < T$, and note that, for any $S \in [t', T)$, the functions $c_i(\cdot), \; i = 0,1$ are of class $C^2$ on the closed interval $[0,S]$, and $c_i(S) > 0$, $i = 0, 1$. The proof of Proposition \ref{prop_Lipschitz_property_spatial_derivative} relies on the relationship \eqref{american_call_value_fn} between the value functions for the original problem $\{V_S(\cdot, \cdot)\}_{0 \leq S < T}$ and the value functions for the appropriate American options $\{u_S(\cdot, \cdot)\}_{0 \leq S < T}$, and several bounds for the latter ones. Therefore, we first establish several results about the functions $\{u_S(\cdot, \cdot)\}_{0 \leq S < T}$ and proceed to prove Proposition \ref{prop_Lipschitz_property_spatial_derivative} afterward.

\begin{Lem}\label{lem_time_changed_american_option}
    For any $S \in [0, T)$ and $(t, y) \in [0, S] \times \mathbb{R}$, the value function $u_S(\cdot, \cdot)$ of \eqref{value_function_sup_american_option}--\eqref{american_call_value_fn} has the following representation:
    \begin{equation}\label{representation_u_time_changed}
        \begin{split}
            u_S(t, y) &= 
            \sup_{\theta \in \overline{\mathcal{T}}_1} 
            \ex\left[
                e^{\int_{0}^{\theta}(S-t) \, \beta_0(t+v(S-t)) \, dv } 
                \cdot
                \psi\left(
                    y 
                    + 
                    \int_{0}^{\theta}(S-t) \, \beta\big(t+v(S-t)\big) \, dv
                    +
                    a \sqrt{S-t} \cdot \overline{W}(\theta)
                    \right)
            \right].
        \end{split}
    \end{equation}
    Here, 
    $\overline{W} \coloneqq \left\{\overline{W}(t), \, 0 \le t \le 1\right\}$
    is a standard Brownian motion, $\overline{\mathcal{T}}_1$ is the family of $\overline{W}$--stopping times $\theta$ such that $\pr(\theta \le 1) = 1$, and we denote
    $\beta(t) \coloneqq \beta_1(t) - \beta_0(t) - a^2/2$.
\end{Lem}

\begin{proof}
    Lemma \ref{lem_time_changed_american_option} follows from the analogous Lemma 3.9 of \cite{Lamberton}. The main idea of the proof is to note that the processes 
    \begin{equation*}
        \Big(\overline{B}\big(t + s \cdot (S-t)\big) - \overline{B}\big(t\big)\Big)_{s \ge 0} 
        \quad \text{ and } \quad
        \left(\sqrt{S - t} \cdot \overline{W}(s) \right)_{s \ge 0},
    \end{equation*}
    where $\overline{B}(\cdot)$ is defined in \eqref{girsanov_BM_definition}, and $\overline{W}(\cdot)$ is a standard Brownian motion, have the same law. Therefore, after the time-change $\tau \mapsto \tau / (S - t)$ and the appropriate changes of variables inside all integrals, the representation \eqref{representation_u_time_changed} follows from the expectation in \eqref{value_function_sup_american_option} readily. We refer the reader to the proof of Lemma 3.9 in \cite{Lamberton} for more details. We note that the same logic used in Remark \ref{naturalfiltrationrelaxationremark} is required here as well.
\end{proof}

With the representation \eqref{representation_u_time_changed} it will be convenient to introduce the following notation. We consider an arbitrary filtered probability space $\left(\overline{\Omega}, \overline{\mathcal{F}}, \overline{\mathbb{F}}, \overline{\pr}\right)$ and adopt the notation from Lemma \ref{lem_time_changed_american_option}.
We define the functions
\begin{equation}\label{function_for_new_state_process}
    f^z_S(t, s, x) \coloneqq z
    + 
    \int_{0}^{s}(S-t) \, \beta\big(t+v(S-t)\big) \, dv
    +
    a \sqrt{S-t} \cdot x,
\end{equation}
and consider the processes $Z_{t, S}^z \coloneqq \left\{Z_{t, S}^z(s), 0 \le s \le 1\right\}$, which satisfy
\begin{equation}\label{new_state_process}
    Z_{t, S}^z(s) \coloneqq 
    f_S^z\left(t, s, \overline{W}(s)\right) = 
    z
    + 
    \int_{0}^{s}(S-t) \, \beta\big(t+v(S-t)\big) \, dv
    +
    a \sqrt{S-t} \cdot \overline{W}(s), \, \, 0 \le s \le 1.
\end{equation}
Note that the functions $f_S^z(\cdot, \cdot, \cdot)$ are differentiable in the $t$ argument, so the flow $t \mapsto Z_{t, S}^z(\cdot)$ is also differentiable. We denote the derivative of $Z_{t, S}^z(\cdot)$ with respect to $t$ by $\partial_t Z_{t, S}^{z}(\cdot)$, and, by differentiating \eqref{function_for_new_state_process} with respect to $t$, obtain
\begin{equation}\label{Z_flow_formula}
    \partial_t Z_{S, t}^{z}(s) 
    = (1-s) \beta(t+s(S-t)) - \beta(t) 
    - \frac{a}{2 \sqrt{S-t}} \cdot \overline{W}(s), \quad 0 \le s \le 1.
\end{equation}
Note that the processes $\partial_t Z_{t, S}^{z}(\cdot)$ satisfy
\begin{equation}\label{process_derivative_def}
    \lim_{h \rightarrow 0} \, \ex\left[\sup_{s \in [0,1]} \left|\frac{Z_{t + h, S}^{z}(s) - Z_{t, S}^z(s)}{h} - \partial_t Z_{t, S}^z(s) \right|^p\right] = 0, \quad \forall \, p \ge 1.
\end{equation}

For convenience of exposition, we define the functions
\begin{equation}\label{new_discount_function}
    \begin{split}
        r_{S}^t(s) &\coloneqq -(S-t) \, \beta_0(t + s(S-t)), \quad 0 \le s \le 1,
        \\
        R_S^t (s) &\coloneqq 
        \exp\left(- \int_0^s r_S^t(v) \, dv \right) 
        = \exp\left(\int_0^s (S-t) \,\beta_0\big(t + v(S-t) \big) \,  dv \right), \quad 0 \le s \le 1,
    \end{split}
\end{equation}
so that, from a straightforward computation, we obtain
\begin{equation}\label{new_discounting_derivative}
    \partial_t R_S^t(s) = R_S^t(s) \cdot \bigg((1-s) \beta_0(t+s(S-t)) - \beta_0(t) \bigg), \quad 0 \le s \le 1.
\end{equation}
The following important representation then holds.

\begin{Lem}\label{lem_time_derivative_representation}
    For any $S \in [0, T)$, the function $\partial_t u_S(\cdot, \cdot)$, in the form of \eqref{representation_u_time_changed}, has the following representation in the continuation region:
    \begin{equation}\label{time_derivative_representation}
        \begin{split}
            \partial_t u_S(t, y) = \ex &\left[
                R_S^t(\theta^*)
                \Bigg\{
                    \Big((1-\theta^*)\beta_0(t+\theta^*(S-t)) - \beta_0(t)\Big) 
                    \cdot
                    \psi\left(Z_{t, S}^y(\theta^*)\right)
                    \right.
                    \\& \left.  + \;
                    \psi'\left(Z_{t, S}^y\big(\theta^*)\right)
                    \cdot
                    \left(
                        (1-\theta^*)\beta(t+\theta^*(S-t)) - \beta(t)  - \frac{a}{2\sqrt{S-t}} \cdot \overline{W}(\theta^*)
                    \right)
                \Bigg\}
            \right].
        \end{split}
    \end{equation}
    Here, $\theta^*$ is an optimal stopping time for the problem \eqref{representation_u_time_changed} with the initial state $(t, y)$.
\end{Lem}

\begin{proof}
    Fix $S \in [0, T)$ and an arbitrary point $(t, y)$ from the continuation region. Note that at $(t, y)$ the derivative $\partial_t u_S(t, y)$ exists, and let $\theta^* \coloneqq \theta^*(t,y)$ be an optimal stopping time for the problem \eqref{representation_u_time_changed} with this initial state. To get the equality \eqref{time_derivative_representation} it suffices to show the inequalities in both directions; we start with ``$\ge$''.

    Consider an arbitrary but small $h > 0$ and note that the stopping time $\theta^*$ is admissible for the problem started $(t+h, y)$ (this is the main reason why we make the time change in Lemma \ref{lem_time_changed_american_option}).
    Therefore, using \eqref{representation_u_time_changed} and the notation \eqref{new_state_process}--\eqref{new_discount_function}, we obtain 
    \begin{equation} \label{u_t_lower_bound_1}
        \begin{split}
            &\frac{u_S(t+h, y) - u_S(t,y)}{h} 
            \ge 
            \frac{1}{h} \Bigg( 
               \ex\Big[R_S^{t+h} (\theta^*) \cdot \psi\Big(Z_{t + h, S}^y(\theta^*) \Big) \Big] 
               - \ex\Big[R_S^t (\theta^*) \cdot \psi\Big(Z_{t, S}^y(\theta^*) \Big) \Big] 
            \Bigg)
            \\
            & \qquad= \ex\left[
                \frac{\psi\left(Z_{t + h, S}^y(\theta^*) \right) - \psi\left(Z_{t, S}^y(\theta^*) \right)}{h} 
                \, R_S^{t+h}(\theta^*) 
            \right] 
             + \ex\left[
                \frac{R_S^{t+h}(\theta^*) - R_S^{t}(\theta^*)}{h} 
                \, \psi\left(Z_{t, S}^y(\theta^*) \right)
            \right].
        \end{split}
    \end{equation}

    Note that the function $\psi(y) = (e^y - 1)^+$ is differentiable everywhere except at $y = 0$.
    Moreover, note that $\pr\big(Z_{t, S}^{y}(\theta^*) = 0\big) = 0$, i.e., we are almost surely away from the ``kink'' in the function $\psi(\cdot)$ at the optimal stopping time $\theta^*$. 
    This follows from the facts that $\pr\big(Z_{t, S}^y(S) = 0\big)$, because $Z_{t, S}^{y}(S)$ has a density, and $u_S(t, \cdot) > 0$ for any $t < S$, so it is never optimal to stop at $y = 0$ before the terminal time $S$.
    Therefore, letting $h \downarrow 0$ on both sides of \eqref{u_t_lower_bound_1}, we obtain
    \begin{equation} \label{u_t_lower_bound_2}
        \partial_t u_S(t,y) 
        \ge 
        \ex \left[
            \psi'\left(Z_{t, S}^y(\theta^*)\right) 
            \cdot \partial_t Z_S^{t,y}(\theta^*)
            \cdot R_S^t(\theta^*)
        \right]
        +
        \ex \left[
            \partial_t R_S^t(\theta^*) 
            \cdot \psi\left(Z_{t, S}^y(\theta^*)\right)
        \right],
    \end{equation}
    where the expressions under the expectations coincide with the one in \eqref{time_derivative_representation} due to \eqref{Z_flow_formula} and \eqref{new_discounting_derivative}.
    Passing to the limit in \eqref{u_t_lower_bound_1} is justified by the Lebesgue-Vitali convergence theorem, which is applicable due to the boundedness of the functions $R_S^t(\cdot)$ and $\partial_t R_S^t(\cdot)$ and the boundedness in $L^2$ of all other terms due to the form of $\psi(\cdot)$ and \eqref{new_state_process}--\eqref{process_derivative_def}.

    The reverse inequality in \eqref{u_t_lower_bound_2} can be obtained similarly by considering the ratio $(u_S(t, y) - u_S(t-h,y))/h$
    and arguing by analogy with \eqref{u_t_lower_bound_1}; we omit the details. The equality \eqref{time_derivative_representation} follows readily.
\end{proof}

\begin{Lem}\label{lem_u_t_final_bound}
    Fix any $T \in [0,\infty]$ and $0 \leq t < T$.
    For any $S \in (t, T)$, the function $\partial_t u_S(\cdot, \cdot)$ has the following bound in the continuation region:
    \begin{equation}\label{u_t_final_bound}
        \left|\partial_t u_S(t, y)\right| \leq
        \frac{D \; e^y}{c_1(t)}
                \left(
                    1 + 2a^2 + |\beta_1(t)| + \frac{a}{\sqrt{2\pi(S-t)}} e^{-a^2(S-t)/2}
                \right)
                +
                \frac{D}{c_0(t)} \Big(1 + |\beta_0(t)| \Big),
    \end{equation}
    where
    \begin{equation} \label{D_definition}
        D \coloneqq D(t,T) \coloneqq\sup_{\substack{s\in [t,T), \\ i \in \{0,1\}}} \Big(|c_i'(s)| \vee |c_i(s)|\Big) < \infty.
    \end{equation}
    The last inequality in \eqref{D_definition} follows from the assumptions \ref{reg_assm_3}, \ref{reg_assm_4}.
\end{Lem}

\begin{proof}
    Fix $0 \leq t < S < T$. To obtain \eqref{u_t_final_bound}, we will manipulate the representation \eqref{time_derivative_representation} using the following observations. First, we recall that $\psi(y) = (e^y - 1) \cdot \mathbf{1}(y > 0)$, $\psi'(y) = e^y \cdot \mathbf{1}(y > 0)$ for all $y \in \mathbb{R}\setminus\{0\}$, and note that we can write 
    \begin{equation*}
        U_i(s) \coloneqq e^{\int_0^{s} (S-t) \beta_i(t+ (S-t)v) \, dv} = \frac{c_i(t + s(S-t))}{c_i(t)}, \quad 0 \le s \le 1, \ i = 0,1.
    \end{equation*} 
    Therefore, using the definitions \eqref{new_state_process} and \eqref{new_discount_function}, we obtain the representation
    \begin{equation*}
        V(\theta^*) \coloneqq R_S^t(\theta^*) \exp\left(Z_{t,S}^y(\theta^*)\right) = U_0(\theta^*) \exp\left(Z_{t,S}^y(\theta^*)\right) = e^y \cdot U_1(\theta^*) \cdot M(\theta^*), 
    \end{equation*} 
    where $M(s) \coloneqq \mathcal{E}\left(a\sqrt{S-t} \cdot \overline{W}(s)\right), \, 0 \le s \le 1,$
    and $\mathcal{E}(\cdot)$ denotes the stochastic exponential. With $\alpha^* \coloneqq t + \theta^*(S-t)$, collecting terms in \eqref{time_derivative_representation} shows that
    \begin{equation*}
        \begin{split}
        \partial_t u_S(t,y)
        &=\ex
            \Bigg[
                \left\{
                    V(\theta^*)
                        \left(
                            (1-\theta^*) \beta_1(\alpha^*)  - \beta_1(t) + \frac{\theta^* a^2}{2} - \frac{a}{2 \sqrt{S-t}} \overline W(\theta^*)
                        \right)\right. \\
                    &\qquad \left. \phantom{\frac{\theta^* a^2}{2}} 
                    - U_{0}(\theta^*)
                        \Big(
                            (1-\theta^*) \beta_0(\alpha^*) - \beta_0(t)
                        \Big)
                \right\}
                \cdot \mathbf{1}_{\left\{Z_{t,S}^y(\theta^*) > 0\right\}}
            \Bigg]
        \\
        &=\ex
            \Bigg[
                \left\{
                    \frac{e^y \cdot M(\theta^*)}{c_1(t)} 
                    \Bigg(
                        (1-\theta^*)c_1'(\alpha^*) \right.+ \left.c_1(\alpha^*)
                        \left(
                            \frac{\theta^* a^2}{2} - \frac{a}{2 \sqrt{S-t}} \overline W(\theta^*) - \beta_1(t)\right)
                    \right)
                    \\
                    &\qquad \left. \phantom{\frac{\theta^* a^2}{2}} 
                    -\frac{1}{c_0(t)} \Big((1-\theta^*) c_0'(\alpha^*) - c_0(\alpha^*)\beta_0(t) \Big)
                \right\}
                \cdot \mathbf{1}_{\left\{Z_{t,S}^y(\theta^*) > 0\right\}}
            \Bigg].
        \end{split}
    \end{equation*}
    Note that the process $M(\cdot)$ is a (square-integrable) martingale, so the optional sampling theorem and the boundedness of $\theta^*$ imply $\ex\left[ M(\theta^*) \right] = 1$.
    Thus, with $D$ given by \eqref{D_definition}, we obtain
    \begin{equation} \label{u_t_prefinal_bound}
        \begin{split}
            \big|\partial_t u_S(t,y)\big|
            &\leq \frac{D \; e^y}{c_1(t)}
                \left(
                    1 + \frac{a^2}{2} + |\beta_1(t)| + \frac{a}{2 \sqrt{S-t}} \; \ex\left[\Big|M(\theta^*) \cdot \overline W(\theta^*) \Big| \right]
                \right)
                +
                \frac{D}{c_0(t)} \Big(1 + |\beta_0(t)| \Big).
        \end{split}
    \end{equation}
    
    To bound this expression, it only remains to find an appropriate bound for $\ex \left[ \big| M(\theta^*) \cdot 
        \overline{W}(\theta^*) \big|
        \right].$
    At this point, we note that the process $M(\cdot)$ in non-negative, so the process $M(\cdot) \, \overline{W}(\cdot)$ is a submartingale, which follows immediately from its dynamics
    \begin{equation} \label{MWdecomp}
        M(s) \, \overline{W}(s) = 
        \int_0^s M(r) \, d\overline{W}(r)
        + \int_0^s \overline{W}(r) \, dM(r) 
        + \int_0^s a \sqrt{S-t} \, M(r) \, dr, \quad 0 \le s \le 1.
    \end{equation}
    Let $N(s)$ denote the sum of the first two terms in \eqref{MWdecomp}, and $A(s)$ denote the last term. Noting that $N(\cdot)$ is a martingale and $A(\cdot)$ is an increasing process, we have
    \begin{align}
            \ex \left[ \Big| M(\theta^*) \cdot 
            \overline{W}(\theta^*) \Big|
            \right]
            &\leq \ex \left[\Big|N(\theta^*)\Big| + A(1) \right] \leq \ex \left[\Big|N(1)\Big| + A(1) \right] 
            \\&\leq  \ex \left[\Big|N(1) + A(1)\Big| + 2A(1) \right]
            =
            \ex \left[ \Big| M(1) \cdot 
            \overline{W}(1) \Big|
            \right] + 2 a \sqrt{S-t}
            \\
            &= 
            \ex \left[
            e^{-a^2(S-t)/2 + a\sqrt{S-t} \cdot \overline{W}(1)} \cdot \Big| \overline{W}(1)  \Big|
            	\right] + 2a \sqrt{S-t}
            \\&=
            e^{-a^2(S-t)/2}
            \left(
                \sqrt{\frac{2}{\pi}} + a\sqrt{S-t} \cdot \erf \left(\frac{a\sqrt{S-t}}{\sqrt{2}} \right) 
                e^{a^2(S-t)/2}
            \right)
            + 2a\sqrt{S-t}
            \\&\le
            e^{-a^2(S-t)/2}\sqrt{\frac{2}{\pi}} + 3a\sqrt{S-t}, \label{submartingale_bound}
    \end{align}
    where $\erf(z) \coloneqq 2/\sqrt{\pi}\cdot \int_{0}^z e^{-t^2} \, dt$.
    The first three inequalities in \eqref{submartingale_bound} follow from the triangle inequality and the fact that $|N(\cdot)|$ is a submartingale, the subsequent two equalities follow from our definitions and Fubini's theorem applied to the $2A(1)$ term, and the third equality follows from the identity 
    \begin{equation*}
	    \ex\left[ \exp(c \, \mathcal{N})  \cdot \big|\mathcal{N}\big| \right]
        =
        \sqrt{\frac{2}{\pi}} + c \cdot \erf \left(\frac{c}{\sqrt{2}} \right) \exp\left( \frac{c^2}{2} \right), \quad c \ge 0,
    \end{equation*}
    for a standard normal distribution, which can be obtained by a straightforward integration.

    Combining \eqref{u_t_prefinal_bound} and \eqref{submartingale_bound}, we obtain \eqref{u_t_final_bound}.
\end{proof}

\begin{Lem}\label{lem_u_y_final_bound}
    For any $S \in [0, T)$, the function $\partial_y u_S(\cdot, \cdot)$ has the following bound:
    \begin{equation}\label{u_y_final_bound}
        \big|\partial_y u_S(t, y)\big|
        \le
        e^y \left(1 + \frac{c_1(t) + c_0(t)}{c_1(t) + c_0(t)e^y} \right).
    \end{equation}
\end{Lem}

\begin{proof}
    Fix arbitrary $S \in [0, T)$, $t \in [0, S)$ and recall \eqref{american_call_value_fn}, which can be written as
    \begin{equation}\label{u_representation_via_V}
        u_S(t, y) = \frac{c_1(t) + c_0(t)e^y}{c_0(t) \, c_1(t)} \, 
        V_S\left(t, \frac{c_0(t)e^y}{c_1(t) + c_0(t) e^y}\right)
    \end{equation}
    by inverting the change of variables \eqref{change_of_initial_position}. 
    Note also the bounds  
    \begin{equation}\label{bounds_for_V_V_pi}
        0 \le V_S(t, \pi) \le c_1(t) 
        \quad \text{ and } \quad
        0 \le \partial_\pi V_S(t, \pi) \le c_1(t) + c_0(t), 
        \quad \forall \, \pi \in (0, 1).
    \end{equation}
    Here, the first bound follows from $0 \le G(t, \pi) \le c_1(t)  \, \pi$ in \eqref{gain_function} and the fact that $\Pi^\pi(\cdot), \pi \in (0, 1)$, is a martingale.
    Thus, the optional sampling theorem and the monotonicity of the function $c_1(\cdot)$ imply $\ex \left[G(t + \tau, \Pi^\pi(\tau)) \right] \le \ex \left[c_1(t + \tau) \, \Pi^\pi(\tau) \right] \le c_1(t) \ex \left[\Pi^\pi(\tau) \right] = c_1(t)$ for any $\pi \in (0, 1), \, \tau \in \mathcal{S}_{S-t}$. The second bound in \eqref{bounds_for_V_V_pi} follows from the monotonicity and Lipschitz properties of the value function proved in parts \ref{value_func_nondec}--\ref{value_func_Lipschitz_pi} of Proposition \ref{prop_properties_of_value_func}.
    The bound \eqref{u_y_final_bound} then follows by differentiating both sides of \eqref{u_representation_via_V} with respect to the variable $y$ and using the bounds \eqref{bounds_for_V_V_pi}.
\end{proof}

\noindent \textit{\textbf{Proof of Proposition \ref{prop_Lipschitz_property_spatial_derivative}:}}
We divide the proof into several steps. 

\underline{Step 1:} 
We fix any $T \in [0,\infty]$, $0 \leq t < t' < T$, and note that it suffices to consider compact sets $I_n \coloneqq[1/n, 1-1/n] \subset (0, 1), n \ge 2$. Fix any such $I_n$ and recall that the functions $\{V_S(t, \cdot)\}_{t' \le S < T}$ are linear in the stopping regions $\{\mathcal{S}_S\}_{t' \le S < T}$ with derivatives equal to $c_1(t) + c_0(t)$. Moreover, note that the functions $\{V_S(t, \cdot)\}_{t' \le S < T}$ are continuously differentiable on $(0, 1)$ by Proposition \ref{prop_value_func_derivative_continuity}, and
twice continuously differentiable in the continuation regions $\{\mathcal{C}_S\}_{t' \le S < T}$ by Proposition \ref{prop_boundary_problem_value_func}.
Therefore, in order to obtain \eqref{lipschitzianity_spatial_derivative}, it suffices to show that the functions $\{\partial_{\pi \pi} V_S(t, \cdot)\}_{t' \le S < T}$ are uniformly bounded on the intersections of $I_n$ with the spatial sections of $\mathcal{C}_S$, respectively. Moreover, the PDE \eqref{value_function_boundary_problem} for the value functions allows us to establish the uniform bounds for the functions $\{\partial_{t} V_S(t, \cdot)\}_{t' \le S < T}$ instead.

\underline{Step 2:} 
Recalling the relationship \eqref{american_call_value_fn} between $V_S(\cdot, \cdot)$ and $u_S(\cdot, \cdot)$, we obtain
\begin{equation*}
        \partial_t V_S(t, \pi) 
        = (1-\pi) \, c_0'(t) \, u_S\big(t, y(t,\pi)\big) 
         + (1-\pi) \, c_0(t) \Big( 
            \partial_t u_S\big(t, y(t,\pi)\big) 
            + \partial_y u_S\big(t, y(t,\pi)\big) \,  \partial_t y(t,\pi)
            \Big).
\end{equation*}
Since we have fixed $t, t', T$, and the compact set $I_n$, to obtain uniform bounds for $\{\partial_t V_S(t, \cdot)\}_{t' \le S < T}$ on $\{I_n \cap \mathcal{C}_S\}_{t' \le S < T}$, it suffices to obtain uniform bounds for the functions $u_S\big(t, \cdot), \partial_y u_S\big(t, \cdot)$, and $\partial_t u_S\big(t, \cdot)$
on compact sets $\mathbb{R} \supset \widetilde{I}_n \coloneqq y(t', I_n)$, uniformly in $S \ge t'$, where the map $\pi \mapsto y(t', \pi)$ is defined in \eqref{change_of_initial_position}.

\underline{Step 3:} The bounds for $\partial_t u_S\big(t, \cdot)$ and $\partial_y u_S\big(t, \cdot)$ follow readily from Lemmas \ref{lem_u_t_final_bound} and \ref{lem_u_y_final_bound}, respectively. As for $u_S\big(t, \cdot)$, we clearly have $u_S\big(t, y) \le e^y$ for any $y \in \mathbb{R}$, due to the fact that
\begin{equation*}
    \ex\left[ R_S^t(\theta^*) \left(e^{Z_{t,S}^y(\theta^*)}-1\right)^+ \right] \leq \ex \left[\left(e^{Z_{t,S}^y(\theta^*)}-1\right)^+ \right] \leq \ex\left[e^{Z_{t,S}^y(\theta^*)} \right] \leq e^y,
\end{equation*}
where the first inequality follows from the non-positivity of $\beta_0(\cdot)$, the second follows trivially since $(e^y-1)^+ \leq e^y$ for every $y \in \mathbb{R}$, the third follows from the fact that $\beta_1(\cdot) - \beta_0(\cdot)$ is non-positive by the assumption \ref{reg_assm_5}, so $e^{Z_{t,S}^y(\cdot)}$ is a supermartingale.

As a result, combining all of the above arguments and bounds and using the fundamental theorem of calculus, which is applicable since the functions $V_S(t, \cdot), \, t \in [0, S)$ are piece-wise $C^2$, we obtain the statement \eqref{lipschitzianity_spatial_derivative} of Proposition \ref{prop_Lipschitz_property_spatial_derivative} with the function
\begin{equation*}
    C(S, t, K) \coloneqq \frac{D(t,T)}{\underline{k}^2 (1-\overline{k})^2} \left(5 |\beta_0(t)| + 4 |\beta_1(t)| + 2 + 2a^2 + \frac{a}{\sqrt{2\pi(S-t)}}\; e^{-a^2(S-t)/2}\right),
\end{equation*}
where $\underline{k} = \inf K$, $\overline{k} = \sup K$, and the constant $D(t,T)$ is defined in \eqref{D_definition}. It only remains to note that the above function is finite and decreasing in $S$. 
    \hfill \qed {\parfillskip0pt\par}

\subsection{Proof of Proposition \ref{ito_for_value_function}}\label{subsec_app_ito}

    Before we begin, we note that certain aspects of the arguments below, particularly those involving mollification, could be shortened by invoking Krylov's extension of It\^o's formula to Sobolev functions; see \cite[Ch.~2, \S10, Thm.~2.10.1]{Krylov80}. For the sake of self-containment, however, we give a direct proof.

    The identity \eqref{ito_equation_for_value_function} is trivial for $t = T$ and in the cases $\pi = 0$, $\pi = 1$, where the process $\Pi^\pi(s) \equiv \pi$ for every $s \in [0, \infty]$. Thus, we will prove \eqref{ito_equation_for_value_function} fixing $(t, \pi) \in [0,T) \times (0,1)$. In a similar fashion to the proof of Proposition \ref{prop_C1_value_func}, we will first prove the result when $T < \infty$, $c_i(\cdot) \in C^2([0, T]), \, i = 0,1$, and $c_i(T) > 0, \, i = 0, 1$, and then handle the remaining cases via limit-based arguments. In the former case, since $V_T(\cdot, \cdot)$ is uniformly continuous on $[0, T] \times [0,1]$ by Proposition \ref{prop_properties_of_value_func}, it may be extended to a bounded continuous function on $\mathbb{R}^2$, and abusing notation, we refer to this extension by $V_T(\cdot, \cdot)$ as well. The remainder of the proof adapts the standard mollification argument in Theorem 2.7.9 of \cite{KarShr98}. For convenience, we divide the proof into several steps.

    \underline{Step 1:}
    We start by introducing the necessary notation. We let the mollifier $\phi: \mathbb{R}^2 \rightarrow [0, \infty)$ be a $C^\infty(\mathbb{R}^2)$ function such that
    \begin{equation*}
        \begin{split}
            \int_{\mathbb{R}^2} \phi(x,y) \, dx \, dy = 1  \quad \text{ and } \quad 
            \phi(x,y) = 0 \, \text{ for } \, (x,y) \not \in [0,1]^2.
        \end{split}
    \end{equation*}
    For any $\eps > 0$ and $(t, \pi) \in [0, T) \times (0,1)$, we define the mollified function $V_T^{(\eps)}(\cdot, \cdot)$ by
    \begin{equation} \label{definition_mollified_value_function}
        V_T^{(\eps)}(t, \pi) 
        \coloneqq \int_{\mathbb{R}^2} V_T(t + \eps u, \pi + \eps v)\phi(u,v) \, du \, dv.
    \end{equation}
    By the change of variables formula, we obtain 
    \begin{align}\label{integral_change_of_variables}
         \int_{\mathbb{R}^2} V_T(t + \eps u, \pi + \eps v)\phi(u,v) \, du \, dv = \frac{1}{\eps^2} \int_{\mathbb{R}^2} V_T(s, y) \phi\left(\frac{s-t}{\eps}, \frac{y-\pi}{\eps}\right) ds \, dy.
    \end{align}
    This form allows to see that $V_T^{(\eps)}(\cdot, \cdot)$ is $C^\infty\left(\mathbb{R}^2\right)$. 
    Indeed, since $V_T(\cdot, \cdot)$ is bounded and the derivatives of $\phi(\cdot, \cdot)$ are continuous to any order and have compact support, the bounded convergence theorem implies that one may differentiate under the integral in the expression on the right-hand side of \eqref{integral_change_of_variables}.
    Finally, we fix $\delta \in \left(0, T\wedge 1\right)$, let $S_\delta \coloneqq [\delta, T-\delta] \times [\delta, 1-\delta]$, and assume that $\eps \in \left(0, \delta/2\right)$. Note that $(t+\eps u, \pi + \eps v) \in S_{\delta/2}$ for any $(u,v) \in [0,1]^2$ and $(t, \pi) \in S_\delta$. 
    
    \underline{Step 2:}
    We then get convenient representations of the derivatives of the function $V_T^{(\eps)}(\cdot, \cdot)$, namely of its temporal derivative and the first two spatial derivatives.

    First, using the fact that $\partial_\pi V_T(\cdot, \cdot)$ is continuous by Proposition \ref{prop_C1_value_func} and bounded on $S_{\delta/2}$ by part \ref{value_func_Lipschitz_pi} of Proposition \ref{prop_properties_of_value_func}, we apply the bounded convergence theorem to obtain
    \begin{equation}\label{partial_pi_representation}
        \partial_\pi V_T^{(\eps)}(t, \pi) = \int_{\mathbb{R}^2} \partial_\pi V_T(t+\eps u,\pi + \eps v) \phi(u,v) \, du \, dv, \quad \forall \, (t, \pi) \in S_{\delta}.
    \end{equation}
    
    Secondly, by arguing similarly to the proof of Proposition \ref{prop_C1_value_func} and using Theorem 3.6 of \cite{Lamberton}, we obtain that the weak derivatives $\partial_t V_T(\cdot, \cdot)$ and $\partial_{\pi \pi}V_T(\cdot, \cdot)$ exist and are bounded on $S_{\delta/2}$. Moreover, from the free-boundary problem \eqref{value_function_boundary_problem} (whose validity is proved in Proposition \ref{prop_boundary_problem_value_func}), it is evident that the strong derivatives $\partial_t V_T(\cdot, \cdot)$ and $\partial_{\pi \pi}V_T(\cdot, \cdot)$ exist and are continuous on $S_{\delta/2} \setminus \partial C$. In particular, this implies that on $S_{\delta/2} \setminus \partial C$ the corresponding weak and strong derivatives coincide.
    
    As a result, using the identity \eqref{integral_change_of_variables} and the definition of weak derivative, we obtain, for every $(t, \pi) \in S_\delta$, 
    \begin{equation*} \label{mollified_t_deriv}
            \partial_t V^{(\eps)}_T(t, \pi) = -\frac{1}{\eps^3} \int_{\mathbb{R}^2} V_T(s, y) \phi_1\left(\frac{s-t}{\eps}, \frac{y-\pi}{\eps}\right) \, ds \,  dy
            = \int_{\mathbb{R}^2} \partial_t V_T(t + \eps u, \pi + \eps v) \phi(u,v) \, du \, dv.
    \end{equation*}
    Here, $\phi_1(\cdot, \cdot)$ denotes the partial derivative of $\phi(\cdot, \cdot)$ with respect to its first argument, and the first equality follows by differentiating under the integral, using the bounded convergence theorem.

    Similarly, for every $(t, \pi) \in S_\delta$, we get the following representation of $\partial_{\pi \pi}V^{(\eps)}_T(\cdot, \cdot)$:
    \begin{equation*}
    \begin{split}
        \partial_{\pi \pi}V^{(\eps)}_T(t, \pi) &= \frac{1}{\eps^4} \int_{\mathbb{R}^2} V_T(s, y) \phi_{22}\left(\frac{s-t}{\eps}, \frac{y-\pi}{\eps}\right) \, ds \, dy
        = 
        -\frac{1}{\eps^3} \int_{\mathbb{R}^2} \partial_\pi V_T(s,y) \phi_2\left(\frac{s-t}{\eps}, \frac{y-\pi}{\eps}\right) \, ds \,
        \\
        &= \frac{1}{\eps^2} \int_{\mathbb{R}^2} \partial_{\pi\pi} V_T(s,y) \phi\left(\frac{s-t}{\eps}, \frac{y-\pi}{\eps}\right) \, dy \, ds 
        = \int_{\mathbb{R}^2} \partial_{\pi \pi} V_T(t + \eps u, \pi + \eps v) \phi(u, v) \, du \, dv. 
        \end{split}\label{mollified_pi_pi_deriv}
    \end{equation*}
    Here, $\phi_2(\cdot, \cdot)$ denotes the partial derivative of $\phi(\cdot, \cdot)$ with respect to its second argument, and $\phi_{22}(\cdot, \cdot)$ denotes the second partial derivative with respect to the same argument. The second equality follows from integration by parts and the fact that $\phi \equiv 0$ outside its support.
    
    \underline{Step 3:}
    Next, we apply Itô's formula to the mollified 
    $V_T^{(\eps)}(\cdot, \cdot)$.
    To this effect, we fix 
    $(t, \pi) \in S_\delta$,
    and define 
    $\tau_h \coloneqq \inf\{s \ge 0 : \Pi^\pi(s) \not \in [h, 1-h]\} \wedge (T-t-h)$ for $h \in (0,\delta)$.
    Note that 
    $(t + \tau_h, \Pi^\pi(\tau_h)) \in S_{h}$
    and 
    $\lim\limits_{h \downarrow 0} \tau_h = T-t$
    almost surely. 
    Applying It\^{o}'s formula to 
    $V_T^{(\eps)}(t+\cdot \wedge \tau_h, \Pi^\pi(\cdot \wedge \tau_h))$
    yields
    \begin{equation} \label{ito_for_mollified}
            V_T^{(\eps)}(t+s\wedge \tau_h, \Pi^\pi(s \wedge \tau_h)) 
            =
            V_T^{(\eps)}(t, \pi) + \int_0^{s \wedge \tau_h} \mathcal{L}V_T^{(\eps)}(t+u, \Pi^\pi(u)) \, du + \int_0^{s \wedge \tau_h} \partial_\pi V^{(\eps)}_T(t+u, \Pi^\pi(u)) \, d\Pi^\pi(u)
    \end{equation}
    for every $s \in [0, T-t]$, where $\mathcal{L}(\cdot)$ denotes the partial differential operator 
    \begin{equation*}
        \mathcal{L} f(t, \pi) \coloneqq \partial_t f(t, \pi) + \frac{a^2}{2} \pi^2 (1-\pi)^2 \partial_{\pi \pi}f(t, \pi).
    \end{equation*}
    We now need to let $\eps \downarrow 0$ and show that all terms in \eqref{ito_for_mollified} converge accordingly.

    \underline{Step 4:}
    For fixed $0 < h < \delta$, we fix a decreasing sequence $\{\eps_n\}_{n \in \mathbb{N}}$ such that $0 < \eps_n < \frac{h}{2}$ and $\eps_n \downarrow 0$, and note that $(t+\eps_n u, \pi + \eps_n v) \in S_{h/2}$ whenever $(t,\pi) \in S_h$ and $(u,v) \in [0,1]^2$. We consider the expression \eqref{ito_for_mollified} for $\eps_n$ and send $n \to \infty$.
    
    First, note that the bounded convergence theorem, the continuity of the function $V_T(\cdot, \cdot)$, and the definition \eqref{definition_mollified_value_function}, imply that the left-hand side of \eqref{ito_for_mollified} converges to $V_T(t+s\wedge \tau_h, \Pi^\pi(s \wedge \tau_h))$; while the first term on the right-hand side converges to $V_T(t, \pi)$.

    Second, since $\partial_\pi V(\cdot, \cdot)$ is bounded and continuous on $S_{h/2}$, so are $\partial_\pi V^{(\eps_n)}(\cdot, \cdot)$ on $S_{h}$ due to the representation \eqref{partial_pi_representation}; thus the dominated convergence theorem for stochastic integrals implies the existence of a subsequence of $\{\eps_n\}_{n\in\mathbb{N}}$ such that the stochastic integral in \eqref{ito_for_mollified} converges to
    \begin{equation*}
        \int_0^{s \wedge \tau_h} \partial_\pi V_T(t+u, \Pi^\pi(u)) \, d\Pi^\pi(u).
    \end{equation*}
    We will abuse notation and also refer to this subsequence as $\{\eps_n\}_{n \in \mathbb{N}}$.
    
    Finally, in order to obtain the convergence of the remaining term in \eqref{ito_for_mollified}, we consider the sets $\partial \mathcal{C}$ and $S_{h} \setminus \partial \mathcal{C}$ separately. Note that, for any $n \in \mathbb{N}$, the integral of $\mathcal{L} V^{(\eps_n)}$ over the set $\partial \mathcal{C}$ is equal to zero, as a consequence of the following chain of inequalities:
    \begin{equation*}
        \begin{split}
            & \ex\left[\left|\int_0^{s \wedge \tau_h} \mathcal{L}V^{(\eps_n)}(t+u, \Pi^\pi(u)) \cdot \mathbf{1}_{\partial \mathcal{C}}(t+u, \Pi^\pi(u)) \, du \right| \right] 
            \\& \quad \quad \quad \quad \le
            \ex\left[\int_0^{T-t} C \cdot \mathbf{1}_{\partial \mathcal{C}}(t+u, \Pi^\pi(u)) \, du \right]
            =
            C \int_0^{T-t} \int_0^1 f(u,\pi) \cdot \mathbf{1}_{\partial \mathcal{C}}(t+u, \pi) \, d\pi \, du = 0.
        \end{split}
    \end{equation*}
    Here, $C$ is an upper bound on the weak derivatives $\partial_t V_T(\cdot, \cdot), \partial_{\pi \pi} V_T(\cdot, \cdot)$, which are bounded on $S_{h/2}$, and $f(u, \cdot)$ denotes the density of the random variable $\Pi^\pi(u)$. The last equality follows from the fact that the set $\partial \mathcal{C}$ has Lebesgue measure zero by Lemma \ref{boundary_measure_zero}, whose proof is independent of the current one.
    
    As for the complement of $\partial \mathcal{C}$, we recall that the derivatives $\partial_t V_T(\cdot, \cdot)$, $\partial_{\pi \pi} V_T(\cdot, \cdot)$ are continuous and bounded on $S_{h/2} \setminus \partial \mathcal{C}$ (moreover, the weak and strong derivatives coincide in this region), so for all $(t, \pi) \in S_{h} \setminus \partial \mathcal{C}$ we have 
    \begin{equation*} 
        \lim_{n \rightarrow \infty} \left(\mathcal{L} V_T^{(\eps_n)}(t, \pi)\right) = \mathcal{L} V_T(t, \pi) = \partial_t V_T(t, \pi) + \frac{a^2}{2} \pi^2(1-\pi)^2\partial_{\pi \pi} V_T(t, \pi).
    \end{equation*}
    Thus, using the boundedness of the derivatives $\partial_t V_T(\cdot, \cdot)$ and $\partial_{\pi \pi} V_T(\cdot, \cdot)$ on $S_{h/2} \setminus \partial \mathcal{C}$ along with the dominated convergence theorem, we obtain that the Lebesgue integral in \eqref{ito_for_mollified} converges to 
    \begin{equation*}
        \int_0^{s \wedge \tau_h} \mathcal{L}V_T(t+u, \Pi^\pi(u)) \cdot \mathbf{1}_{S_{h} \setminus \partial \mathcal{C}} (t+u, \Pi^\pi(u)) \, du.
    \end{equation*}
    
    \underline{Step 5:}
    Combining the above results, noting that $\mathcal{L}V(t, \pi) \equiv 0$ when $\pi < b(t)$, and $\mathcal{L}V(t, \pi) = c_1'(t)\pi - c_0'(t)(1-\pi)$ for $(t, \pi) \in \text{int}(\mathcal{S}_T)$, and observing that $\{(t, \pi) : \pi = b(t)\} \subseteq \partial \mathcal{C}$, we obtain
    \begin{equation*}
        \begin{split}
            V_T(t+s\wedge &\tau_h, \, \Pi^\pi(s \wedge \tau_h)) 
            = 
            \int_0^{s \wedge \tau_h} \mathbf{1}_{\{\Pi^\pi(u) \ge b(t+u) \}} \Big(c_1'(t+u)\Pi^\pi(u) -c_0'(t+u)(1-\Pi^\pi(u)) \Big) du
            \\
            &+ 
            V_T(t, \pi) + 
            \int_0^{s \wedge \tau_h} \partial_\pi V_T(t+u, \Pi^\pi(u)) \, \Pi^\pi(u) \, (1-\Pi^\pi(u)) \, dB(u), \quad s \in [0, T-t],\,  a.s.
        \end{split}
    \end{equation*}
    Since all integrands in the above expression are bounded, the dominated convergence theorem allows us to let $h \downarrow 0$ to obtain the result of the proposition, \eqref{ito_equation_for_value_function}, for any $s \leq T-t$ and $(t, \pi) \in S_\delta$. Since $\delta$ was arbitrary, this holds for any $(t, \pi) \in (0, T) \times (0,1)$. Again, the boundedness of all integrands allows us to take limits as $t \downarrow 0$ as well, and this concludes the proof for the case $T < \infty$, $c_i(\cdot) \in C^2([0, T])$, $i = 0,1$ and $c_i(T) > 0, \, i = 0, 1$.

    To see that \eqref{ito_equation_for_value_function} holds for any of the remaining cases, we fix $t < T$, $s < T-t$, and any increasing sequence $\{T_n\}_{n \in \mathbb{N}}$ such that $T_n \in (t+s, T)$ and $T_n \uparrow T$. Since $c_i(\cdot) \in C^2([0,T_n])$ and $c_i(T_n) > 0$ for $i = 0, 1$, we have that \eqref{ito_equation_for_value_function} holds for every $T_n, n \in \mathbb{N}$, from the above proof. We now pass to the limit as $n \to \infty$.

    We first handle the stochastic integral term. The almost sure convergence (formally, over a subsequence $T_{n_j}$, which we will refer to as just $T_n$, abusing notation)
    \begin{equation*}
        \int_0^s \partial_\pi V_{T_{n}} (t+u, \Pi^\pi(u)) \, d\Pi^\pi(u) \to \int_0^s \partial_\pi V_{T}(t+u, \Pi^\pi(u)) \, d\Pi^\pi(u)
    \end{equation*}
    follows from the dominated convergence theorem for stochastic integrals, which is applicable due to the bounds
    $0 \leq \partial_\pi V_{T_n}(t+\cdot, \cdot) \leq c_0(0) + c_1(0)$
    and the pointwise convergence of $\partial_\pi V_{T_n}(t+\cdot, \cdot)$ to $\partial_\pi V_T(t+\cdot, \cdot)$. The latter follows from the fact that, for any $(u,p) \in [0,s] \times (0,1)$, any subsequence of $\{\partial_\pi V_{T_n}(t+u, p)\}_{n \in \mathbb{N}}$ has a further subsequence that converges to $\partial_\pi V_T(t+u, p)$ as a consequence of the proof of Proposition \ref{prop_C1_value_func}.

    It suffices now to show that the Lebesgue integral in \eqref{ito_equation_for_value_function} converges almost surely to the appropriate limit, since the same statement for the remaining terms in \eqref{ito_equation_for_value_function} follows immediately from the continuity of the function $V_T(\cdot, \cdot)$. Moreover, it is clear from the bounded convergence theorem that it suffices to show the almost sure equality
    \begin{equation}\label{boundary_indicator_limit}
        \lim_{n \rightarrow \infty} \mathbf{1}_{\{\Pi^\pi(u) \ge b_{T_n}(t+u)\}} = \mathbf{1}_{\{\Pi^\pi(u) \ge b_{T}(t+u)\}}, \quad \forall \, u \in [0, s].
    \end{equation}
    To obtain \eqref{boundary_indicator_limit}, we note that $V_{R}(u, \pi) \leq V_S(u, \pi)$ for every $R \leq S$ and all $(u, \pi) \in [0,R]  \times [0,1]$, so the definition \eqref{boundary_definition_2} of the stopping boundary implies $b_{R}(u) \leq b_S(u)$ for all $u \in [0,R]$.
    Thus, $\overline b(t+u) \coloneqq \lim_{n \rightarrow \infty} b_{T_n}(t+u)$, $u \in [0,s]$ is well defined and satisfies
    $\overline b(t+u) \leq b_T(t+u)$ for all $u \in [0, s]$. We claim equality holds. Proceeding by contradiction, suppose that $\overline b(t+u) < b_T(t+u)$ for some $u \in [0,s]$. Then, since $\overline b(t+u) > 0$ because $c_i(t+u) > 0$, $i = 0, 1$, we have
    \begin{equation*}
            V_T\left(t+u, \overline b(t+u)\right) = \lim_{n \rightarrow \infty} V_{T_n}\left(t+u, \overline b(t+u)\right)
            = G\left(t+u, \overline b(t+u)\right)
             < V_T\left(t+u, \overline b(t+u)\right),
    \end{equation*}
    an obvious contradiction. Here, the first equality follows from the fact that $V_{T_n}(\cdot, \cdot)$ converges pointwise to $V_T(\cdot, \cdot)$, the second from the fact that $\overline b(t+u) \ge b_{T_n}(t+u)$ for any $n \in \mathbb{N}$, and the inequality follows from the fact the point $(t+u, \overline b(t+u)) \in [0,T) \times (0,1)$ is in the continuation region if $\overline b(t+u) < b_T(t+u)$. Thus, we obtain \eqref{boundary_indicator_limit} due to the left-continuity of the function $x \mapsto \mathbf{1}_{\{\Pi^\pi(u) \ge x\}}$. As a result, we obtain \eqref{ito_equation_for_value_function} for any $s < T-t$. The dominated convergence theorem implies the result for $s=T-t$, concluding the proof.
    \hfill \qed {\parfillskip0pt\par}

\subsection{Proof of Proposition \ref{prop_uniqueness_free_boundary}}\label{subsec_app_uniqueness_free_boundary}
    Fix any $T \in [0, \infty]$ and let $\big(f(\cdot, \cdot), d(\cdot)\big)$ be a pair that satisfies the conditions of the proposition. We first show that $f(\cdot, \cdot) \equiv V_T(\cdot, \cdot)$, and then obtain $b_T(\cdot) = d(\cdot)$ Lebesgue almost everywhere.
    
    To obtain $f(\cdot, \cdot) \equiv V_T(\cdot, \cdot)$, we first establish that $f(\cdot, \cdot)$ satisfies the Itô's formula similar to \eqref{ito_equation_for_value_function}. We proceed similarly to the proof of Proposition \ref{ito_for_value_function}.
    In particular, we start with mollifying $f(\cdot,\cdot)$ 
    to obtain functions $\{f^{(\eps)}(\cdot, \cdot)\}_{\eps > 0}$ such that $f^{(\eps)}(t,\pi) \xrightarrow{\eps \downarrow 0} f(t, \pi)$
    for every $(t, \pi) \in [0,T] \times [0,1]$, and which satisfy the It\^{o} formula
    \begin{equation} \label{ito_for_mollified_f}
        \begin{split}
            f^{(\eps)}(t + \tau, \Pi^\pi(\tau)) = f^{(\eps)}(t,\pi) &+ \int_0^\tau \mathcal{L} f^{(\eps)}(t+s, \Pi^\pi(s)) \, ds 
            \\&
            + \int_0^\tau \partial_\pi f^{(\eps)}(t+s, \Pi^\pi(s)) \, a\, \Pi^\pi(s)(1-\Pi^\pi(s)) \, dB(s),\ \  a.s.
        \end{split}
    \end{equation}
    for any stopping time $\tau \in \mathcal{T}_{T-t}$ and any $(t, \pi) \in [0,T] \times [0,1]$. Here, $\mathcal{L}$ denotes the PDE operator $\partial_t + \frac{a^2}{2}\pi^2(1-\pi)^2 \partial_{\pi\pi}$. Note that the last term on the right-hand side of \eqref{ito_for_mollified_f} is a martingale, which is bounded in $L^2$ on any finite time interval due to the boundedness of the integrand. Thus, for any bounded stopping time $\tau$, we can take expectations in \eqref{ito_for_mollified_f} to obtain
    \begin{equation} \label{exp_ito_for_mollified_f}
        \ex\big[f^{(\eps)}(t+\tau, \Pi^\pi(\tau)) \big] = f^{(\eps)}(t, \pi) + \ex\left[\int_0^\tau \mathcal{L}f^{(\eps)}(t+s, \Pi^\pi(s)) \, ds \right].
    \end{equation}
    
    Fix $\delta > 0$, $h \in (0, \delta)$, $(t, \pi) \in {S}_\delta \coloneqq [\delta, T-\delta]  \times [\delta, 1-\delta]$, and let $\tau_h \coloneqq \inf\{s \ge 0 : (t+s, \Pi^\pi(s)) \notin S_h\}$. Note that the condition \eqref{uninq_1} implies that the functions $f^{(\eps)}(\cdot, \cdot)$ and $\mathcal{L}f^{(\eps)}(\cdot,\cdot)$ are bounded uniformly for $\eps \in (0,h/2)$, $(t, \pi) \in S_h$. Thus, applying the dominated convergence theorem and passing to the limit as $\eps \downarrow 0$ in \eqref{exp_ito_for_mollified_f}, we obtain
    \begin{equation} \label{exp_ito_for_f}
        \begin{split}
            &\ex\big[f(t+\tau\wedge\tau_h, \Pi^\pi(\tau \wedge \tau_h)) \big] 
            \\&=
            f(t,\pi) + \ex\left[\int_0^{\tau \wedge \tau_h} \bigg(\Pi^\pi(s) \Big(c_1'(t+s) + c_0'(t+s)\Big) - c_0'(t+s) \bigg) \, \mathbf{1}\Big((t+s, \Pi^\pi(s)) \in \mathring{(\mathcal{D}^c)}\Big) \, ds \right] \\
            &= f(t,\pi) + \ex\left[\int_0^{\tau \wedge \tau_h} \bigg(\Pi^\pi(s) \Big(c_1'(t+s) + c_0'(t+s)\Big) - c_0'(t+s) \bigg) \, \mathbf{1}\Big(\Pi^\pi(s) > d(t+s)\Big) \, ds \right],
        \end{split}
    \end{equation}
    where $\mathring{(\mathcal{D}^c)}$ denotes the interior of the complement of $\mathcal{D}$. Here, we used the facts that $\mathcal{D}$ is an open set due to \eqref{uninq_5} and $\partial \mathcal{D}$ has Lebesgue measure 0 by \eqref{uninq_6}, so we only need to account for the limit when $(t+s, \Pi^\pi(s)) \in \mathcal{D} \cup \mathring{(\mathcal{D}^c)}$, regions where $f(\cdot,\cdot)$ is of class $C^{1,2}$. \eqref{uninq_2} and \eqref{uninq_3_1} then imply the first equality in \eqref{exp_ito_for_f}. The second equality is an obvious consequence of the fact that $\mathring{(\mathcal{D}^c)} \subseteq \{(t,\pi): \pi > d(t)\}$ and $\partial \mathcal{D}$ has Lebesgue measure $0$. Passing to the limit $h \downarrow 0$ in \eqref{exp_ito_for_f} shows that this formula holds for all $(t, \pi) \in S_\delta$, and $\tau \wedge \tau_h$ replaced by $\tau$, since $\lim_{h \downarrow 0} \tau_h = T-t$ almost surely. Since $\delta > 0$ was arbitrary, the formula holds for any $(t, \pi) \in (0, T) \times (0,1)$.
    
    Since $f(\cdot,\cdot) \ge 0$ as a consequence of \eqref{uninq_4}, the condition \eqref{uninq_3_1} implies
    \begin{equation}\label{boundary_lower_bound}
        d(t) \ge \frac{c_0(t)}{c_0(t) + c_1(t)}, \quad 0 \le t < T.
    \end{equation}
    This, in turn, implies that for $\pi > d(t)$ we have
    \begin{equation} \label{generatorstrictlynegativeongain}
        \pi\Big(c_1'(t) + c_0'(t)\Big) - c_0'(t) < \frac{c_0(t) \; c_1(t)}{c_0(t) + c_1(t)} 
        \left(\frac{c_1'(t)}{c_1(t)} - \frac{c_0'(t)}{c_0(t)}\right) \le 0,
    \end{equation}
    due to the assumptions \ref{reg_assm_1} and \ref{reg_assm_5}. Appealing to \eqref{exp_ito_for_f} and \eqref{uninq_4}, we then have
    $f(t,\pi) \ge \ex\big[G(t+\tau, \Pi^\pi(\tau))\big]$ 
    for all bounded stopping times $\tau$ and every $(t, \pi) \in (0,T) \times (0,1)$. This implies $f(t, \pi) \ge V_T(t, \pi)$ for every $(t, \pi) \in (0,T) \times (0,1)$ via a trivial localization argument.
    
    To see the reverse direction, for any fixed $(t, \pi) \in (0,T) \times (0,1)$, define the stopping time
    \begin{equation*}
        \sigma \coloneqq\inf\{s \ge 0: \Pi^\pi(s) \ge d(t+s) \} \wedge (T-t)
    \end{equation*}
    and note that \eqref{uninq_3_1} and \eqref{uninq_3_2} imply $f(t+\sigma, \Pi^\pi(\sigma)) = G(t+\sigma, \Pi^\pi(\sigma))$.
    Therefore, we obtain
    $
    f(t, \pi) = \ex\left[G(t+\sigma, \Pi^\pi(\sigma)\right] \leq V_T(t,\pi)
    $
    by \eqref{exp_ito_for_f} and the definitions  of $\sigma$ and $V_T(\cdot,\cdot)$. Since both functions are continuous, equality of $f(\cdot,\cdot)$ and $V_T(\cdot,\cdot)$ holds on $[0,T] \times [0,1]$ as well.

    We now show $d(t) = b_T(t)$ for Lebesgue almost every $t \in [0, T)$. First, we note that, for every $t \in [0,T)$, we have
    \begin{equation*}
            b_T(t) = \inf\{\pi : V_T(t,\pi) = g(t,\pi)\} 
            = \inf\{\pi : f(t,\pi) = g(t,\pi)\} \leq d(t).
    \end{equation*}
    We show the reverse inequality by contradiction, supposing that the set $\{s \in [0,T): b_T(s) < d(s)\}$ has positive Lebesgue measure. 
    Applying the It\^{o} formula of Proposition \ref{ito_for_value_function} to $V_T(\cdot,\cdot)$, taking expectations in \eqref{ito_equation_for_value_function} as in the arguments above, and canceling terms from \eqref{exp_ito_for_f} by using $f(\cdot, \cdot) \equiv V_T(\cdot, \cdot)$, we obtain 
    \begin{align}\label{expectation_zero_between_boundaries}
        \ex \left[\int_0^T \bigg(\Pi^\pi(s) \Big(c_1'(s) + c_0'(s)\Big) - c_0'(s) \bigg) \, \mathbf{1}\Big(b_T(s) \leq \Pi^\pi(s) \leq d(s)\Big) \, ds \right] = 0.
    \end{align}
    In particular, since $\Pi^\pi(s)$ has a strictly positive density for all $s \in (0,T)$, we also get
    \begin{align}\label{expectation_zero_between_boundaries_2}
        \ex \left[\int_0^T \bigg(\Pi^\pi(s) \Big(c_1'(s) + c_0'(s)\Big) - c_0'(s) \bigg) \, \mathbf{1}\Big(b_T(s) < \Pi^\pi(s) \leq d(s)\Big) \, ds \right] = 0.
    \end{align}
    However, \eqref{generatorstrictlynegativeongain} holds for $\pi > b_T(t)$ by the same arguments, since the boundary $b_T(\cdot)$ again satisfies \eqref{boundary_lower_bound} as follows from strict positivity of the value function. Therefore,
    \eqref{expectation_zero_between_boundaries_2} implies
    \begin{equation*}
        \mathbf{1}\Big(b_T(s) < \Pi^\pi(s) \leq d(s)\Big) = 0 \qquad \text{Leb} \times \pr - a.s.,
    \end{equation*}
    where Leb is the Lebesgue measure on $[0,T)$. This, however, is a contradiction, as
    \begin{equation*}
        \int_0^T \pr\Big(b_T(s) < \Pi^\pi(s) \le d(s)\Big) ds > 0, \quad \forall \; \pi \in (0,1), 
    \end{equation*}
    since the random variable $\Pi^\pi(s)$ has a strictly positive density for all $s \in (0,T)$.
    \hfill \qed {\parfillskip0pt\par}

\subsection{Proof of Lemma \ref{boundary_measure_zero}}\label{subsec_app_measure_zero}

    We fix $T \in [0, \infty)$, let $u_T(t, x)$ be the value function defined in \eqref{value_function_american_options}, and recall \eqref{american_call_value_fn}. 
    Using Propositions \ref{prop_boundary_problem_value_func}, \ref{prop_boundary_monotonicity} and the relationship \eqref{american_call_value_fn}, one verifies that on $[0, T) \times \mathbb{R}$ the function $u_T(\cdot, \cdot)$ satisfies
    \begin{align*}
        \begin{cases} \label{freeboundaryforamericancall}
            \frac{a^2}{2} \partial_{xx} u_{T} + \left(\beta_1(t) - \beta_0(t) - \frac{a^2}{2} \right) \partial_x u_T + \beta_0(t) u_T + \partial_t u_T = 0, & \quad x < \widecheck{b}_T(t), \\
        u_T(t,x) = e^x - 1, & \quad x \ge \widecheck{b}_T(t),
        \end{cases}
    \end{align*}
    with the boundary $\widecheck{b}_T: [0, T) \to \mathbb{R}$ defined by
    \begin{equation*}
        \widecheck{b}_T(t) \coloneqq \log \left( \frac{b_T(t) \, c_1(t)}{(1 - b_T(t))\, c_0(t)} \right) = y(t, b_T(t)),
    \end{equation*}
    with $y(\cdot, \cdot)$ of \eqref{change_of_initial_position}.
    Therefore, since the boundary $\partial \widecheck{\mathcal{C}}_T$ of the region
    $$
    \widecheck{\mathcal{C}}_T \coloneqq \left\{(t,x) \in [0, T) \times \mathbb{R} : x < \widecheck{b}_T(t)\right\}
    $$
    satisfies $(t, \pi) \in \partial \mathcal{C}_T \iff (t, y(t, \pi)) \in \partial \widecheck{\mathcal{C}}_T$,
    and the mapping $(t, \pi) \mapsto (t, y(t, \pi))$ is a continuously differentiable bijection with strictly positive Jacobian on its domain, it suffices to show that the Lebesgue measure of the boundary $\partial \widecheck{\mathcal{C}}_T$ is zero. 
    
    Note that, for any $t \in [0, T]$, the point $(t, b_T(t))$ belongs to the stopping region. 
    Thus, by definition the gain function $G(\cdot, \cdot)$ in \eqref{gain_function} and non-negativity of the value function $V_T(\cdot, \cdot)$, we have 
    $b_T(t) c_1(t) - (1-b_T(t)) c_0(t) \ge 0$,
    which implies 
    $\widecheck{b}_T(t) \ge 0$. 
    In particular, we immediately get
    $
    \partial \widecheck{\mathcal{C}}_T \subseteq [0,T] \times [0, \infty).
    $
    Hence, it suffices to show that for every $t_0, x_0, R \in \mathbb{Q}$ such that 
    \begin{equation}\label{parabolic_cylinder}
        Q_R(t_0, x_0) \coloneqq (t_0 - R^2, t_0 + R^2) \times (x_0 - R, x_0 + R) \subseteq (0, T) \times (0, \infty),
    \end{equation}
    the set 
    $\partial \widecheck{\mathcal{C}}_T \cap Q_R(t_0, x_0)$
    has Lebesgue measure zero.
    
    This, however, is an immediate consequence of Theorem 2.6 of \cite{Blanchet}. Indeed, for fixed $t_0, x_0, R \in \mathbb{Q}$ such that \eqref{parabolic_cylinder} holds, we define
    $
    h(t,x) \coloneqq u_T(T-t,x) - (e^x - 1),
    $
    and, by arguing similarly to the proof of Proposition \ref{prop_C1_value_func} and employing Theorem 3.6 of \cite{Lamberton} again, obtain that the function $h(\cdot, \cdot)$ satisfies the variational inequalities
    \begin{equation*}
        \begin{cases}
            \mathcal{L}h(t,x) = -e^x \beta_1(t) + \beta_0(t), & \qquad h(t, x) > 0, \\
            \mathcal{L}h(t,x) \leq -e^x \beta_1(t) + \beta_0(t), & \qquad \text{Lebesgue a.e. on } Q_R(t_0, x_0) \\
            h(t,x) \ge 0,
        \end{cases},
    \end{equation*}
    where $\mathcal{L}(\cdot)$ is the differential operator defined by 
    \begin{equation*} 
        \mathcal{L}h(t,x) \coloneqq
        \frac{a^2}{2} \partial_{xx} h(t,x) + \left(\beta_1(t) - \beta_0(t) - \frac{a^2}{2} \right)\partial_x h(t,x) + \beta_0(t) h(t,x) - \partial_t h(t,x).
    \end{equation*}
    It remains to note that, by the assumptions \ref{reg_assm_2} and \ref{reg_assm_5}, the coefficients of the operator $\mathcal{L}(\cdot)$ and the function 
    $(x, t) \mapsto -e^x\beta_1(t) + \beta_0(t)$
    are differentiable, and that on $Q_R(t_0, x_0)$ we have 
    \begin{align*}
        -e^x \beta_1(t) &+ \beta_0(t) 
        =
        \beta_0(t) - \beta_1(t) + (1 - e^{x}) \beta_1(t)
        \\&\ge
        \inf_{s \in [t_0 - R^2, t_0 + R^2]} \Big(\beta_0(s) - \beta_1(s)\Big) 
        + \left(1 - e^{x}\right) \cdot \sup_{s \in [t_0 - R^2, t_0 + R^2]} \beta_1(s)
        \eqqcolon
        \delta_0(t_0, x_0, R) > 0.
    \end{align*}
    In particular, we use here that $x > 0$ and that $\beta_1(\cdot)$ is continuous and negative, i.e., strictly bounded away from zero on any compact set.
    Therefore, the function $h(\cdot, \cdot)$ solves the problem (2.1) of \cite{Blanchet} and satisfies the conditions of (1.2) of that paper; from which an appeal to Theorem 2.6 of \cite{Blanchet} yields the desired result.
    \hfill \qed {\parfillskip0pt\par}

\end{appendices}
 
\printbibliography

\end{document}